\documentclass[11pt]{amsart}

\usepackage{epigamath}

\usepackage[english]{babel}

\numberwithin{equation}{section}

\usepackage[shortlabels]{enumitem}
\setlist[enumerate,1]{label={\rm(\arabic*)}, ref={\rm\arabic*}} 

\usepackage[utf8]{inputenc}
\usepackage{amsmath,amsfonts,amsthm,amssymb}

\usepackage{mathrsfs}
\usepackage{tikz}
\usepackage{tikz-cd} 
\usepackage{graphicx}
\usepackage[]{algorithm2e}
\usepackage{subfig}
\usepackage{wrapfig}
\usepackage{xcolor}
\usepackage{comment}
\usepackage{relsize}
\usepackage{xfrac}
\usepackage{faktor}
\usepackage{mathabx}
\usepackage{ifpdf}
\ifpdf
\fi

\newtheorem{theorem}{Theorem}[section]
\newtheorem{corollary}[theorem]{Corollary}
\newtheorem{lemma}[theorem]{Lemma}
\newtheorem{proposition}[theorem]{Proposition}

\theoremstyle{definition}
\newtheorem{definition}[theorem]{Definition}

\theoremstyle{remark}
\newtheorem{remark}[theorem]{Remark}
\newtheorem{example}[theorem]{Example}

\newcommand{\C}{\mathbb C}
\newcommand{\Z}{\mathbb Z}

\newcommand{\PP}{\mathbb P}

\newcommand{\F}{\mathcal{F}}
\newcommand{\OO}{\mathcal{O}}
\newcommand{\chapter}{\section}

\newcommand{\Ee}{\mathcal{E}}

\newcommand{\ZZ}{\mathcal{Z}}
\newcommand{\ttt}{\mathfrak{t}}

\newcommand{\V}{\mathcal V}
\newcommand{\w}{w}
\newcommand{\gl}{\mathfrak{gl}}
\newcommand{\ssl}{\mathfrak{sl}}

\newcommand{\he}{\mathfrak{h}}
\newcommand{\geg}{\mathfrak{g}}
\newcommand{\kek}{\mathfrak{k}}
\newcommand{\p}{\mathfrak{p}}
\newcommand{\nen}{\mathfrak{n}}
\newcommand{\lel}{\mathfrak{l}}

\newcommand{\ov}{\overline}
\newcommand{\W}{\mathrm W}

\newcommand{\eps}{\varepsilon}
\newcommand{\m}{\mathfrak{m}}

\newcommand{\bb}{\mathfrak b}
\newcommand{\uu}{\mathfrak u}
\newcommand{\I}{\mathcal I}
\newcommand{\Ss}{\mathcal S}
\newcommand{\into}{\hookrightarrow}
\newcommand{\Aa}{\mathcal A}
\newcommand{\Hs}{\mathrm H}
\newcommand{\Gs}{\mathrm G}
\newcommand{\Ts}{\mathrm T}
\newcommand{\Ks}{\mathrm K}
\newcommand{\Bs}{\mathrm B}
\newcommand{\Ps}{\mathrm P}
\newcommand{\Ws}{\mathrm W}
\newcommand{\Ls}{\mathrm L}
\newcommand{\Ns}{\mathrm N}
\newcommand{\Qs}{\mathrm Q}
\newcommand{\U}{\mathrm U}
\newcommand{\Cs}{\mathbb C^\times}
\newcommand{\odd}{\mathrm{odd}}

\renewcommand{\D}{\mathrm{D}}
\newcommand{\Vect}{\operatorname{Vect}}
\newcommand{\Tr}{\operatorname{Tr}}
\newcommand{\diag}{\operatorname{diag}}
\newcommand{\im}{\operatorname{im}}

\newcommand{\Hom}{\operatorname{Hom}}
\newcommand{\Spec}{\operatorname{Spec}}

\newcommand{\pt}{\operatorname{pt}}
\newcommand{\ch}{\operatorname{ch}}
\newcommand{\td}{\operatorname{td}}

\newcommand{\id}{\operatorname{id}}
\newcommand{\ad}{\operatorname{ad}}
\newcommand{\Ad}{\operatorname{Ad}}
\newcommand{\Lie}{\operatorname{Lie}}

\newcommand{\GL}{\operatorname{GL}}
\newcommand{\SL}{\operatorname{SL}}
\newcommand{\PSL}{\operatorname{PSL}}

\newcommand{\Span}{\operatorname{span}}

\newcommand{\coker}{\operatorname{coker}}

\newcommand{\Stab}{\operatorname{Stab}}
\newcommand{\reg}{\mathrm{reg}}
\newcommand{\tot}{\mathrm{tot}}
\newcommand{\Gr}{\operatorname{Gr}}
\newcommand{\Sym}{\operatorname{Sym}}

\newcommand{\lra}{\longrightarrow}

\newcommand{\supth}[1]{\ensuremath{#1^{\mathrm{th}}}}

\EpigaVolumeYear{9}{2025} \EpigaArticleNr{1} \ReceivedOn{November 24, 2023}
\InFinalFormOn{April 12, 2024}
\AcceptedOn{May 8, 2024}

\title{Spectrum of equivariant cohomology as a fixed point scheme}
\titlemark{Spectrum of equivariant cohomology as a fixed point scheme}

\author{Tam\'as Hausel}
\address{Institute of Science and Technology Austria (ISTA), Am Campus 1,  
3400 Klosterneuburg, Austria }
\email{tamas.hausel@ist.ac.at}
\author{Kamil Rychlewicz}
\address{Institute of Science and Technology Austria (ISTA), Am Campus 1, 
3400 Klosterneuburg, Austria }
\email{kamil.rychlewicz@ist.ac.at}

\authormark{T.~Hausel and K.~Rychlewicz}

\AbstractInEnglish{An action of a complex reductive group $\Gs$ on a smooth projective variety $X$ is {\em regular} when all  regular unipotent elements in $\Gs$ act with finitely many fixed points.
Then the complex $\Gs$-equivariant cohomology ring of $X$ is isomorphic to the coordinate ring of a certain regular fixed point scheme. 
Examples include partial flag varieties, smooth Schubert varieties and Bott--Samelson varieties. We also show that a more general version of the fixed point scheme allows a generalisation to GKM spaces, such as toric varieties.}

\MSCclass{14L30, 55N91}

\KeyWords{Group actions on varieties, equivariant cohomology, fixed point scheme}

\acknowledgement{The first author was supported by an FWF grant ``Geometry of the top of the nilpotent cone'' number P 35847. The second author was supported by an Austrian Academy of Sciences DOC Fellowship ``Topology of open smooth varieties with a torus action''.}

\begin{document}


\maketitle

\begin{prelims}

\DisplayAbstractInEnglish

\bigskip

\DisplayKeyWords

\medskip

\DisplayMSCclass

\end{prelims}


\newpage

\setcounter{tocdepth}{1}

\tableofcontents


\section{Introduction}

In recent work \cite{hausel-icm,hausel-icmtalk} a certain infinitesimal fixed point scheme for the action of $\GL_n$ on $\Gr(k,n)$ -- the Grassmannian of $k$-planes in $\C^n$ -- is used to model the Hitchin map on a particular minuscule upward flow in the $\GL_n$-Higgs moduli space. In turn, it was noticed that this fixed point scheme is isomorphic to the spectrum of equivariant cohomology of $\Gr(k,n)$, and thus the Hitchin system on these minuscule upward flows can be modelled as the spectrum of equivariant cohomology of Grassmannians. In this paper we show that the appearance of the spectrum of equivariant cohomology as a fixed point scheme is not a coincidence and holds in more general situations. 

We start  more generally with partial flag varieties. Let $\Gs$ be a connected complex reductive group and $\Ps\subset \Gs$ be a parabolic subgroup. The partial flag variety is the projective homogeneous space $\Gs/\Ps$ of parabolic subgroups of $\Gs$ conjugate to $\Ps$. Equivalently, we can think of points in $\Gs/\Ps$ as parabolic Lie subalgebras conjugate to the parabolic Lie subalgebra $\p:=\Lie(\Ps)\subset \geg$. Using this point of view, we can define the Grothendieck--Springer partial resolution as
\begin{align}\label{gs}\mu_\Ps\colon\tilde{\geg}_\Ps:=\{(x,\p^\prime)\in \geg\times\Gs/\Ps:x\in\p^\prime\}\lra \geg,\end{align}
given by projection to the first coordinate. It is a proper dominant morphism. Over regular elements in $\geg$, the morphism $\mu_\Ps$ is finite; \textit{cf.} \cite{Akyildiz} and Lemma~\ref{isoreg}. Recall that $x\in\geg$ is regular when its centraliser $\geg^x\subset \Gs$ under the adjoint action has dimension equal to the rank of $\Gs$. The regular elements of $\geg$ form an open dense subset in $\geg$. An equivalent definition of being regular is that the corresponding fiber of $\pi$ is finite for $\Ps=\Bs$ a Borel subgroup.  This implies that the Grothendieck--Springer partial resolution is generically finite-to-one, \textit{i.e.} an alteration. 

One often studies the Grothendieck--Springer map as part of the commutative diagram

$$ \begin{tikzcd}
  \tilde{\geg}_\Ps\arrow{r}{\nu_\Ps} \arrow{d}{\mu_\Ps} &
    \ttt/\!\!/\W_\Ls \arrow{d}{\pi}  \\
   \geg\arrow{r}[swap]{\rho} &
   \ttt/\!\!/\W\rlap{.}
\end{tikzcd}$$
Here $\rho$ is the natural map
$$\rho\colon\geg\lra\geg/ \!\!/\Gs\cong \ttt/ \!\!/\W,$$
where $\ttt = \Lie(\Ts)$ is the Lie algebra of the maximal torus and $\Ws = N_\Gs(\Ts)/\Ts$ is the Weyl group of $\Gs$.
We can define a map
$$\nu_\Ps\colon\tilde{\geg}_\Ps\lra \p/ \!\!/\Ps\cong {\mathfrak l}/ \!\!/\Ls\cong \ttt/ \!\!/\W_\Ls,$$
where $\Ls:=\Ps/\Ps_u$, the quotient with the unipotent radical of $\Ps$, is the Levi quotient of $\Ps$, $\mathfrak{l}$ is its Lie algebra and $\W_\Ls$ is its Weyl group. If $P'$ is conjugate to $P$, with Levi quotient $L'$, then for any $x\in\p'$ the map is given by sending $(x,\p^\prime)\in \tilde{\geg}_\Ps$ to the image of $x\in \p^\prime$ in $\p^\prime/ \!\!/\Ps^\prime\cong \mathfrak{l}^\prime/\!\!/\Ls^\prime$ and then canonically identifying $\mathfrak{l}^\prime/\!\!/\Ls^\prime\cong {\mathfrak l}/ \!\!/\Ls$. If $\Ps\cong\Bs$ is Borel, then $\Ls\cong \Ts$ is the maximal torus and this later is used to define the universal Cartan subalgebra; see \textit{e.g.} \cite[Lemma 6.1.1]{ChGi}. 

 Fix a principal $\ssl_2$-triple $\langle e,f,h\rangle\cong \ssl_2\subset\geg$, where $e\in \geg$ is regular nilpotent. Let
 \begin{align}\label{kostant}\Ss:=e+C_{\geg}(f)\subset \geg \end{align}
 be the Kostant section, where $C_\geg(f)$ is the centraliser of $f$ in $\geg$.
 We have a corresponding principal $\SL_2\to \Gs$ subgroup giving
 \begin{align}\label{tau}\tau\colon\Cs \subset \SL_2 \lra \Gs.\end{align}
 We define a $\Cs$-action on $\geg$ by
 \begin{align} \label{actions} \lambda\cdot x=\lambda^{-2} \Ad_{\tau(\lambda)}(x).\end{align}
 As $\ad_{\tau(\lambda)}(e)=\lambda^2 e$, we see that this $\Cs$-action leaves the Kostant section invariant.\footnote{This action is considered \textit{e.g.} in \cite{GanGin}, where the associated grading is referred to as the \emph{Kazhdan grading}.}

Now write $\tilde{\Ss}_\Ps:=\mu_\Ps^{-1}(\Ss)$. Then we have the commutative diagram\footnote{In the case $\Ps=\Bs$ this diagram was communicated to us by Zhiwei Yun.} 
$$
\begin{tikzcd}
\tilde{\Ss}_\Ps\arrow[hookrightarrow]{r}{\tilde{\iota}}\arrow{d}{\mu_\Ps} \arrow[ bend left]{rr}[black]{\cong} 
  & \tilde{\mathfrak{g}}_\Ps\arrow{r}{\nu_\Ps}\arrow{d}{\mu_\Ps} 
  &  \ttt/\!\!/\W_\Ls \arrow{d}{\pi} \\
 \Ss\arrow[hookrightarrow]{r}[swap]{\iota} \arrow[bend right]{rr}[swap,black]{\cong} 
  & \geg\arrow{r}[swap]{\rho} 
  & \ttt/\!\!/\W
\end{tikzcd}
$$
such that $\rho\circ \iota$ is an isomorphism, as
$\Ss$ is the Kostant section. On the other hand, $\nu_\Ps \circ \tilde{\iota}$ is finite as $\mu_\Ps$ and $\pi$ are, when restricted to regular elements. Finally, the degree of the finite maps $\mu_\Ps$ and $\pi$ both equal the Euler characteristic $\chi(\Gs/\Ps)$. Thus it follows that $\nu_\Ps \circ \tilde{\iota}$ is a finite map  to a normal variety $\ttt/\!\!/\W$ of degree $1$, thus an isomorphism. 

 We also note that the equivariant cohomology algebra
 $$H^*_\Gs(\Gs/\Ps;\C)\cong H^*_{\Gs\times \Ps}(\Gs;\C)\cong H^*_\Ps\cong H^*_\Ls\cong \C[\ttt]^{\W_\Ls}$$
 is naturally an $H^*_\Gs:= H^*(B\Gs;\C) \cong \C[\ttt]^\W$-algebra. From this algebra structure we have a canonical algebra homomorphism $\varphi\colon H^*_\Gs\to H^*_\Gs(\Gs/\Ps;\C)$. We denote the induced map between the affine spectra by
 $$f\colon \Spec\left(H^*_\Gs(\Gs/\Ps;\C)\right)\lra \Spec\left(H^*_\Gs\right),$$
 which is $\Cs$-equivariant with respect to the actions induced by the gradings on both sides. As the odd cohomology $H^{\odd}(\Gs/\Ps;\C)$ is trivial, the space $\Gs/\Ps$ is equivariantly formal; see \cite{GKM}. In other words, the $H^*_\Gs$-module $H_\Gs^*(\Gs/\Ps;\C)$ is free. 
Then we have the following commutative diagram: 
\begin{equation}
\label{big} \begin{tikzcd}
\tilde{\Ss}_\Ps\arrow[hookrightarrow]{r}{\tilde{\iota}}\arrow{d}{\mu_\Ps} \arrow[ bend left]{rr}[black]{\cong} 
  & \tilde{\mathfrak{g}}_\Ps\arrow{r}{\nu_\Ps}\arrow{d}{\mu_\Ps} 
  &  \ttt/\!\!/\W_\Ls \arrow{r}{\cong}\arrow{d}{\pi} 
  & \Spec\left(H^*_\Gs(\Gs/\Ps;\C)\right)\arrow{d}{f}  \\
 \Ss\arrow[hookrightarrow]{r}[swap]{\iota} \arrow[bend right]{rr}[black,swap]{\cong} 
  & \geg\arrow{r}[swap]{\rho} 
  & \ttt/\!\!/\W\arrow{r}{\cong}
 & \Spec\left(H^*_\Gs\right)\rlap{.}
\end{tikzcd}
\end{equation}

Thus we see that the partial Grothendieck--Springer resolution $\mu_\Ps$ over the Kostant section ${\Ss}$ is precisely the spectrum of the $\Gs$-equivariant cohomology algebra of the partial flag variety $\Gs/\Ps$. In this paper our motivation is to show that the appearance of the spectrum of equivariant cohomology in \eqref{big} is not a coincidence. 
We will show that the same holds for $\Hs$-regular actions of a principally paired group $\Hs$   on a smooth projective variety $X$.

\begin{definition}\leavevmode
\begin{enumerate}
\item A complex linear algebraic group $\Hs$ is {\em principally paired} if it contains a pair $\{e,h\}\subset \he$ in its Lie algebra such that $[h,e]=2e$ and $e$ is a regular nilpotent, as well as an algebraic group homomorphism $B(\SL_2)\to \Hs$ from the Borel subgroup of $\SL_2$ whose differential maps the regular unipotent to $e$ and the appropriate diagonal element to $h$.
\item An action of a principally paired group $\Hs$ on a smooth projective variety $X$ is {\em regular} when a regular unipotent element $u\in\Hs$ has finitely many fixed points. 
\end{enumerate}
\end{definition}

In fact, a unipotent element always has a connected fixed point set, see \cite{Horrocks}, so for a regular action we have $X^u=\{ o\}$ for some $o\in X$. Examples of principally paired groups include parabolic subgroups of reductive groups (see Lemma~\ref{parabolic}), such as Borel subgroups and reductive groups themselves.
While examples of $\Hs$-regular varieties include for $\Hs=\Gs$ the partial flag varieties $\Gs/\Ps$ considered above (see \cite{Akyildiz}), smooth Schubert varieties are regular when $\Hs=\Bs\subset \Gs$ is a Borel subgroup, and Bott--Samelson resolutions will be examples for parabolic subgroups $\Hs=\Ps\subset\Gs$ of reductive groups.

We construct (see Section~\ref{vecsec}) a vector field $V_\he$ on $\he\times X$ such that for any $y\in \he$ its restriction 
$$(V_\he)_y\in H^0(X;T_X)$$ 
to $\{y\}\times X$ is the infinitesimal vector field on $X$ generated by $y$. 
Recall the Kostant section from \eqref{kostant} for a reductive group. For an arbitrary principally paired group, we proceed as follows. Choose a Levi subgroup $\Ls$ in $\Hs$, so that $\Hs = \Ns \rtimes \Ls$, where $\Ns$ is the unipotent radical of $\Hs$. The regular nilpotent $e\in \Hs$ then splits into $e = e_n + e_l$ with $e_n\in \nen$, $e_l\in \lel$. The latter can be completed to an $\ssl_2$-triple $(e_l,f_l,h_l)$ in $\lel$, and we take $\Ss = e + C_{\lel}(f_l)$. We prove in Theorem~\ref{restkos} that it is a section of the natural map $\he\to \he/\!\!/ \Hs\cong \ttt/\!\!/W\cong \Spec(H^*_\Hs)$; in particular, $\Ss\cong \Spec(H^*_\Hs)$.

Denote by $V_\Ss:=V_\he|_{\Ss\times X}$ the vector field $V_\he$ restricted to $\Ss\times X$. Let $\ZZ_\Ss\subset \Ss\times X$ be the zero scheme of $V_\Ss$, \textit{i.e.} the subscheme defined by the sheaf of ideals generated by $V_{\Ss}(\OO_{\Ss\times X})\subset \OO_{\Ss\times X}$, where $V_{\Ss}$ acts on $\OO_{\Ss\times X}$ as a derivation. The distinguished homomorphism $B(\SL_2)\to \Hs$ restricted to the diagonal torus gives a map $\tau\colon\Cs\to \Hs$, and a $\Cs$-action defined as in \eqref{actions} will preserve $\Ss$. We also pull back the action of $\Hs$ via $\tau\colon\Cs\to \Hs$ on $X$ to an action of $\Cs$ on $X$. Then $\ZZ_\Ss$ will be preserved by the diagonal $\Cs$-action on $\Ss\times X$. Our main theorem is the following. 

\begin{theorem}\label{main}
Suppose a principally paired group $\Hs$ acts regularly on a smooth projective complex variety $X$. Then the zero scheme $\ZZ_\Ss\subset \Ss\times X$ of the vector field $V_\Ss$ is reduced and affine, and its coordinate ring, graded by the $\Cs$-action above, is isomorphic as a graded ring
$$
\begin{tikzcd}
  \C[\ZZ_{\Ss}]\arrow{r}{\cong} &
  H_\Hs^*(X;\C)   \\
   \C[\Ss]\arrow{r}{\cong} \arrow{u}{\pi^*}&
  H^*_\Hs \arrow{u}
\end{tikzcd}
$$
to the $\Hs$-equivariant cohomology of\, $X$, so
that the structure map $H^*_\Hs\to H_\Hs^*(X;\C)$ agrees with the pullback map $H^*_\Hs\cong \C[\Ss]\to \C[\ZZ_{\Ss}]$ of the natural projection $\pi\colon\ZZ_\Ss \to \Ss$. In particular, we have 
$$ \begin{tikzcd}
 \ZZ_\Ss  \arrow{d}{\pi}  &
  \Spec(H_\Hs^*(X;\C)) \arrow{d} \arrow{l}{\cong} \\
   \Ss&
  \Spec(H^*_\Hs) \arrow{l}{\cong}\rlap{;}
\end{tikzcd}$$
\textit{i.e.} the spectrum of equivariant cohomology of\, $X$ is $\Cs$-equivariantly isomorphic to the zero scheme $\ZZ_\Ss\subset \Ss\times X$ over $\Ss\cong \Spec(H^*_\Hs)$.
\end{theorem}

We first study  the case of solvable principally paired groups. Then the general case is reduced to the Borel subgroup.

There is another version of our Theorem~\ref{main} where we do not restrict to the Kostant section $\Ss$. Namely, if a reductive group $\Gs$ acts regularly on $X$ and we denote by $\ZZ_{\geg}\subset \geg\times X$ the zero scheme of $V_\geg$, then the $\Gs$-action on $\geg\times X$ leaves $\ZZ_{\geg}$ invariant.
We have the following. 

\begin{theorem}\label{maintot}
Suppose a complex reductive group $\Gs$ acts regularly on a smooth projective complex variety $X$. Then the $\Gs$-invariant part of the algebra of the global functions on the total zero scheme $\ZZ_\geg$
$$ \begin{tikzcd}
  \C[\ZZ_{\geg}]^\Gs \arrow{r}{\cong} &
  H_\Gs^*(X;\C)  \\
  \C[\geg]^\Gs \arrow{r}{\cong} \arrow{u}&
  H^*_\Gs \arrow{u}
\end{tikzcd}$$
is graded isomorphic to the equivariant cohomology of\, $X$ over $\C[\geg]^\Gs\cong H^*_\Gs$. The gradings on $\C[\geg]^\Gs$ and $\C[\ZZ_{\geg}]^\Gs$ are induced by the weight $-2$ action of\, $\Cs$ on $\geg$ and the trivial action on $X$. 
 \end{theorem}
 
Note that for partial flag varieties $X=\Gs/\Ps$, the total zero scheme $\ZZ_{\geg}\cong \tilde{\geg}_\Ps\to \geg$ is just the Grothendieck--Springer resolution as above. 

However, here the total zero scheme is no longer affine. On the other hand, this version also holds for GKM  (Goresky--Kottwitz--MacPherson) spaces, including toric varieties. Recall from \cite{GKM} that a smooth projective variety $X$ with an action of a torus $\Ts$ is a GKM space if the number of both the zero- and the $1$-dimensional orbits is finite. We can form the total zero scheme $\ZZ_\ttt\subset \ttt\times X$ as the zero scheme of the vector field $V_\ttt$ generated by the $\Ts$-action, as before.
 
\begin{theorem} \label{maingkm}
 Suppose that a torus $\Ts$ acts on a smooth projective complex variety $X$ with finitely many zero- and $1$-dimensional orbits. Then the algebra of the global functions on the total zero scheme $\ZZ_\ttt$
 $$ \begin{tikzcd}
  \C[\ZZ_{\ttt}] \arrow{r}{\cong}  &
  H_\Ts^*(X;\C)   \\
   \C[\ttt] \arrow{r}{\cong} \arrow{u}&
  H^*_\Ts \arrow{u}
 \end{tikzcd}$$
 is graded isomorphic to the equivariant cohomology of\, $X$ over $\C[\ttt]\cong H^*_\Ts$. The gradings on $\C[\ttt]$ and $\C[\ZZ_{\ttt}]$ are induced by the weight $-2$ action on $\ttt$.
 \end{theorem}
 
 The proof is straightforward, using the explicit description of $ H_\Ts^*(X;\C)$ from \cite{GKM}. We expect this version to hold for an even larger class of group actions, including spherical varieties. However, in this paper we concentrate on a more restrictive class of regular group actions. In that case, as in Theorem~\ref{main}, we can find an affine zero scheme $\ZZ_\Ss \subset \Ss\times X$ which is precisely the spectrum of equivariant cohomology of $X$.
 
 Our main Theorem~\ref{main} was proved for the case of regular actions of the Borel $\Bs(\SL_2)$ by Brion--Carrell; see \cite[Theorem 1 and Proposition 2]{BC}. The strategy of our proof of Theorem~\ref{main} -- in the case of more general Borel subgroups -- is broadly following the approach of the proof in \cite{BC}. Using vector fields with possibly degenerate isolated zeros to obtain topological information on a complex manifold from infinitesimal information goes back to the pioneering works of Bott \cite{bott,baum-bott}. For a comprehensive survey; see \cite{carrell}. 
 
 We should also mention that there are other papers in the literature which study the spectrum of equivariant cohomology geometrically, see \textit{e.g.} \cite{goresky-macpherson} and the references therein. A more recent example is \cite{hikita}, where the spectrum of equivariant cohomology of certain varieties also appears as a fixed point scheme, albeit of another -- 3D-mirror -- variety. 
 
 We finally note that many of our examples in this paper will be equivariant cohomology rings of partial flag varieties, and as such they model the Hitchin system on various Lagrangian upward flows; see  \cite{hausel-icm,hausel-icmtalk}. The pictures arising \textit{e.g.} in Section~\ref{examples}
 could then be  thought of as depicting the various fixed point schemes, spectra of equivariant cohomology or the Hitchin systems on corresponding upward flows. 
 
 The contents of the paper is as follows. In Section~\ref{secgen} we describe the basic properties of actions of algebraic groups and vector fields associated with them. In particular, in Section~\ref{vecsec} we introduce the \emph{total vector field} which underlies the constructions used throughout the paper. In Sections~\ref{secreg} and~\ref{sl2p} we discuss regular elements and principal integrable $\bb(\ssl_2)$-pairs. In Section~\ref{kostsecsec} we generalize the Kostant section to arbitrary principally paired groups, and in Section~\ref{secregact} we discuss basic properties of regular actions. Section~\ref{solvsec} contains the proof of Theorem~\ref{finsolv}, which is the equivalent of Theorem~\ref{main} for the solvable group. Based on that, in Section~\ref{redarbsec} we prove Theorem~\ref{main}. In Section~\ref{secsing} we generalize the theorem to some singular varieties. Finally, Sections~\ref{sectot} and~\ref{secgkm} contain the proofs and examples for Theorems~\ref{maintot} and~\ref{maingkm}.
 
\subsection*{Acknowledgments} We would like to thank David Ben-Zvi, Michel Brion, Jim Carrell, Harrison Chen, Nigel Hitchin, Quoc Ho, Vadim Kaloshin, Friedrich Knop, Jakub L\"owit, Anne Moreau, Rich\'ard Rim\'anyi, Andr\'as Szenes, Zsolt Szil\'agyi, Michael Thaddeus and Zhiwei Yun for useful comments and discussions. We also thank the referees for useful comments.  All figures were generated in Mathematica.

\section{Generalities} 
\label{secgen}

\subsection{Notation}

We consider all  the algebraic varieties, including algebraic groups, to be defined over $\C$. For an algebraic variety $X$, by $\C[X] = \OO_X(X)$ we denote the algebra of regular functions on $X$. All the cohomology groups will be understood to have complex coefficients. For a Lie algebra $\geg$ and a subset $V\subset \geg$, we denote by $C_\geg(V)$, $N_\geg(V)$ the centraliser and normaliser of $V$ in $\geg$, respectively. If $V = \{v\}$, then we also write $C_\geg(v)$, $N_\geg(v)$. We drop the lower index if the ambient Lie algebra is obvious. For any $\Z_{\ge 0}$-graded $\C$-algebra $R = \bigoplus_{n=0}^\infty R_n$, we denote by $P_R(t)$ its Poincar\'{e} series, \textit{i.e.}
$$P_R(t) = \sum_{n=0}^\infty \dim_\C(R_n) t^n.$$
Let $\diag(v_1,v_2,\dots,v_n)$ be the diagonal $n\times n$ matrix with diagonal entries $v_1$, $v_2$, \dots, $v_n$. We will denote by $I_n = \diag(1,1,\dots,1)$ the $n\times n$ identity matrix. For any algebraic group $\Gs$ with Lie algebra $\geg$, by $\geg_n$ we denote the set (in general not a subalgebra) of nilpotent elements of $\geg$, as defined in \cite[Section~I.4.5]{Borel}. For a commutative algebra $A$ with a filtration $F_\bullet$, we denote by $\Gr_F(A)$ the associated graded algebra.

\subsection{Vector fields}
\label{vecsec}

Recall that a vector field on a smooth algebraic variety $X$ is a derivation on the sheaf of regular functions on $X$. This means that for any Zariski-open subset $U\subset X$, we are given a $\C$-linear derivation $\OO_X(U) \to \OO_X(U)$, and it is natural with respect to $U$. Given a vector field $V$ on $X$, if $x\in X$ is a closed point in $X$, we can restrict the derivation defined by $V$ to the local ring $\OO_{X,x}$. By restricting to the maximal ideal $\m_x\subset \OO_{X,x}$ and evaluating the derivations of functions at $x$, we get a map $\m_x\to \C$. In fact, by the Leibniz rule it has to vanish on $\m_x^2$; hence we get a tangent vector $V_x\in \Hom_{\C}(\m_x/\m_x^2,\C) \simeq T_{x,X}$.

Whenever an algebraic group $\Hs$ acts on a variety $X$, it yields a Lie algebra homomorphism $\phi\colon \he\to\Vect(X)$ from $\he = \Lie(\Hs)$ to vector fields on $X$; see \cite{CoDr}. We will want to define the total vector field on $\he\times X$. As this is a local problem on $X$, we can restrict to an affine open set $U$. Then 
\begin{align}
\label{ohtu}
\C[\he\times U] = \C[\he] \otimes_\C \C[U], 
\end{align}
and we need to define a derivation on this $\C$-algebra. We can view $\phi|_U$ as an element of $\he^*\otimes_\C \Vect(U)$. As $\C[\he] = S^*(\he^*)$, we have a multiplication map $\he^* \otimes \C[\he] \to \C[\he]$. Additionally, the elements of
$\Vect(U)$ are by definition the derivations on $\C[U]$, which gives a $\C$-bilinear map $\Vect(U) \otimes \C[U] \to \C[U]$. Those two maps together with \eqref{ohtu} lead to a $\C$-bilinear map
$$(\he^* \otimes \Vect(U)) \otimes \C[\he\times U] \lra \C[\he\times U]. $$
Fixing $\phi|_U\in (\he^* \otimes \Vect(U))$ gives a derivation $\C[\he\times U] \to \C[\he\times U]$.

\begin{definition}\label{totvec}
The vector field defined by this derivation will be called the \emph{total vector field} of $\Hs$-action on~$X$.
\end{definition}
Explicitly, let $\phi = \sum \psi_i \otimes D_i$ for $\psi_i\in \he^*$, $D_i\in\Vect(U)$.
Then the defined derivation on $f \otimes g \in \C[\he]\otimes \C[U]$ takes value
\begin{align}
\label{derim}
    \sum (\psi_i \cdot f) \otimes D_i(g) \in \C[\he]\otimes \C[U].
\end{align}
This gives the total vector field on $\he\times X$. One can note that the vector field is tangent to $\{y\}\times X$ for any $y\in \he$; \textit{i.e.} as a derivation it preserves the set of functions vanishing on $\{y\}\times X$. Indeed, locally such functions are sums of $f \otimes g \in \C[\he]\otimes \C[U]$ such that $f(y) = 0$, and in such case the image of the derivation \eqref{derim} also vanishes at $\{y\}\times X$. The vector field restricted to $\{y\} \times X$ is precisely $\phi(y)$, and for any $y\in\he$ with $\Hs$ acting on $X$, we will denote this vector field by $V_y$. Later we will consider restrictions of the total zero schemes to bigger subsets of $\he$.
 
One sees that for any $y\in\he$ and $x\in X$, the value $V_y|_x$ of the vector field $V_y$ at $x$ can be recovered by considering the derivative at $1_{\Hs}$ of the map $\Hs\to X$ defined as $g\mapsto g\cdot x$ and evaluating it on $y$.

\begin{definition} Let $V$ be a vector field on a smooth variety $X$. For each open set $U\subset X$, it gives a derivation $D_V^U\colon  \OO_X(U) \to \OO_X(U)$. Let us consider the ideal sheaf generated by the image $D_V(\OO_X) \subset \OO_X$. This is the defining ideal of the \emph{zero scheme} of $V$ on $X$.
\end{definition}

\begin{remark}\label{remtan}
One can also view vector fields on smooth varieties as sections of the tangent bundle. As the tangent bundle is a locally free sheaf, we can define the zero scheme of the vector field by considering it locally as a tuple of regular functions (see Lemma~\ref{lemcm}). In other words, if the tangent bundle is free over an open subset $U \subset X$, after the choice of a trivialisation, its section $V$ is defined by $n$-tuple of regular functions $f_1$, $f_2$,\dots, $f_n$. Then the zero scheme of $V$ on $U$ is the zero scheme of the ideal $(f_1,f_2,\dots,f_n)\in \OO_X(U)$.
\end{remark}

\subsection{Background results on algebraic groups and vector fields}

We first recall (part of) the theorem of Borel on solvable groups (\textit{cf.} \cite[Theorem 10.6]{Borel}, see also \cite[Theorem 16.33]{Milne}) that we will often tacitly use throughout.

\begin{theorem}\label{solv}
Let $\Hs$ be a connected solvable group with Lie algebra $\he$ and $\Hs_u$ its set of unipotent elements. Then:
\begin{enumerate}
\item $\Hs_u$ is a connected normal closed, unipotent subgroup of\, $\Hs$ containing $[\Hs,\Hs]$. 
\item The maximal tori in $\Hs$ are all conjugate. If\, $\Ts$ is a maximal torus, then $\Hs = \Hs_u \rtimes \Ts$. The Lie algebra of $\Hs_u$ consists of all nilpotent elements of\, $\he$.
\item If\, $\Ts$ is a maximal torus, then any semisimple element of\, $\Hs$ is conjugate to a unique element of\, $\Ts$.
\end{enumerate}
\end{theorem}

\begin{remark}
\label{nilalg}
Let $\he_n$ be the set of nilpotent elements of $\he$. It follows from the above that $\he_n$ is a Lie subalgebra of~$\he$. As it consists of nilpotent elements, hence acts nilpotently by the adjoint action, by Engel's theorem it is nilpotent itself. Moreover, it contains $[\he,\he]$. In addition, from the second statement we get that $\he = \he_n \oplus \ttt$ for $\ttt = \Lie(\Ts)$.
\end{remark}

Now assume we are given a group action $\Hs\lefttorightarrow X$ of an algebraic group. For any $g\in \Hs(\C)$ the action of $g$ is an isomorphism $X\to X$. If we fix any closed point $x\in X$, its derivative $\D g|_{x}$ at $x$ is an isomorphism $T_{x,X}\to T_{gx,X}$. We will simply write it as $\D g$ if $x$ can be inferred from the context.

\begin{lemma} \label{lemad}
 Let an algebraic group $\Hs$ act on a variety $X$. Then for any $g\in \Hs$, $y\in\he = \Lie(\Hs)$ and $x\in X$, we have
 $${V_{\Ad_g(y)}}|_{gx} = \D g\left({V_y}|_{x}\right).$$
\end{lemma}

\begin{proof}
Let $\mu\colon \Hs\times \Hs\to \Hs$ denote the multiplication map and $\rho\colon \Hs\times X\to X$ denote the action of $\Hs$ on $X$. Consider the following commutative diagram: 
$$
\begin{tikzcd}
& \Hs\times \Hs\times \Hs\times X \arrow[ld, "\mu\times\id\times\id"] \arrow[rd, "\id\times\id\times\rho"]
\\
\Hs\times \Hs\times X\arrow[dd, "\mu\times\id"] & & \Hs\times \Hs\times X\arrow[dd, "\id\times\rho"]
\\ \\
\Hs\times X\arrow[dr,"\rho"] && \Hs\times X\arrow[dl,"\rho"]
\\
& X\rlap{.}
\end{tikzcd}
$$
If we fix a point at the top, it yields an analogous commutative diagram of differential maps. Take $(g,1,g^{-1},gx) \in \Hs\times \Hs\times \Hs\times X$ and $(0,y,0,0)$ in its tangent space. Going through the left branch, it is mapped to ${V_{\Ad_g(y)}}|_{gx}$, and going through the right one, it is mapped to $\D g({V_y}|_{x})$.
\end{proof}

\begin{lemma} \label{algder}
 Let $A$ be a commutative $\C$-algebra. Let $D_Y\colon A\to A$ be a $\C$-linear derivation and $\V$ a $\C$-vector space of\, $\C$-derivations $A\to A$ normalised by $D_Y$; \textit{i.e.} for any $D_W\in \V$ we have $[D_W,D_Y]\in \V$. Let $\m_x$ be a radical ideal in $A$ that contains $\im D_W$ for all $D_W\in \V$. Then for any $f\in \sqrt{(\im D_W)_{D_W\in \V}}$ we have $D_Y f\in \m_x$.
\end{lemma}

\begin{proof}
Let $\I = (\im D_W)_{D_W\in \V}$ be the ideal generated by images of all the derivations from $\V$. We first prove by induction that $(D_Y)^n(I)\subset \m_x$ for all $n\ge 0$. The case $n=0$ follows from the assumption that $\I\subset\m_x$. Now assume that $(D_Y)^n(I)\subset \m_x$ for some $n\ge 0$. Fix one particular derivation $D_W\in \V$; we will want to prove that $(D_Y)^{n+1} \im D_W\subset \m_x$. We have $D_Y D_W - D_W D_Y = D_Z\in \V$; therefore, 
$$D_Y^{n+1}D_W - D_Y^{n} D_W D_Y = D_Y^n D_Z,$$
 hence
$$D_Y^{n+1}D_W = D_Y^{n}(D_W D_Y + D_Z).$$
Now clearly $\im D_W D_Y + D_Z\subset \I$; hence by the inductive assumption the image of the right-hand side is always in $\m_x$. Therefore, $\im D_Y^{n+1}D_W \subset \m_x$, as we wanted to prove.

Now assume that $f\in \sqrt{\I}$ and let $f^k \subset \I$. We then know that $f^k\in \m_x$, therefore $f\in\m_x$. By the above we also know that $D_Y^k f^k\in \m_x$. By the Leibniz rule $D_Y^k f^k$ is the sum of terms of the form
$$\prod_{i=1}^k \left(D_Y^{\alpha_i} f\right)$$
for non-negative integers $\alpha_1,\alpha_2\dots,\alpha_k$ such that $\alpha_1+\alpha_2+\dots+\alpha_k = k$. Note that for all the terms except for $(D_Y f)^k$, at least one of $\alpha_1$, $\alpha_2$, \dots, $\alpha_k$ is zero, and all those terms belong to $\m_X$ as $f\in\m_x$. Therefore, we get $(D_Y f)^k\in \m_x$, hence $D_Y f\in \m_x$.
\end{proof}

As a geometric counterpart, we get the following lemma, which will prove very useful in our proofs.

\begin{lemma} \label{lemfix}
 Let $Y$ be a vector field on a smooth variety $X$. Assume that $\V$ is a subspace of the $\C$-vector space of all global vector fields. If\, $Y$ normalises $\V$, \textit{i.e.} $[Y,\V]\subset \V$, then $Y$ is tangent to the reduced zero scheme of\, $\V$.
 
 In particular, if a Lie group $\Hs$ acts on $X$ and a subspace $\V\subset \he$ has isolated $($simultaneous$)$ fixed points, then they are fixed by the normaliser $N_{\he}(\V)$ of\, $\V$ in $\he$. 
\end{lemma}

Note that even the reduced zero scheme of $\V$ might be singular. A vector from a tangent space to $X$ is considered tangent to a subscheme $Z$ if it is in the image of the tangent space of $Z$; see the discussion in Section~\ref{vecsec}. Equivalently, in a local affine neighbourhood, it annihilates all the functions that vanish on $Z$, \textit{i.e.} those from the defining ideal of $Z$.

\begin{proof}
As the statement is local, we can assume that $X = \Spec A$ is affine. Let $x\in X$ be a simultaneous zero of $\V$. Then $x$ corresponds to a maximal ideal $m_x\triangleleft A$. The space $\V$ gives rise to a vector space of $\C$-derivations $A\to A$, and $Y$ to a single derivation $D_Y\colon A\to A$. By the assumption on $x$, for any $D_W\in V$ we have $\im D_W\subset m_x$. Hence, by Lemma~\ref{algder}, the derivation $D_Y$ vanishes at the point $x$ on the ideal of the reduced zero scheme of $\V$. Thus $Y$ is tangent to that scheme.
\end{proof}

\begin{remark}
 There is an alternate, analytic proof, which works under the assumption that $\V$ is finite-dimensional -- which will  always be the case for us. It is non-algebraic and hence also non-translatable to other fields, but one could argue it is less technically demanding and moreover works in a smooth, not necessarily algebraic setting, so we present it here as well. In fact, the assumption that $\V$ is finite-dimensional can also be dropped if we use the fact that the functions we deal with are all analytic, hence they vanish locally if all the derivatives in a point vanish -- this approach mimics the algebraic proof.

 Let $\phi = [Y,-]\big|_{\V}$ be the commutator map $\V\to \V$ induced by $Y$.
 Let $x$ be fixed by $\V$, and let us consider a local one-parameter subgroup $\Psi_t$ around $x$ defined by the vector field $Y$.
 For any vector field $W$, we have
 $$[Y,W]_x = \frac{d}{dt}\left(\left(D_x \Psi_t\right)^{-1} W_{\Psi_t(x)}\right)\big|_{t=0}$$
 and analogously
 $$[Y,W]_{\Psi_t(x)} = \frac{d}{du}\left((D_x \Psi_u)^{-1} W_{\Psi_{t+u}(x)}\right)\big|_{u=0}.$$
 Composing this with the linear map $(D_x\Psi_t)^{-1}$, we get, for $W\in \V$, the following:
 $$\left(D_x\Psi_t\right)^{-1}\phi(W)_{\Psi_t(x)} = \frac{d}{du}\left( \left(D_x \Psi_u\right)^{-1} W_{\Psi_{u}(x)}\right)\big|_{u=t}.$$
 Hence if we consider the map $\tau\colon (-\eps,\eps)\to \Hom(\V,T_p X)$ defined as
 $$\tau(t)(Y) = (D_x \Psi_u)^{-1} W_{\Psi_{t}(x)}, $$
 we get
 $$\frac{d}{dt} \tau(t) = \phi^* \tau(t).$$
 
 We get a linear equation, and in particular, as $\tau(0)$ vanishes (because $\V$ vanishes at $x$), we get  that $\tau$ vanishes  around $0$ as well; hence $\tau$ moves along fixed points of $\V$.
\end{remark}

The next lemma will be used to show that zeros of generalised Jordan matrices are zeros of the torus.

\begin{lemma} \label{lemzer}
Let a Lie algebra $\he$ acts on a smooth variety $X$. Let $d,n \in \he$ commute, and assume that the Lie subalgebra generated by $[\he,\he]$ and $n$ is nilpotent. Let $x\in X$ be an isolated zero of the vector field $V_j$ associated to $j=d+n$. Then $x$ is also a simultaneous zero of\, $C_{\he}(d)$. In particular, $x$ is a zero of any abelian subalgebra of\, $\he$ containing $d$.
\end{lemma}

\begin{proof}
Let $\kek$ be the Lie subalgebra generated by $[\he,\he]$ and $n$. By Lemma~\ref{lemfix} we first get that $x$ is a zero of $d$ and $n$, as they commute with $j$.

We will first prove that $x$ is a zero of $C'(d) = C_{\he}(d)\cap \kek$. As $\kek$ is nilpotent by assumption, its subalgebra $C'(d)$ is nilpotent as well.

By definition $d$ is in the center of $C(d)$; in particular, it commutes with $C'(d)$. Hence from Lemma~\ref{lemfix} we have that $x$ is a zero of $N_{C'(d)}(\C\cdot n)$. It is therefore an isolated simultaneous zero of $d$ and $N_{C'(d)}(\C\cdot n)$, and we can apply the same argument repeatedly to get that for $i=1,2,\dots$ it is a zero of $N_{C'(d)}^i(\C\cdot n)$.

The sequence $(N_{C'(d)}^i(\C\cdot n))_{n=1}^\infty$ has to stabilise at a Lie subalgebra of $C'(d)$ which is its own normaliser in $C'(d)$. As $C'(d)$ is nilpotent, it then has to be equal to the whole $C'(d)$ (see \cite[Section~I.4.1, Proposition~3]{BouLie13}). Therefore, $d$ and $C'(d)$ vanish at $x$. But $[C_{\he}(d),C_{\he}(d)] \subset C_{\he}(d)\cap [\he,\he] \subset C_{\he}(d)\cap \kek = C'(d)$; hence $C'(d)$ is normalised by the whole $C_{\he}(d)$. Therefore, by Lemma~\ref{lemfix} the whole $C_{\he}(d)$ vanishes at $x$.
\end{proof}

From Remark~\ref{nilalg} the assumptions about $d$ and $n$ hold whenever $\he$ is solvable, $[d,n]=0$ and $n\in \he_n$ (as $\he_n$ is nilpotent and contains $[\he,\he]$ as well as $n$).

\subsection{Regular elements} \label{secreg}
Let $\Hs$ be an algebraic group and $\Ts\subset \Hs$ be a maximal torus, of dimension $r$. We will call an element $v\in \he = \Lie(\Hs)$ \emph{regular} if $\dim C_\he(v) = r$. 
This is stronger than the usual notion of a regular element in the literature (see \textit{e.g.} \cite{Chev}) -- an element whose centraliser has minimal possible dimension. All the centralisers have dimension at least $r$, but it is possible that no regular element exists. For example for $\Hs = \Cs \times \C$ -- the product of the multiplicative and the additive group -- all centralisers are $2$-dimensional.

\begin{example}
 For $\Hs = \GL_n(\C)$ or $\Hs = \SL_n(\C)$, a regular element of $\he$ is a matrix with all eigenspaces of dimension $1$. Among the regular elements, the regular semisimple ones are the diagonalisable matrices with distinct eigenvalues, and a regular nilpotent matrix is conjugate to a single Jordan block.
\end{example}

\begin{example}\label{exregnilp}
More generally, any reductive group $\Gs$ contains regular elements in its Lie algebra, in particular a regular nilpotent element. Indeed, once we choose a maximal torus $\Ts\subset \Gs$ and positive roots, we can take $e = x_1 + x_2 + \dots + x_s$, where $x_1$, $x_2$, \dots, $x_s$ are the root vectors of $\geg$ corresponding to the positive simple roots ($s=r - \dim Z(\Gs)$). Then $e$ is a regular nilpotent in $\Gs$ (see \cite[Section 4.2, Theorem 4]{Kostsec}).
\end{example}

 The condition $\dim C_{\he}(w) > r$ is a Zariski-closed condition on $w$ as it means that $[w,-]$ has sufficiently small rank, which amounts to the vanishing of some minors of a matrix. Therefore, if $\Hs$ admits a regular element in its Lie algebra, the subset of regular elements $\he^\reg\subset \he$ is open and dense.

Note that if $\Hs$ is solvable, then by Theorem~\ref{solv} we have $[\he,\he]\subset \he_n$. This means that for any $v\in\he$ we have $[v,\he]\subset \he_n$. As the codimension of $\he_n$ is exactly $r = \dim \Ts$, the dimension of maximal torus, $v$ being regular is equivalent to $[v,\he] = \he_n$.

Also note  that if $\Hs'\subset \Hs$ is a subgroup which contains a maximal torus $\Ts$ of $\Hs$, then any regular $v\in \he$ contained in $\he'$ is also regular in $\he'$. Indeed, if $r = \dim \Ts$, then $\dim C_{\he'}(v) \le \dim C_{\he}(v) = r$, but at the same time $\dim C_{\he'}(v)$ cannot be less than the dimension of the maximal torus $\Ts$ of $\Hs'$. This means in particular that the centraliser $C_\he(v)$ is contained in $\he'$.

\subsection{\texorpdfstring{$\ssl_2$-triples and $\bb(\ssl_2)$-pairs}{sl2-triples and b(sl2)-pairs}}
\label{sl2p}

 The classical version of the Carrell--Liebermann theorem (\textit{cf.} \cite[Main Theorem and Remark 2.7]{CL}) deals with an arbitrary vector field $V$ on a smooth projective variety $X$, which vanishes in a discrete, non-empty set. They prove the following

 \begin{theorem}
  Let $X$ be a smooth projective complex variety and $V$ a vector field with finitely many zeros, and denote its zero scheme by $Z$. Then there exists an increasing filtration $F_\bullet$ on $\C[Z]$ such that
  $$H^*(X) \simeq \Gr_F(\C[Z]).$$
It is an isomorphism of graded $\C$-algebras, and with respect to the grading, the degree of an element is multiplied by 2 when we switch from the right-hand side to the left-hand side. In particular, $X$ only has even cohomology.
 \end{theorem}

 The theorem therefore gives some information on cohomology, but this depends on determining the filtration $F_\bullet$. This can be hard in general. Only if $V$ comes with a $\Cs$-action which satisfies $t_*(V) = t^k V$ for some non-zero integer $k$, do we get $H^*(X) \cong \C[Z(V)]$ (see \cite{ACLS}, \cite[Theorem 1.1]{AC}). We will consider those vector fields as coming from an action of a Lie group. Hence we have the following definition.

\begin{definition}
 For any complex Lie algebra $\he$, by a \emph{$\bb(\ssl_2)$-pair} in $\he$, we mean a pair $(e,h)$ of elements of $\he$ that satisfy the condition $[h,e] = 2e$. By an \emph{$\ssl_2$-triple} in $\he$, we mean a triple $(e,f,h)$ of elements of $\he$ such that $[h,e]=2e$, $[h,f] = -2f$, $[e,f] = h$.
\end{definition}

If $\Gs$ is a semisimple group, then by the Jacobson--Morozov theorem (see \textit{e.g.} \cite[Theorem 3.7.1]{ChGi}), for any nilpotent element $e\in\geg$, there exists an $\ssl_2$-triple $(e,f,h)$ in $\geg$ such that $f$ is nilpotent and $h$ is semisimple. The same is then true for any reductive Lie group $\Gs$ as a reductive Lie algebra is a direct sum of its center and a semisimple ideal (\textit{cf.}~\cite[Theorem II.11]{Jac}).

Let us consider the connected subgroup $\Ks\subset \Gs$ whose Lie algebra $\mathfrak k$ is the smallest one which contains $e$, $f$, $h$ (see \cite[Section~II.7.1]{Borel}).
Then the Lie algebra of $[\Ks,\Ks]$ is equal to $[\mathfrak k,\mathfrak k]$ (see \cite[Proposition 7.8]{Borel}). However, by \cite[Corollary 7.9]{Borel} we have $[\mathfrak k,\mathfrak k] = [\Span(e,f,h),\Span(e,f,h)] = \Span(e,f,h)$. Hence we get an algebraic subgroup $[\Ks,\Ks]$ (contained in $\Ks$, hence equal to $\Ks$) of $\Gs$ whose Lie algebra is $\Span(e,f,h)$. As its Lie algebra is semisimple, the group itself is semisimple. By \cite[Theorem 20.33]{Milne}, if it is non-trivial, it has to be either $\SL_2(\C)$ or $\PSL_2(\C)$. In either case, there is a covering map
\begin{align}\label{phi}\phi\colon \SL_2(\C)\lra \Ks.\end{align}
As any automorphism of $\ssl_2(\C)$ lifts to an automorphism of $\SL_2(\C)$, we can assume that the canonical basis $e_0$, $f_0$, $f_0$ of $\ssl_2$ maps to $e$, $f$, $h$, respectively. Hence we get the following.

\begin{proposition}\label{reductive}
 For any nilpotent element $e$ in the Lie algebra $\geg$ of an algebraic reductive group, there exists an $\ssl_2$-triple $(e,f,h)$ within $\geg$ with $f$ nilpotent and $h$ semisimple. If $e\neq 0$, the element $h$ integrates to a map $\Cs\to \Gs$ with discrete kernel, whose differential is $h$.
\end{proposition}

\begin{remark}\label{princesl2}
As we saw in Example~\ref{exregnilp}, if $\Gs$ is reductive, then there exists a principal nilpotent $e\in \geg$. By Proposition~\ref{reductive} this means that there is an $\ssl_2$-triple $(e,f,h)$ with $e$ principal nilpotent, $f$ nilpotent and $h$ semisimple. By the general theory of representations of $\ssl_2$, the ranks of the operators $[e,-]$ and $[f,-]$ are equal; hence $f$ is also regular. This motivates the following definition.
\end{remark}

\begin{definition}
 An $\ssl_2$-triple $(e,f,h)$ will be called \emph{principal} if $e$ and $f$ are regular nilpotents.
\end{definition}

\begin{definition}
 For a linear algebraic group $\Hs$, an \emph{integrable $\bb(\ssl_2)$-pair} in $\he = \Lie(\Hs)$ is an $\ssl_2$-pair $(e,h)$ in $\he$ which consists of a nilpotent element $e$ and a semisimple element $h$ which is tangent to some one-parameter subgroup $H\colon \Cs\to \Hs$; \textit{i.e.} $h = DH_{|1}(1)$. This means that $(e,h)$ comes from an algebraic group morphism $\Bs_2 = \Bs(\SL_2)\to \Hs$. We call an integrable $\bb(\ssl_2)$-pair \emph{principal} if $e$ is a regular element of~$\he$.
\end{definition}

\begin{remark}
 Note that, unlike an $\ssl_2$-triple, a $\bb(\ssl_2)$-pair does not have to be integrable. As an easy counterexample, we may take
 $$h = 
 \begin{pmatrix}
  \pi & 0 & 0\\
  0 & \pi-2 & 0 \\
  0 & 0 & 2-2\pi
 \end{pmatrix},
 \quad
 e = 
 \begin{pmatrix}
  0 & 1 & 0\\
  0 & 0 & 0 \\
  0 & 0 & 0
 \end{pmatrix}
 $$
 for $\Hs = \SL_3(\C)$. Then $[h,e] = 2e$, but $h$ is not tangent to a $1$-dimensional torus (we can replace $\pi$ with any irrational number).
\end{remark}

\begin{definition} We call a connected linear algebraic group $H$ {\em principally paired} if it contains a principal integrable $\bb(\ssl_2)$-pair. 
\end{definition}

For example, a reductive group is principally paired because of Proposition~\ref{reductive}. More generally, we have the following. 

\begin{lemma} \label{parabolic}
  Let $\Gs$ be a reductive group. Then any parabolic subgroup $\Ps\subset \Gs$ is principally paired.
 \end{lemma}

\begin{proof}
Because there is a Borel subgroup $\Bs\subset \Ps$, it is enough to prove the result for $\Bs=\Ps$. Note that if $\Bs = \Bs_2$ is the Borel subgroup of $\SL_2(\C)$, then the image $\phi(\Bs_2)$ of \eqref{phi} is a solvable connected subgroup of $\Gs$; hence it is contained in a Borel subgroup of $\Gs$. All Borel subgroups of $\Gs$ are conjugate (see \cite[Theorem 11.1]{Borel}); hence they are all principally paired.
\end{proof}

\subsection{Kostant section and generalisations}\label{kostsecsec}
The seminal work of Kostant shows the following theorem (\textit{cf.} \cite[Theorem 0.10]{Kostsec}).

\begin{theorem}\label{kostant2}
Assume that $\Gs$ is a semisimple group and $(e,f,h)$ is a principal $\ssl_2$-triple. 
Then every regular element of $\geg = \Lie(\Gs)$ is conjugate to exactly one element of $\Ss = e + C_\geg(f)$. Moreover, the restriction $\C[\geg]^\Gs \to \C[\Ss]$ is an isomorphism.
\end{theorem}

The affine plane $\Ss$ is called the \emph{Kostant section}. We will provide in Theorems~\ref{kostarb} and~\ref{restkos} a version that works for arbitrary principally paired groups.

\subsubsection{Solvable groups} \label{secsolvkos}
First assume  that $\Hs$ is a solvable group. Let $\Ts$ be its maximal torus and $\he_n$ be the nilpotent part of $\he = \Lie(\Hs)$. Assume that $e\in\he_n, h\in\ttt$ are such that $(e,h)$ is a principal integrable $\bb(\ssl_2)$-pair. Let $\{H^t\}_{t\in\Cs}$ be the one-parameter subgroup in $\Hs$ to which $h\in\he$ integrates.

\begin{lemma}\label{etreg}
 All elements of $e+\ttt$ are regular and not conjugate to one another.
\end{lemma}

\begin{proof}
 Assume that for some $v\in \ttt$ the element $e+v$ is not regular. This means that $\dim C_\he(v) \ge r+1$. As $\Ad_{H^t}(e+v) = t^2 e + v$, for any $t\in \Cs$ we have
 $$\dim C_\he\left(e + v/t^2\right) = \dim C_\he\left(t^2e + v\right) = \dim C_\he (e+v) \ge r+1.$$
 As the set of non-regular elements is closed in $\he$, we get $\dim C_\he(e) \ge r+1$. This contradicts the regularity assumption.
 
 For any $x \in \he$ and $M\in \Hs$, we have $\Ad_M(x) - x \in [\he,\he] \subset \he_n$ by \cite[Propositions 3.17 and~7.8]{Borel}. Therefore, no two distinct elements from $e+\ttt$ can be conjugate to one another as they differ on the $\ttt$-component. 
\end{proof}

The following lemma is based on an argument provided  by Anne Moreau.

\begin{lemma}\label{etkost}
 Every regular element of\, $\he$ is conjugate to a unique element of $e+\ttt$.
\end{lemma}

\begin{proof}
  We know that $\he = \ttt \oplus \he_n$. Assume that $x = v + n$, where $v\in\ttt$ and $n\in\he_n$, is regular. This means that $[x,\he] = \he_n$ (see Section~\ref{secreg}). Let us consider the map
  \begin{align} \label{ad} \Ad_{-}(x)\colon \Hs \lra \he.\end{align}
  As in the proof of Lemma~\ref{etreg}, we see that the image is actually contained in $v + \he_n$.
 
 Note that the image of the derivative of \eqref{ad} at $1$ is $[x,\he] = \he_n = T_0(v+\he_n)$. Therefore, by \cite[Theorem 4.3.6]{Springer} the morphism $\Ad_{-}(x)\colon \Hs \to v+\he_n$ is dominant. Analogously, as $e+v$ is regular by Lemma~\ref{etreg}, the morphism $\Ad_{-}(e+v)\colon \Hs \to v+\he_n$ is dominant. Therefore, the images of $\Ad_{-}(x)$ and $\Ad_{-}(e+v)$ are both dense in $v+\he_n$. By \cite[Theorem 1.9.5]{Springer} they both contain open dense subsets of $v+\he_n$, and hence they intersect, which means that $x$ and $e+v$ are conjugate.
 
 The uniqueness follows from Lemma~\ref{etreg}.
\end{proof}

Now we will also provide an equivalent of the classical Jordan form, for arbitrary solvable groups. Recall that by Remark~\ref{nilalg}
every $x\in\he$ is of the form $x = \w + n$, where $\w\in\ttt$ and $n\in\he_n$.

\begin{theorem}\label{jordan}
 For any $x = \w + n\in\he$ with $\w\in\ttt$, $n\in\he_n$, there exists an $M\in \Hs$ such that $x = \Ad_M(\w+n')$ with $[\w,n']=0$ and $n'\in\he_n$.
\end{theorem}

\begin{proof}
  We have the Jordan decomposition (see \cite[Theorem 4.4]{Borel}) $x = x_s + x_n$, where $x_s$ is semisimple, $x_n$ is nilpotent and $[x_s,x_n]=0$. Then by Theorem~\ref{solv} the element $x_s$ is conjugate to an element of $\ttt$. Hence there exists an $M\in \Hs$ such that $\Ad_{M^{-1}}(x_s) \in \ttt$. Note that
  \begin{align*}\Ad_{M^{-1}}(x_s) - x_s\in [\he,\he]\end{align*}
  as in the proof of Lemma~\ref{etreg}. Moreover, 
  \begin{align*} x_s-\w = (x-x_n) - (x-n) = n-x_n\in \he_n.\end{align*}
  As $[\he,\he]\subset \he_n$ by Theorem~\ref{solv}, we therefore get $\Ad_M^{-1}(x_s) - \w \in \he_n$. As both $\Ad_{M^{-1}}(x_s)$ and $\w$ lie in $\ttt$, we get that they are equal. Therefore, putting $n' = \Ad_{M^{-1}}x_n$ we get
$$x = x_s + x_n = \Ad_M(\w) + \Ad_M(n') = \Ad_M(\w + n'),$$
and the conditions are satisfied.
\end{proof}

\noindent
Note that if $\w\in\ttt^\reg:=\ttt\cap \he^\reg$ is a regular element in $\ttt$, then the only nilpotent $n'$ commuting with $\w$ is $0$. Therefore, we get the following. 

\begin{corollary}\label{corre}
For every $\w\in\ttt^\reg$ and $n\in\he_n$, the elements $\w$ and $n+\w$ are conjugate.
\end{corollary}

\begin{example}\label{exjor} Let us give two examples for $\Hs = \Bs_3$, the Borel subgroup (of upper-triangular matrices) of $\SL_3(\C)$. Let the principal nilpotent element $e$ be of the form
$$e = \begin{pmatrix}
 0 & 1 & 0\\
 0 & 0 & 1\\
 0 & 0 & 0
\end{pmatrix}.$$
\begin{enumerate}
\item Let $\w\in\ttt$ be of the form $\w = \diag(0,v_1,v_2) - \frac{v_1+v_2}{3}I_3$ with $v_1\neq 0$, $v_2\neq 0$, $v_1\neq v_2$. Then note that the matrix $e+\w$ is diagonalisable in the basis defined by the matrix
$$
M_{\w} =
\begin{pmatrix}
 1 & \frac{1}{v_1} & \frac{1}{v_2(v_2-v_1)} \\
 0 & 1 & \frac{1}{v_2-v_1} \\
 0 & 0 & 1 
\end{pmatrix}; 
$$
\textit{i.e.}
$
 e + \w = 
M_{\w} \w M_{\w}^{-1}
$.

\item Consider the matrix $e+\w$, where $\w\in\ttt$ is of the form $\w = \diag(0,v_1,0) - \frac{v_1}{3}I_3$ with $v_1\neq 0$.
If we take
$$M_{\w} =
\begin{pmatrix}
 1 & \frac{1}{v_1} & 0 \\
 0 & 1 & 1 \\
 0 & 0 & -v_1 
\end{pmatrix},$$
then 
$$
\begin{pmatrix}
 0 & 1 & 0 \\
 0 & v_1 & 1 \\
 0 & 0 & 0
\end{pmatrix} = 
M_{\w}
\begin{pmatrix}
 0 & 0 & 1 \\
 0 & v_1 & 0 \\
 0 & 0 & 0
\end{pmatrix}
M_{\w}^{-1}.
$$
Therefore, for $e+\w = \begin{pmatrix}
 0 & 1 & 0 \\
 0 & v_1 & 1 \\
 0 & 0 & 0
\end{pmatrix} - \frac{v_1}{3}I_3$ we get 
$$(e+\w) = M_{\w}
\begin{pmatrix}
 -v_1/3 & 0 & 1 \\
 0 & 2v_1/3 & 0 \\
 0 & 0 & -v_1/3
\end{pmatrix} M_{\w}^{-1}.$$
The matrix $M_{\w}$ used here does not have determinant $1$. We can however multiply it by any cubic root of $v_1^{-1}$ to get a matrix from $\Bs_3$.
\end{enumerate}
In the case of $\Hs = \Bs_n$, one can apply the following intuition. If $\w$ is a regular element of $\ttt$, then it is a diagonal matrix with distinct eigenvalues. Then if we add any upper-triangular matrix, it is still diagonalisable. Moreover, as all the entries are on or above the diagonal, we can diagonalise it by conjugating with an upper-triangular
matrix.  
\end{example}

\begin{remark}
 Even for $\Hs = \Bs_m$, the Borel subgroup of $\SL_m$, we cannot require $\w+n'$ from Theorem~\ref{jordan} to be of the classical Jordan form under no additional assumption on $x$. Even for $\w = 0$, there is an infinite number of nilpotent orbits of adjoint action of $\Bs_m$ on $\bb_m$ for $m\ge 6$; see \cite{DjoMal}.
 One can prove that if $x$ is a regular matrix, then we can actually find an $n'$ which is a nilpotent Jordan matrix.
\end{remark}

\subsubsection{Reductive groups}
Assume that $\Gs$ is a reductive group. Let $\Ts$ be its maximal torus, $\Bs$ a Borel subgroup containing $\Ts$, $\Bs^-$ the opposite Borel, $\U$ and $\U^-$ the respective unipotent subgroups. Let $\geg$, $\ttt$, $\bb$, $\bb^-$, $\uu$, $\uu^-$ be the corresponding Lie algebras.  Let $(e,f,h)$ be a principal $\ssl_2$-triple in $\geg$ such that $e\in \uu$, $f\in\uu^-$, $h\in \ttt$. Let
$$\Ss = e + C_\geg(f)$$
be the Kostant section.
 
\begin{lemma}\label{lembal}
 Under the assumptions above 
 $$ \Ad_{-}(-)\colon  \U^- \times \Ss \lra e+\bb^-$$
 is an isomorphism.
\end{lemma}

\begin{proof}
If $\Gs$ is semisimple, then the map
 $$ \Ad_{-}(-)\colon  \U^- \times \Ss \lra e+\bb^-$$
is an isomorphism (see \cite[Theorem 1.2]{Kost}; see also another proof in \cite[Theorem 7.5]{Ginz}).

Now if $\Gs$ is an arbitrary reductive group, let $\Gs^{\ad}$ be its adjoint group, and let $\pi\colon \Gs\to \Gs^{\ad}$ be the quotient map. From \cite[Proposition 17.20]{Milne} we have that $\pi(\Bs)$ and $\pi(\U^-)$ are Borel and maximal unipotent in $\Gs^{\ad}$, respectively. Note that $\ker \pi = Z(\Gs)$ and the connected component of $Z(\Gs)$ is a torus (\textit{cf.} \cite[Proposition 19.12]{Milne}). As a torus contains no non-trivial unipotent elements, we have $\ker \pi\cap U^- = \{1\}$. Therefore, $\pi|_{\U^-}$ is an isomorphism $\U^-\cong \pi(\U^-)$. We then know from the above that
$$\Ad_{-}(-)\colon  \pi(\U^-) \times \Ss_{\Gs^{\ad}} \lra e+\bb^-_{\Gs^{\ad}}$$
is an isomorphism. From \cite[Theorem II.11]{Jac} we can identify $\geg^{\ad}$ with an ideal inside $\geg$ such that $\geg = Z(\geg) \oplus \geg^{\ad}$. Then we have 
$$\pi(\U^-) \times \Ss_{\Gs} \cong (\pi(\U^-) \times \Ss_{G^{\ad}}) \times Z(\geg)$$
and 
$$\bb^-_{\Gs} = \bb^-_{\Gs^{\ad}}\times Z(\geg).$$
As the adjoint representation is trivial on the center of a Lie algebra, we have the following diagram, where the middle column is the product of the left and right and the horizontal arrows are the projections: 
$$
\begin{tikzcd}
\pi(\U^-) \times \Ss_{\Gs^{\ad}} \arrow[dd, "\Ad_{-}(-)"]
& \arrow[l, twoheadrightarrow] \U^- \times \Ss_{\Gs} \arrow[dd, "\Ad_{-}(-)"] \arrow[r, twoheadrightarrow] &
Z(\geg) \arrow[dd, "="] 
\\ \\
e+\bb^-_{\Gs^{\ad}} &
\arrow[l, twoheadrightarrow] e+\bb^-_{\Gs} \arrow[r, twoheadrightarrow] 
&
Z(\geg)\rlap{.}
\end{tikzcd}
$$
As the outer vertical arrows are isomorphisms, we get that  for $\Gs$ the map
$$ \Ad_{-}(-)\colon  \U^- \times \Ss \lra e+\bb^-$$
is also an isomorphism.
\end{proof}
\noindent
Let us now consider the preimage of $e+\ttt$, and for any $\w\in\ttt$ denote by $A(\w)\in \U^-$, $\chi(\w)\in \Ss$ the elements such that
 \begin{equation} \label{adchi0}
 \Ad_{A(\w)}(e+\w) = \chi(\w).
 \end{equation}
 Note that we have two inclusions of affine spaces $\Ss\into \geg$ and $e+\ttt \into \geg$. The former  induces the isomorphism $\Ss\cong \geg/\!\!/\Gs$, \textit{i.e.} $\C[\geg]^\Gs\xrightarrow{\cong} \C[\Ss]$ (by \cite[Section 4.7, Theorem 7]{Kostsec}). The latter induces a map $\C[\geg]^\Gs\to \C[e+\ttt]$. However, a regular element $\w\in\ttt$ is conjugate to $\w+e$ (see Corollary~\ref{corre}). Let us then consider the composition $\C[\geg]^\Gs\to \C[e+\ttt]\to \C[\ttt]$, where the last map comes from translation by $e$. It is equal to the map $\C[\geg]^\Gs\to \C[\ttt]$ coming from the inclusion $\ttt\to\geg$ -- as the dual maps of schemes agree on a dense subset of $\ttt$. 
 
 Note that if we compose $\chi^*\colon  \C[\Ss]\to \C[\ttt]$ with the isomorphism $\C[\geg]^\Gs\to \C[\Ss]$ described above, then we get the composite map above $\C[\geg]^\Gs\to \C[\ttt]$, which  we now know is induced by the inclusion $\ttt\to\geg$. By Chevalley's restriction theorem (\textit{cf.} \cite[Theorem 3.1.38]{ChGi}), this map is an inclusion whose image is $\C[\ttt]^\W$.\footnote{\label{footchev}Chevalley's theorem is originally formulated for semisimple groups. However, if we again consider $\geg^{\ad}$ as an ideal of $\geg$ such that $\geg = \geg^{\ad}\oplus Z(\geg)$, we have
\[ \C[\geg]^\Gs = \C[\geg^{\ad}\oplus Z(\geg)]^{\Gs^{\ad}} = \C[\geg^{\ad}]^{\Gs^{\ad}} \oplus Z(\geg) = \C[\ttt\cap\geg^{\ad}]^\Ws \oplus Z(\geg) = \C[\ttt]^\Ws,\]
 where the third equality follows from Chevalley's original theorem for $\Gs^{\ad}$.}
 Therefore, we get the following. 

 \begin{proposition}\label{isoquot}
 \label{kostiso}
  The map $\chi\colon \ttt\to\Ss$ defined by the property \eqref{adchi0} induces an isomorphism $\ttt/\!\!/\W\to \Ss$.
 \end{proposition}

\subsubsection{Principally paired groups}\label{kostgensec}
Now let $\Hs$ be any principally paired group. Let $\Ns$ be the unipotent radical of $\Hs$. Then $\Ns$ is a normal subgroup of $\Hs$, and $\Hs/\Ns$ is reductive. Let $\Ls\subset \Hs$ be any \emph{Levi subgroup}, \textit{i.e.} a section of $\Hs\to \Hs/\Ns$. By Mostow's Levi decomposition (see \cite{Mostow}), we can take for $\Ls$ any maximal reductive subgroup of $\Hs$. We have $\Hs = \Ns\rtimes \Ls$ and hence $\he = \nen \oplus \lel$, where $\he$, $\nen$, $\lel$ are the Lie algebras of $\Hs$, $\Ns$, $\Ls$, respectively. Let $r$ be the dimension of a maximal torus.

Assume that $(e,h)$ is an integrable principal $\bb(\ssl_2)$-pair within $\he$, and let $\{H^t\}$ be the embedding of $\Cs$ to which $h$ integrates. We can choose $\Ls$ such that $h\in\lel$; hence we will assume this from now on. We then have $e = e_n + e_l$, where $e_n\in\nen$, $e_l\in\lel$. Let us consider, by the Jacobson--Morozov theorem (\textit{cf.} Section~\ref{sl2p}), the $\ssl_2$-triple $(e_l,f_l,h_l)$ within $\lel$. 

\begin{lemma}\label{redpartreg}
 For $\Hs$ and $(e,h)$ as above, $e_l$ is a regular element of\, $\lel$.
\end{lemma}

\begin{proof}
We know that $e$ is a regular element of $\he$. This means that $[e,\he]$ is of codimension $r$ in $\he$. But note that $[e,\he] \subset \nen \oplus [e_l,\lel]$ as $\nen$ is an ideal. Therefore, $[e_l,\lel]$ is of codimension at most $r$ in $\lel$. Therefore, $\dim C_\lel(e_l) \le r$; hence actually $\dim C_\lel(e_l) = r$, and $e_l$ is regular in $\lel$.
\end{proof}

Now, let $\Bs_l$ be a Borel subgroup of $\Ls$ whose Lie algebra contains $e_l$ and $h$, and inside it let $\Ts$ be a torus whose Lie algebra contains $h$. In fact, $\Bs_l$ is defined uniquely by those properties; see \cite[Proposition 3.2.14]{ChGi}.  Let $\Bs = \Ns\rtimes \Bs_l$ -- it is easy to see that $\Bs$ is then a Borel subgroup of $\Hs$. Let $\U$ be its subgroup of unipotent elements. Given $\Bs_l$ and $\Ts$, let $\Bs_l^-$ be the opposite Borel subgroup of $\Ls$ and $\U_l$, $\U_l^-$ the unipotent subgroups of $\Bs_l$ and $\Bs_l^-$. By $\bb$, $\bb_l$, $\ttt$, $\bb_l^-$, $\uu$, $\uu_l$, $\uu_l^-$, we denote the corresponding Lie algebras. Let $\Ws$ be the Weyl group of $\Hs$ (equal to the Weyl group of $\Ls$).

\begin{lemma}\label{posint}
 All of the weights of the $\{H^t\}$-action on $\uu$ are positive even integers.
\end{lemma}

\begin{proof}
As $e$ is regular in $\Hs$, it has to be regular in $\Bs$ as well. Therefore, $[e,\bb] = \uu$ (\textit{cf.} Section~\ref{secreg}).

We can choose a basis of $\bb$ which consists of eigenvectors of $[h,-]$. We then choose from it a subset $\{v_1,v_2,\dots,v_k\}$ such that $\{[e,v_i]\}_{i=1}^k$ forms a basis of $\uu$. Then $[e,-]$ is an isomorphism $\Span(v_1,\dots,v_k) \to \uu$. Let $\phi$ denote this restricted commutator operator $[e,-]$. For any $v\in\bb$ we have
$$[h,[e,v]] = [[h,e],v] + [e,[h,v]] = 2[e,v] + [e,[h,v]];$$
hence if $[h,v] = \lambda v$, we get $[h,[e,v]] = (\lambda+2)[e,v]$. Therefore, for an $h$-weight vector $v$, $\phi$ satisfies the condition
$$[h,v] = \lambda v \iff [h,\phi(v)] = (\lambda+2)\phi(v).$$
Let us consider a weight vector $w\in \uu$ such that $[h,w] = \lambda w$ and assume that $\lambda$ is not a positive even integer. We now know that $w = \phi(w_1)$ for some $w_1\in\bb$ with $[h,w_1] = (\lambda-2) w_1$. As $\lambda-2 \neq 0$, we have $w_1\in \uu$ (as $\ttt$ has only zero weights of $H^t$-action). Then analogously $w_1 = \phi(w_2)$ for $w_2\in\bb$ of weight $\lambda-4$. As again $\lambda-4\neq 0$, we get $w_2 = \phi(w_3)$, and we continue this procedure to get an infinite sequence $w=w_0$, $w_1$, $w_2, \dots$ such that $w_i$ is a weight vector of weight $w_i - 2i$. However, $\bb$ is finite-dimensional, so we get a contradiction.
\end{proof}

\noindent
For our principally paired $\Hs$, the role of the Kostant section will be played by 
\begin{align}\label{ppS} \Ss := e + C_{\lel}(f_l)\subset \he.\end{align}
Note that in case $\Hs$ is solvable, this is the same as what we consider in Section~\ref{secsolvkos}; \textit{i.e.} $\Ss = e +\ttt$, where $\ttt = \lel$ is the Lie algebra of a maximal torus.

\begin{lemma}
\label{congen}
The conjugation map
$$
\Ad_{-}(-)\colon  \U_l^- \times \Ss \lra e + \bb_l^-
$$
is an isomorphism.
\end{lemma}

\begin{proof}
By Lemma~\ref{redpartreg} we know that the conjugation map
\begin{equation}
\label{ident}
\Ad_{-}(-)\colon  \U_l^- \times (e_l + C_\lel(f_l)) \lra e_l + \bb_l^-
\end{equation}
is an isomorphism. But note that the weights of the $\Ts$-action on $\uu_l^-$ are exactly the negatives of the weights on $\uu_l$. Hence by Lemma~\ref{posint}, evaluated on $h$ they are all negative even integers. As $\nen$ is an ideal in $\he$, we have
$$[\uu_l^-,e_n] \subset \nen.$$
However, we know (again from Lemma~\ref{posint}) that the $h$-weight of $e_n$ (equal to 2) is the lowest possible among the weights in $\nen$. All the $h$-weights in $[\uu_l^-,e_n]$ would be lower. Therefore, in fact $[\uu_l^-,e_n] = 0$. Hence $\U_l^-$ commutes with $e_n$.

Then we get the conclusion simply by adding $e_n$ to both sides of \eqref{ident}.
\end{proof}
Now note that we are given two one-parameter subgroups: $H^t$ and $H_l^t$, generated by $h$ and $h_l$, respectively. We show that they actually only differ by a center of $\Ls$.

\begin{lemma}\label{semcen}
 Let $\Gs$ be a reductive group and $e$ a regular nilpotent element in $\geg = \Lie(\Gs)$. Then the only semisimple elements in its centraliser $C_\geg(e)$ are the ones in the center $Z(\geg)$.
\end{lemma}

\begin{proof}
 Assume that $v\in\geg$ is a semisimple element such that $[v,e] = 0$. Choose a Borel subgroup $\Bs\subset \Gs$ whose Borel subalgebra $\bb\subset \geg$ contains $e$ and $v$, and let $\nen$ be the nilpotent part of $\bb$. We can choose a maximal torus $\Ts$ within $\Bs$ whose Lie algebra contains $v$. Let $r = \dim \Ts = \dim C_{\geg}(e)$.
 
 As $e$ is regular in $\Hs$, it is also regular in $\bb$ and $\nen = [\bb, e]$ (\textit{cf.} Section~\ref{secreg}). However, $\bb = \ttt \oplus \nen$, so by iterating we easily see that $\bb$ is generated by $e$ and $\ttt$. Then as $[v,e]=0$, this easily leads to $[v,\bb] = 0$. As $\bb$ was a Borel subalgebra and $v$ is semisimple, from this $[v,\geg] = 0$ follows.
\end{proof}
From this lemma, as $[h,e] = [h_l,e] = 2e$, we infer $h-h_l\in Z(\lel)$.
In the map $\Ad_{-}(-)$ from Lemma~\ref{congen}, let us consider the preimage of $e+\ttt$, and for any $w\in\ttt$ denote by $A(\w)\in \U^-$, $\chi(\w)\in \Ss$ the elements such that
 \begin{equation}
 \Ad_{A(\w)}(e+\w) = \chi(\w).
 \end{equation}

We will now want to generalise Kostant's Theorem~\ref{kostant2}. First, we find the contracting $\Cs$-action on $\Ss$ from \eqref{ppS}. Note that as $e_l$ is regular in $\Ls$,  $f_l$ is also regular in $\Ls$ (see Remark~\ref{princesl2}). Moreover, as all of the weights of the ${H^t}$-action on $\uu_l$ are positive integers, on $\uu_l^-$ they are all negative integers. As the weight of the action on $f_l$ is $-2$ (note that we use Lemma~\ref{semcen} to switch between the actions of $h_l$ and $h$), $f_l$ must lie in $\uu_l^-$. In particular, $f_l \in \bb_l^-$, and as $\bb_l^-$ contains the Lie algebra of the maximal torus of $\Ls$, we have that $f_l$ is regular in $\bb_l^-$. This means that $C_{\he}(f_l)\subset \bb_l^-$ (\textit{cf.} Section~\ref{secreg}). In particular, all of the weights of the ${H^t}$-action on $C_{\he}(f_l)$ are non-positive integers. Therefore, for any $x\in C_{\he}(f_l)$ we have
$$\Ad_{H^t}(x+e) = \Ad_{H^t}(x) + t^2 e = t^2 \left( \Ad_{H^t}(x)/t^2 + e\right)$$
and
$$\lim_{t\to\infty} \Ad_{H^t}(x)/t^2 = 0.$$
Therefore, if we define the action of $\Cs$ on $\Hs$ by
$$t \cdot v = t^{-2} \Ad_{H^t}(v),$$
then it preserves $\Ss$ and for any $v\in \Ss$ we have 
$$\lim_{t\to\infty} t\cdot v = e.$$

\begin{theorem}\label{kostarb}
 Every element of $\Ss$ is regular in $\he$. Moreover, every regular orbit of adjoint action of\, $\Hs$ on $\he$ meets~$\Ss$.
\end{theorem}

\begin{proof}
 For the first part, we proceed as in the proof of Lemma~\ref{etreg}. Assuming that for some $x\in C_{\he}(f_l)$ the element $x+e$ is not regular, we get that $\Ad_{H^t}(x)/t^2+e$ is not regular for any $t$, and from continuity ($t\to \infty$) we get that $e$ is not regular.
 
 Now assume that some $y\in \he$ is regular. It lies in a Borel subalgebra, and by \cite[Section~16.4]{HumphAlg} all Borel subalgebras are conjugate; hence we can assume $y\in \bb$. As $\Bs$ contains a maximal torus of $\Hs$, we have that $y$ is regular in $\bb$ as well. Therefore, by Lemma~\ref{etkost} it is conjugate to an element of the form $e + v$ for $v\in\ttt$. It is then conjugate to $\chi(v) \in \Ss$.
\end{proof}

To finish the proof of $\C[\he]^\Hs = \C[\Ss]$, we need to state the following lemma, already known for reductive groups.

\begin{lemma}\label{genrest}
We have $\C[\he]^\Hs = \C[\lel]^\Ls = \C[\ttt]^\Ws$.
\end{lemma}

\begin{proof}
 The second equality is just Chevalley's restriction theorem (\textit{cf.} \cite[Theorem 3.1.38]{ChGi} and footnote~\ref{footchev}). We need to prove that the restriction map $\C[\he]^\Hs\to \C[\lel]^\Ls$ is an isomorphism.
 
 Let us first prove that it is surjective. We have the projection map $\pi\colon \Hs\to \Ls$, and then we can use it to pull back any $\Ls$-invariant function on $\lel$. If $f$ is such function, its pullback is $f\circ \pi_*$, and for any $g\in \Hs$ and $v\in\he$, we have
 $$(f\circ\pi_*)(\Ad_g(v)) = f(\Ad_{\pi(g)}(\pi_*(v))) = f(\pi_*(v)) = (f\circ\pi_*)(v);$$
 hence $f\circ \pi_*$ is $\Hs$-invariant (and obviously restricts to $f$ on $\lel$).
 
 Now we prove the injectivity. As every element of $\Hs$ is contained in a Lie algebra of a Borel subgroup, and they are all conjugate (\textit{cf.} \cite[Theorem 11.1]{Borel}), a function from $\C[\he]^\Hs$ is fully determined by its values on $\bb$. We know that $\bb = \ttt \oplus \uu$ and the weights of the $\{H^t\}$-action on $\ttt$ are all 0, and on $\uu$ they are all positive. 
 
 Therefore, any polynomial on $\bb$ which is invariant under this action can only contain the $\ttt$-variables. Hence it is uniquely determined by its values on $\ttt$.
 \end{proof}
 
 From the proof of Lemma~\ref{congen} and from Proposition~\ref{isoquot}, the map $\chi$ defines an isomorphism $\C[\Ss]\to \C[\ttt]^\Ws$, and when composed with the restriction from $\C[\he]^\Hs$, it clearly gives the restriction $\C[\he]^\Hs \to \C[\ttt]^\Ws$ (note that $x$ and $\chi(x)$ are always conjugate). Then from Lemma~\ref{genrest} we get the following. 
 
 \begin{theorem}\label{restkos}
  The restriction map $\C[\he]^\Hs\to \C[\Ss]$ is an isomorphism.
 \end{theorem}

 In particular, this means that no elements of $\Ss$ are conjugate to each other. Together with Theorem~\ref{kostarb} this gives the following. 

 \begin{corollary}
  Every regular orbit of adjoint action of\, $\Hs$ on $\he$ meets $\Ss$ exactly once.
 \end{corollary}

\subsection{Regular actions and fixed point sets}
\label{secregact}

\begin{definition}\label{defreg}
Assume we are given a principally paired $\Hs$ with $(e,h)$  the integrable principal $\bb(\ssl_2)$-pair in $\he$. If $\Hs$ acts on a smooth projective variety $X$, we say that it acts \emph{regularly} if $e$ has a unique zero $o\in X$. 
\end{definition}

\begin{remark}
The choice of integrable principal pair $(e,h)$ in $\he$ is not unique. However, we will see below in  Lemma~\ref{isoreg} that the property of the action being regular does not depend on the choice. 

Note that as $e$ is nilpotent, it generates an additive subgroup of $\Hs$ (by \cite[Proposition 1.10, Theorem~4.4 and Section~II.7.3]{Borel}), and hence by \cite[Theorem 4.1]{Horrocks} the zero scheme $X^e$ of $V_e$ is connected. It is therefore enough to assume that the fixed points of $e$ are isolated. We will in fact prove in Lemma~\ref{isoreg} that all of the regular elements of $\he$ have isolated fixed points on $X$.
\end{remark}

\begin{example}\label{exgr2} This example is from  the PhD thesis of Ersan Akyildiz \cite{Akyphd}; see also \cite{Akyildiz}.
Consider a complex reductive group $\Gs$, with the choice of $e$ as in Example~\ref{exregnilp}. By the discussion in Section~\ref{sl2p}, there exists an $h\in\geg$ which makes $\Gs$ principally paired. Let $X= \Gs/\Bs$ be the full flag variety of $\Gs$. Then for any $x=g\Bs\in X$, by Lemma~\ref{lemad}
 $$V_e|_{x} = \D g(V_{\Ad_{g^-1}(e)}|_{[1]}).$$
Therefore, $V_e$ vanishes at $x$ if and only if $\Ad_{g^-1}(e)$ vanishes at $[1] = \Bs$. This means that $\Ad_{g^-1}(e) \in \bb = \Lie(\Bs)$, or in other words, $e\in\Lie(g\Bs g^{-1})$. The subgroup $g\Bs g^{-1}$ is of course a Borel subgroup of $\Gs$. By Section~\ref{sl2p} the group $\Bs$ is the unique Borel subgroup of $\Gs$ whose Lie algebra contains $e$. Therefore, $e\in\Lie(g\Bs g^{-1})$ only if $g\Bs g^{-1} = \Bs$. By \cite[Theorem~11.16]{Borel} this is true only for $g \in \Bs$; \textit{i.e.} $x = [1]$. Therefore, $\Gs$ acts regularly on the full flag variety $\Gs/\Bs$.
 
 Hence $G$ also acts regularly on all the partial flag varieties $\Gs/\Ps$. Indeed, assume that $x\in \Gs/\Ps$ is fixed by~$e$. If we denote by $\pi_\Ps$ the projection $\pi_\Ps\colon \Gs/\Bs\to \Gs/\Ps$, then $\pi_\Ps^{-1}(x)$ is a closed subvariety of $\Gs/\Bs$, closed under the action of $\Gs_a$ generated by $e$. Hence by the Borel fixed point theorem (see \cite[Corollary 17.3]{Milne}), it contains a fixed point of $\Gs_a$, which is unique. Therefore, $x$ is its image.
\end{example}

\begin{example}[see \protect{\cite[Section 6]{BC}}]\label{exsl2}
 Let $\Hs = \SL_2(\C)$, and consider the irreducible representation $V$ of $\SL_2(\C)$ of dimension $n+1$. In particular, the regular nilpotent
 $$e = \begin{pmatrix} 0 & 1 \\ 0 & 0\end{pmatrix}$$
   acts on $V$ with  matrix
 $$\begin{pmatrix}
       0 & 1 & 0 & 0 & \dots & 0\\
       0 & 0 & 1 & 0 & \dots & 0\\
       0 & 0 & 0 & 1 & \dots & 0\\
       \vdots & \vdots & \vdots & \vdots & \ddots & \vdots \\
       0 & 0 & 0 & 0 & \dots & 1 \\
       0 & 0 & 0 & 0 & \dots & 0
      \end{pmatrix}.
 $$
 If we consider $X = \PP(V)$, the action is clearly regular and the only fixed point of $e$ corresponds to the vector of highest weight in $V$.
\end{example}

\subsubsection{Solvable groups}

\begin{lemma}\label{regzer}
Let $\Hs$ be a solvable group. Let $\Ts$ be its maximal torus and $\he_n$ be the nilpotent part of\, $\he = \Lie(\Hs)$. Assume that $e\in\he_n, h\in\ttt$ are such that $(e,h)$ is an integrable $\bb(\ssl_2)$-pair and that $\Hs$ acts regularly on a smooth projective variety $X$. Then any element of $e+\ttt$ has isolated zeros on $X$.
\end{lemma}

\begin{proof}
 We will denote by $\{H^t\}_{t\in\Cs}$ the one-parameter subgroup to which $h$ integrates.
 Define $\ZZ \in \ttt \times X$ as the zero scheme of the total vector field restricted to $e+\ttt \cong \ttt$. In other words, for any $\w\in\ttt$ that vector field restricted to $\{\w\} \times X$ equals $V_{e+\w}$ (\textit{cf.} Definition~\ref{defz}).
Also consider  an action of $\Cs$ on $\ttt \times X$ which is defined on $\ttt$ by multiplication by $t^{-2}$  and on $X$ by the action of $H^t$. By Lemma~\ref{lemad} this action preserves $\ZZ$ as $\Ad_{H^t}(e) = t^2 e$.
 
Consider the map $\pi\colon \ZZ\to \ttt$ defined as the projection onto the first factor of $\ttt \times X$. As it is a morphism of schemes locally of finite type, by Chevalley’s semicontinuity theorem, \textit{cf.} \cite[Th\'eor\`eme~13.1.3]{EGA43}, the set 
 $$D = \{ (\w,x)\in \ZZ : \dim \pi_{\w} \ge 1 \}$$
is closed. Here
$$\pi_\w:=\pi^{-1}(\w)\subset \ZZ$$
denotes the fiber.  Suppose $D$ is non-empty. Hence we have some $\w\in\ttt$ such that $\dim \{x\in\ZZ : (\w+e)|_{x} = 0\} \ge 1$. Note that for any $t\in\Cs$ we have
$$t^2\w + e = t^2(\w + t^{-2}e) = t^2 \Ad_{H^t}^{-1}(\w + e).$$
Therefore, the zero set of $t^2\w + e$ is the same as the zero set of $\Ad_{H^t}^{-1}(\w + e)$, which by Lemma~\ref{lemad} is isomorphic -- via the action of $H^t$ -- to the zero set of $\w + e$. Hence for each $t\neq 0$ we have $(t\w,o)\in D$, where $o\in X$ is the unique fixed point of $e$. Because $D$ is closed, we get $(0,o)\in D$. Hence $\dim \pi_0 \ge 1$, which is impossible as $\pi_0 = \{(0,o)\}$ by our regularity assumption.
\end{proof}

\begin{theorem}\label{isolsolv}
 Assume $\Hs$ and $X$ are as in Lemma~\ref{regzer} and $e$ is a principal nilpotent. Then any regular element of\, $\he$ has isolated zeros on $X$.
\end{theorem}

\begin{proof}
This now follows directly from Lemmas~\ref{regzer} and~\ref{etkost}. 
\end{proof}

In particular, regular semisimple elements have isolated zeros on $X$. Therefore, we get the following. 

\begin{corollary}\label{fintor}
There are finitely many $\Ts$-fixed points on $X$.
\end{corollary}

\subsubsection{General principally paired groups}
With the use of the results of Section~\ref{kostgensec}, we can also provide a version of Theorem~\ref{isolsolv} for arbitrary principally paired groups.

\begin{lemma}\label{isoreg}
Let a principally paired group $\Hs$ act regularly on a smooth projective variety $X$. Then all the regular elements of\, $\he$ have isolated zeros on $X$.
\end{lemma}

\begin{proof}
 We know from Lemma~\ref{kostarb} that every regular element of $\he$ is conjugate to an element of $\Ss$ from \eqref{ppS}. Therefore, it is enough to prove the statement for the elements of $\Ss$. The argument is the same as in the proof of Lemma~\ref{regzer}, using the contracting action from Section~\ref{kostgensec}. Note that if $p\in X$ is a zero of $x+e$, then $H^t p$ is a zero of $\Ad_{H^t}(x)/t^2 + e$. Therefore, if $(x+e,p) \in D$, we have $(\Ad_{H^t}(x)/t^2 + e,H^tp)\in D$ for any $t\in\Cs$ and then $\left(e,\lim_{t\to \infty} H^t p\right) \in D$.
\end{proof}

\section{Main theorem for solvable groups}
\label{solvsec}

We first consider a solvable group $\Hs$ acting on a variety $X$. We will prove that if the action is regular, then for a maximal torus $\Ts\subset \Hs$, we can find $\Spec H_\Ts^*(X)$ as a particular subscheme of $\ttt\times X$. This generalises the result of \cite{BC} for the Borel subgroup of $\SL_2(\C)$. The goal of this section is to find necessary assumptions on $\Hs$ and construct the scheme $\ZZ = \Spec H_\Ts^*(X)$ inside $\ttt\times X$.

\subsection{Principally paired solvable groups}
Assume that $\Hs$ is a principally paired solvable group and $(e,h)$ the principal integrable $\bb(\ssl_2)$-pair within~$\Hs$. By $\{H^t\}_{t\in\Cs}$ we denote the one-parameter subgroup to which $h$ integrates. Let $\Ts\subset \Hs$ be the maximal torus which contains it. From Theorem~\ref{solv} we have $\Hs = \Ts \ltimes \Hs_u$, where $\Hs_u\subset \Hs$ is the subgroup of unipotent elements. We denote by $r$ the dimension of $\Ts$ (or $\ttt$), equal to the rank of $\Hs$. The torus $\Ts$ acts on the Lie algebra $\he$ by the adjoint action $\Ad$. It splits into two representations $\he = \ttt\oplus \he_n$, where $\he_n = \Lie(\Hs_u)$. The first one is trivial, and the weights of the other, $\alpha_1, \alpha_2, \dots, \alpha_k \in\ttt^*$, will be called the \emph{roots} of $\Hs$.
This means that if $v_1$, $v_2$, \dots, $v_k$ are the root vectors, then for any map $\phi\colon \Cs\to\Ts$ we have
$$\Ad_{\phi(t)} (v_i) = t^{\alpha_i(\D \phi|_1(1))} v_i.$$
We denote by $\ttt^\reg = \ttt\cap\he^\reg$ the subset of $\ttt$ consisting of regular elements. As any element of $\ttt$ commutes with the whole $\ttt$, the condition of $v\in\ttt$ being regular means $C_\he(v) = \ttt$. This means that $[v,-]$ does not have zeros on $\he_n$; \textit{i.e.} $\alpha_1(v)$, $\alpha_2(v)$, \dots, $\alpha_k(v)$ are all non-zero. Hence we see that the elements of $\ttt^\reg$ are those in $\ttt$ that are not annihilated by any root of $\Hs$. As $h\in\ttt$ is regular, all of the roots are non-zero on $h$ -- by Lemma~\ref{posint} they are even positive integers when evaluated on $h$ -- hence non-zero. Therefore, $\ttt^\reg$ is a non-empty open subset of $\ttt$, and its complement is a union of hyperplanes.

In our applications $\Hs$ will mostly be the Borel subgroup of some principally paired algebraic group $\Gs$. Let us give an example below.

\begin{example}\label{exgr}
 A simple case of the above is $\Hs = \Bs_m:=\Bs(\SL_m)$, the Borel subgroup of $\SL_m$ consisting of upper-triangular matrices. Let $\bb_m$ be its Lie algebra. We have the torus $\Ts\subset \Bs_m$ consisting of the diagonal matrices of determinant $1$ and its Lie algebra $\ttt\subset\bb_m$ consisting of the traceless diagonal matrices.
   
 We can identify $\ttt$ with $\C^{m-1}$ via the isomorphism
 $$(v_1,v_2,\dots,v_{m-1}) \longmapsto 
  \diag(0,v_1,v_2,\dots,v_{m-1}) - \frac{v_1+v_2+\dots+v_{m-1}}{m} I_m; 
 $$
  \textit{i.e.} $(v_1,v_2,\dots,v_{m-1})$ corresponds to the unique matrix $A$ in $\ttt$ with $a_{ii}-a_{11} = v_{i-1}$ for $i=1,2,\dots,m-1$.
 Then we can take \textit{e.g.}
 $$e = \begin{pmatrix}
       0 & 1 & 0 & 0 & \dots & 0\\
       0 & 0 & 1 & 0 & \dots & 0\\
       0 & 0 & 0 & 1 & \dots & 0\\
       \vdots & \vdots & \vdots & \vdots & \ddots & \vdots \\
       0 & 0 & 0 & 0 & \dots & 1 \\
       0 & 0 & 0 & 0 & \dots & 0
      \end{pmatrix}\in\geg
 $$
 and 
 $$
 h = \begin{pmatrix}
       m-1 & 0 & 0 & \dots & 0\\
       0 & m-3 & 0 & \dots & 0\\
       0 & 0 & m-5 & \dots & 0\\
       \vdots & \vdots & \vdots & \ddots & \vdots \\
       0 & 0 & 0 & \dots & 1-m
      \end{pmatrix},
 $$
 or equivalently $h = (-2,-4,\dots,2-2m)\in \C^{m-1}$. Then
 $$
 H^t = \diag(t^{m-1}, t^{m-3},t^{m-5},\dots, t^{3-m}, t^{1-m}).
 $$
 The regular elements of $\ttt$ are the diagonal traceless matrices with pairwise distinct diagonal entries.
 
 We can generalise this example by taking $\Hs$ to be a Borel subgroup of any reductive group $\Gs$. This choice defines the choice of positive roots (as those whose root vectors lie in $\he$). We can therefore take $e = x_1 + x_2 + \dots + x_s$, where $x_1$, $x_2$, \dots, $x_s$ are the root vectors of $\geg$ corresponding to the positive simple roots ($s=r - \dim Z(\Gs)$). Then $e$ is a regular nilpotent in $\Gs$ and $\Hs$ (see Example~\ref{exregnilp}). From the discussion in Section~\ref{sl2p}, we see that there exists an $h$ that satisfies the conditions.
\end{example}

\subsection{Uniform diagonalisations}
We saw in Corollary~\ref{corre} that $e+\w$ is always conjugate to $\w$ if $\w\in\ttt^\reg$. In the first case in Example~\ref{exjor}, we have a closed formula for the conjugating matrix. We generalise this observation here.

\begin{theorem}\label{unif}
There exists a morphism $M\colon \ttt^\reg\to \Hs$, denoted by $\w\mapsto M_\w$,  that satisfies the equality
$$\Ad_{M_\w}(\w) = e+\w$$
for any $\w\in\ttt^\reg$.
\end{theorem}

\begin{proof}
  From Corollary~\ref{corre} we know that for each $\w\in\ttt^\reg$ and $n\in\he_n$, there exists an $A\in\Hs$ such that
  \begin{align}\label{conjugate}\Ad_{A}(\w) = n+\w.\end{align}
  We have to prove that for $n = e$ we can choose such matrices in a way that varies regularly when $\w$ varies.
 
We know by Theorem~\ref{solv} that there exists a $V\in\Ts$ such that $AV\in \Hs_u$. Any element of $\Ts$ clearly centralises $\w$. Therefore, $AV$ also satisfies $ \Ad_{AV}(\w) = n+\w.$  Hence we can assume that $A\in \Hs_u$. We first show that $A\in \Hs_u$ is unique with respect to \eqref{conjugate}. Indeed, assume to the contrary that $A$, $A'$ are both unipotent and $\Ad_{A}(\w) = \Ad_{A'}(\w) = n + \w$. Then
 $$\Ad_{A^{-1}A'}(\w) = \Ad_{A^{-1}}(n+\w) = \w.$$
 Thus $A^{-1}A'$ centralises $\w$. Hence it centralises $\Aa(\w)$, the smallest closed subgroup of $\Hs$ whose Lie algebra contains $\w$. The group $\Aa(\w)$ is contained in the torus $\Ts$; therefore, by \cite[Section~19.4]{Humph} its centraliser $C_{\Hs}(\Aa(\w))$ is connected. But $\Lie(C_{\Hs}(\Aa(\w)))$ has to commute with $\Lie(\Aa(\w))$, which contains $\w$. By the regularity assumption $C_\he(\w) = \ttt$; thus from the connectivity we get $C_{\Hs}(\Aa(\w))\subset \Ts$. Therefore, $A^{-1}A'\in\Ts$, but as $A^{-1}A'$ is unipotent, we get $A^{-1}A' = 1$, hence $A = A'$.

Now consider the map
$$\phi\colon  \Hs_u \times \ttt^\reg \lra \he_n\oplus\ttt^\reg$$
$$\phi(A, \w) = \Ad_{A}(\w).$$
We have just proved that $\phi$ is a bijection. Now by Grothendieck's version of Zariski's main theorem (\textit{cf.} \cite[Th\'eor\`eme 4.4.3]{EGA31}), it can be factored as $\phi = \tilde{\phi}\circ\iota$, where $\iota\colon \Hs_u\times \ttt^\reg \to Y$ is an open embedding and $Y\to\tilde{\phi}$ is finite. By restricting $Y$ to the closure of $\im\iota$, we can assume that $\im\iota$ is dense in $Y$. The map $\phi$ is clearly dominant, and its source is irreducible; hence by \cite[Proposition 7.16]{Harr} it is birational. Therefore, $\tilde{\phi}$ is birational as well, but it is finite and its target is normal; hence $\tilde{\phi}$ is an isomorphism. Therefore, $\phi$ is an open embedding, which has to be an isomorphism as it is surjective.

Hence we get the desired map $M\colon \ttt^\reg\to \Hs_u$ by considering the first coordinate of $\phi^{-1}|_{\{e\}\times\ttt^\reg}$.
\end{proof}

\subsection{Regular actions}
From now on we will assume that our principally paired solvable group $\Hs$ acts regularly on a smooth projective variety $X$  (see Definition~\ref{defreg}). By Lemma~\ref{lemzer} the unique zero $o\in X$ of $e$ is a zero of the whole $\he$.

\begin{example}
In Example~\ref{exgr2} we see regular actions of a reductive group $\Gs$ on flag varieties. In Example~\ref{exsl2} we constructed a regular action of $\SL_2$ on $\PP^n$. In both cases, when we restrict to a Borel subgroup, we get a solvable principally paired group (see Example~\ref{exgr}) acting regularly on smooth projective varieties.
\end{example}

By Corollary~\ref{fintor} there are finitely many fixed points of the torus $\Ts$ acting on $X$. We will call them $\zeta_0 = o$, $\zeta_1$, $\dots$, $\zeta_s$. Moreover, combining Lemma~\ref{regzer} with Lemma~\ref{lemfix}, we get that for any $\w\in\ttt^\reg$ the only zeros of $V_\w$ on $X$ are $\zeta_0$, $\zeta_1$, \dots, $\zeta_s$.

Now, following the idea of \cite{BC}, we define the scheme whose coordinate ring will turn out to be the $\Hs$-equivariant cohomology of $X$. As $\Hs$ is homotopically equivalent to its maximal torus $\Ts$, this is the same as the $\Ts$-equivariant cohomology.

\begin{definition}\label{defz}
Let $\ZZ \subset \ttt\times X$ be defined as the zero scheme of the total vector field (see Definition~\ref{totvec}) restricted to $e+\ttt \cong \ttt$. We will denote that restricted vector field by $V_{e+\ttt}$. In other words, for any $\w\in\ttt$ the vector field $V_{e+\ttt}$ restricted to $\{\w\}\times X$ equals $V_{e+\w}$.
\end{definition}

We will also consider an action of $\Cs$ on $\ttt\times X$ which is defined on $\ttt$ by multiplication by $t^{-2}$ and on $X$ by the action of $H^t$. Clearly from Lemma~\ref{lemad} this action preserves $\ZZ$ as $\Ad_{H^t}(e) = t^2 e$. Our goal will be to prove the following theorem.

\begin{theorem}\label{finsolv}
Let $\Hs$ be a principally paired solvable group acting regularly on a smooth complex projective variety~$X$. Then there is a homomorphism
$$\rho\colon  H_\Ts^*(X)\lra \C[\ZZ],$$
to be defined in \eqref{rho}, which is an isomorphism of graded $\C[\ttt]$-algebras. Moreover, the zero scheme $\ZZ$ is affine, so that we have the following diagram with vertical isomorphisms:
$$ \begin{tikzcd}
 \ZZ  \arrow{d}{\pi} \arrow{r}{\rho^*}  &
  \Spec H_\Ts^*(X;\C) \arrow{d} \\
   \ttt \arrow{r}{\cong}&
  \Spec H^*_\Ts.
\end{tikzcd}$$
\end{theorem}

We will first study the structure of $\ZZ$ with connection to the torus-fixed points $\zeta_0$, \dots, $\zeta_s$. We will also prove that $\ZZ$ is reduced. This will allow us to define a map $\rho\colon  H^*_\Ts(X)\to \C[\ZZ]$ by specifying $\rho(c)$ by its values. To show that $\rho(c)$ is a regular function on $\ZZ$, we will prove that $H^*_\Ts(X)$ is generated by Chern classes of $\Hs$-equivariant vector bundles.

\subsection{Equivariant cohomology and Bia{\l}ynicki-Birula decomposition}
We know that the $\Ts$-equivariant cohomology $H_{\Ts}^*(\pt) = \C[\ttt]$ of the point is the ring of polynomials on $\ttt$. By $\I$ we will denote the ideal of polynomials vanishing at $0$, equivalently $\I = \bigoplus_{n>0} H_{\Ts}^n(\pt)$. The multiplicative group $\Cs$ acts on $X$ by the means of the morphism $H\colon \Cs\to \Hs$, $t\mapsto H^t$. 
This action has finitely many fixed points $\zeta_0$, $\zeta_1$, $\dots$, $\zeta_s$. We may then consider its Bia{\l}ynicki-Birula plus- and minus-decompositions (see \cite{BB1}), \textit{i.e.}
$$W_i^+ = \{x\in X: \lim_{t\to 0} H^t\cdot x = \zeta_i \}, \quad W_i^- = \{x\in X: \lim_{t\to \infty} H^t\cdot x = \zeta_i \}.$$
All these sets are locally closed varieties, isomorphic to affine spaces.

When such decompositions exist, the odd cohomology of $X$ vanishes; see \cite{BB2}. Then by Goresky--Kottwitz--MacPherson (\textit{cf.}
 \cite[Corollary 1.3.2]{GKM}), the $\Ts$-space $X$ is equivariantly formal. In particular,
\begin{align}\label{formal}H_{\Ts}^*(X) \cong H_{\Ts}^*(\pt)\otimes H^*(X)\end{align}
as $H_{\Ts}^*(\pt)$-modules and $H^*(X) \cong H_{\Ts}^*(X)/\I H_{\Ts}^*(X)$ as $\C$-algebras.

\begin{theorem}\label{ABBst}
The Bia{\l}ynicki-Birula plus-decomposition $X = \bigcup_{i=0}^s W_i^+$ is $\Hs$-stable.
\end{theorem}

\begin{proof}
Assume that $x\in W_i^+$, \textit{i.e.} $\lim_{t\to 0} H^t\cdot x = \zeta_i$. Let $M\in \Hs$ and $x' = Mx$, and let $\zeta_j = \lim_{t\to 0} H^t\cdot x'$. Then
\begin{align}\label{bbeq1}
    H^t x' = H^tMx = (H^t M (H^t)^{-1}) H^t x.
\end{align}
Let $M = D \cdot U$, where $D\in\Ts$ and $U\in \Hs_u$. As $H^t\in \Ts$, it commutes with $D$; therefore, 
\begin{align}\label{bbeq2}
H^t M (H^t)^{-1} = D H^t U (H^t)^{-1}.
\end{align}
Now as $U \in \Hs_u$, we have $U = \exp(u)$ for some $u\in\he_n$. Here $\exp$ should be understood as the algebraic exponential for unipotent groups (see \cite[Proposition 14.32]{Milne}). We then have
$$H^t U (H^t)^{-1} = H^t \exp(u) (H^t)^{-1} = \exp(\Ad_{H^t}(u)).$$
By Lemma~\ref{posint} the weights of the $H^t$-action on $\he_n$ are positive. Therefore, $\lim_{t\to 0} \Ad_{H^t}(u) = 0$, hence $\lim_{t\to 0} H^t U (H^t)^{-1} = 1.$ Combining \eqref{bbeq1} and \eqref{bbeq2} gives
$$H^t x' = D H^t U (H^t)^{-1} H^t x.$$
Passing to limit $t\to 0$ then yields
$$\zeta_j = D \zeta_i.$$
As $\zeta_i$ is fixed by $\Ts$, we get $i=j$, hence $x'\in W_i^+$ as desired.
\end{proof}

\subsection{Structure of \texorpdfstring{$\ZZ$}{Z}}
\label{structure}

In order to prove $H^*_\Ts(X)\cong \C[\ZZ]$, we study the structure of $\ZZ$ and construct a map $H^*_\Ts(X)\to \C[\ZZ]$. Let $(\w,x)\in\ZZ$. This means that $e+\w$ vanishes on $x$, and by Lemma~\ref{regzer} it is an isolated zero. By Theorem~\ref{jordan} there exists an $M\in \Hs$ such that $e+\w = \Ad_M(\w + n')$, where $[\w,n'] = 0$ and $n'\in [\he,\he]$. Then by Lemma~\ref{lemad} we have that $M^{-1} x$ is a zero of $\w+n'$, and by Lemma~\ref{lemzer} it is a zero of $\ttt$. Hence we get $x = M\zeta_i$ for some $i\in\{0,1,\dots,s\}$. Moreover, $\zeta_i$ is a zero not only  of $\ttt$ but also of $n'$.

\begin{example}
We continue Example~\ref{exjor} and use the notation from Example~\ref{exgr} for the elements of $\ttt$.
\begin{enumerate}
\item Let $\w\in\ttt\cong \C^2$ be of the form $\w = (v_1,v_2)$ with $v\neq 0$, $w\neq 0$, $v\neq w$. We know that $e+\w = M_{\w} \w M_{\w}^{-1}$, and therefore any zero of $e+\w$ is of the form $x = M_\w \zeta_i$, and conversely, for any $i$ the point $M_\w \zeta_i$ is fixed by $\w+e$.
\item If $\w = (v_1,0)$ with $v_1\neq 0$, then we have a matrix $M_\w\in\Bs_3$ such that
$$(e+\w) = M_{\w}
\begin{pmatrix}
 -v_1/3 & 0 & 1 \\
 0 & 2v_1/3 & 0 \\
 0 & 0 & -v_1/3
\end{pmatrix} M_{\w}^{-1}.$$
Therefore, every zero of $e+\w$ is of the form $x = M_\w \zeta_i$ for an $i$ such that $\zeta_i$ is also a zero of
$$E_{13}= \begin{pmatrix}
 0 & 0 & 1 \\
 0 & 0 & 0 \\
 0 & 0 & 0
\end{pmatrix}.$$
But conversely, if $\zeta_i$ is additionally a zero of $E_{13}$, then $M_\w\zeta_i$ is a zero of $e+\w$.
\end{enumerate}
\end{example}

\begin{remark}\label{remun}
By Theorem~\ref{ABBst}, if $x = M\zeta_i$, then $\zeta_i$ is in the same plus-cell as $x$. But $\zeta_i$ itself is a torus-fixed point; hence $\zeta_i = \lim_{t\to 0} H^t\cdot x$. In particular, this means that regardless of the potential choice of $M$ we might make, we always get the same torus-fixed point; \textit{i.e.} if $x = M_1\zeta_{i_1} = M_2\zeta_{i_2}$, then $i_1 = i_2$. The elements $M$ and $n'$ are however not unique.

Note that for $i=0,1,\dots,s$ and $\w\in\ttt$, there is at most one zero of $e+\w$ in the plus-cell of $\zeta_i$. Indeed, assume that there are two such points. By the above, if we choose any $M$ such that $e+\w = \Ad_M(\w+n')$, then they are of the form $x_1 = M\zeta_{i_1}$, $x_2 = M\zeta_{i_2}$. But as in the last paragraph, in fact we have $i_2 = i_1 = i$. Therefore, $x_1 = x_2$.
\end{remark}

The converse statement also holds for particular torus-fixed points. Assume that we are given $\w\in\ttt$ and $M_\w\in \Hs$, $n'\in\he_n$ such that $e+\w = \Ad_{M_\w}(\w+n')$ and $[w,n']=0$. In this case, if $\zeta_i$ is a zero of $n'$, then $M_\w \zeta_i$ is a zero of $e+\w$. However, for given $\w$ the corresponding vector field $V_{n'}$ in general does not vanish in all the torus-fixed points.

\begin{example}
Let us consider the standard action of $\Bs_3$ on $\PP^2$; \textit{i.e.} we define
$$
\begin{pmatrix}
a & b & c \\
0 & d & e \\
0 & 0 & f
\end{pmatrix}
\cdot [v_0:v_1:v_2]
= [u_0:u_1:u_2]
$$
for $u_0$, $u_1$, $u_2$ such that
$$\begin{pmatrix}
a & b & c \\
0 & d & e \\
0 & 0 & f
\end{pmatrix}
\begin{pmatrix}
v_0 \\ v_1 \\ v_2
\end{pmatrix}
= 
\begin{pmatrix}
u_0 \\ u_1 \\ u_2
\end{pmatrix}.
$$

We have three torus-fixed points $\zeta_1=o=[1:0:0]$, $\zeta_2 = [0:1:0]$, $\zeta_3= [0:0:1]$. For $\w = (v_1,v_2)\in\C^2\cong \ttt$ regular, there exists a matrix $M_\w$ such that $e+\w = M_\w \w M_\w^{-1}$. Then $M_\w \zeta_i$ is a fixed point of $e+\w$ for $i=1,2,3$.

However, if $\w = (v_1,0)$ with $v_1\neq 0$, then there exists a matrix $M_\w$ such that $e+\w = M_\w (\w + e_{13}) M_\w^{-1}$. The vector field $V_{e_{13}}$ corresponding to $e_{13}$ vanishes at $\zeta_1$ and $\zeta_2$ (but not at $\zeta_3$); therefore, the zeros of $e+\w$ are exactly of the form $M_\w \zeta_1$ and $M_\w \zeta_2$.

Specializing even more, if we consider $\w = (0,0)$, then $e+\w = e$ is already a Jordan matrix (we can take $M_\w = I_3$). Its only zero is $\zeta_1=o$, so the only fixed point of $e+\w$ is $o$.
\end{example}

We will define a map $H_\Ts^*(X)\to \C[\ZZ]$ by constructing, for each element of $H_\Ts^*(X)$, a function in $\C[\ZZ]$ by its values. So that it is well defined, we first show that $\ZZ$ is reduced.

Remember that we defined a $\Cs$-action on $X$ and $\ttt$ -- see the comment below Definition~\ref{defz}. It turns out (\textit{cf.} \cite[Proposition 1]{CarDef}) that if we consider the Bia{\l}ynicki-Birula minus-decomposition on $X$, then the minus-cell $X_o:=W_0^-$ corresponding to $o$ is open. In other words, all of the weights of the action around $o$ are negative. Therefore, we can choose on $X_o$ coordinates $x_1$, $x_2$, \dots, $x_n$ that are weight vectors of $\Ts$, and the   values of weights on $h$ are positive integers $a_1$, $a_2$, \dots, $a_n$. Using these coordinates we model $X_o$ as a vector space; thus we can identify the tangent spaces to its points with $X_o$ itself.

We also have the grading on $\C[\ttt]$ defined by the action of $\Cs$ on $\ttt$ (of weight $-2$). Therefore, choosing coordinates $v_1,\dots,v_r$ on $\ttt$, we have
$$\C[\ttt\times X_o] = \C[v_1,v_2,\dots,v_{r},x_1,x_2,\dots,x_n]$$
with $\deg v_i = 2$ (for $i=1,2,\dots,r$) and $\deg x_i = a_i$ (for $i=1,2,\dots,n$). The tangent bundle of $X_o$, as an affine space, is trivial, and the coordinates on $X_o$ define its trivialisation; hence we can speak of coordinates of $V_{e+\ttt}$ (\textit{cf.} Remark~\ref{remtan}). We now prove the following Lemma, which for $\Hs = \Bs_2$ was proved in \cite[Theorem 4]{CarDef}.

\begin{lemma}\label{lemcm}
The scheme $\ZZ$ is complete intersection and reduced and contained in $\ttt\times X_o$, hence affine. The ideal of $\ZZ$ in $\C[\ttt\times X_o] = \C[v_1,v_2,\dots,v_{r},x_1,x_2,\dots,x_n]$ is then generated by the vertical coordinates of the vector field $V_{e+\ttt}$:
$$\left(V_{e+\ttt}\right)_1,\left(V_{e+\ttt}\right)_2,\dots, \left(V_{e+\ttt}\right)_n.$$
The degree of each $\left(V_{e+\ttt}\right)_i$ is equal to $a_i+2$, and together with $v_1$, $v_2$, \dots, $v_{r}$, they form a homogeneous regular sequence in $\C[v_1,v_2,\dots,v_{r},x_1,x_2,\dots,x_n]$.
\end{lemma}

\begin{proof}
First, let us show that $\ZZ$ is contained in $\ttt\times X_o$. Let $(\w,x)\in \ZZ$. We then know that $x$ is a zero of the vector field $V_{e+\w}$. For any $t\in\Cs$, by Lemma~\ref{lemad} we have that $H^t\cdot x$ is a zero of $V_{\Ad_{H^t}(e+\w)}$. As $\Ad_{H^t}(e+\w) = t^2e + \w$, this means that $H^t \cdot x$ is a zero of $e+t^{-2}\w$. When we take $t\to\infty$, this converges to $e$. Therefore, $\lim_{t\to\infty} H^t\cdot x = o$. This means that $x\in X_o$.

Now we will prove that $\left(V_{e+\ttt}\right)_i$ is homogeneous of degree $a_i+2$. We have
$$\left(V_{e+\ttt}\right)_i |_{t\cdot(x,\w)} = 
\left(V_{e+\w/t^2}|_{H^t\cdot x}\right)_i =
\left(H^t_* \left(V_{\Ad_{H^{t^{-1}}}(e+\w/t^2)}|_x\right)\right)_i =
\left(H^t_* \left(V_{e/t^2+\w/t^2}|_x\right)\right)_i, 
$$
and $H^t$ acts on $\supth{i}$ coordinate of tangent space by multiplying it by $t^{-a_i}$, so 
$$\left(H^t_* \left(V_{e/t^2+\w/t^2}|_x\right)\right)_i = t^{-a_i} \left(V_{e/t^2+\w/t^2}|_x\right)_i
=t^{-a_i-2} \left(V_{e+\ttt}\right)_i |_{(x,\w)}.$$
Since $v_1$, $v_2$, \dots, $v_{r}$ have degree $2$, the sequence
$$\left(V_{e+\ttt}\right)_1,\left(V_{e+\ttt}\right)_2,\dots, \left(V_{e+\ttt}\right)_n, v_1, v_2, \dots, v_{r}$$
consists of homogeneous functions on the $(r+n)$-dimensional affine space $\ttt\times X_o$. There are $r+n$ of them, and they have only one common zero. Therefore (see \cite[Proposition 4.3.4]{Benson}), they form a regular sequence. In particular, $\ZZ$ is the zero scheme of a regular sequence $\left(V_{e+\ttt}\right)_1$, $\left(V_{e+\ttt}\right)_2$, \dots, $\left(V_{e+\ttt}\right)_n$; therefore, it is complete intersection.

Now we have to prove that $\ZZ$ is reduced. Let $\pi\colon \ZZ\to\ttt$ be the first projection. By Theorem~\ref{unif} we get an isomorphism $\pi^{-1}(\ttt^\reg) \cong \ttt^\reg\times X^\Ts$. The first factor, as an open subscheme of affine space, is reduced. The fixed points of the torus are also reduced (\textit{cf.} \cite[Theorem 13.1]{Milne}); therefore, $\pi^{-1}(\ttt^\reg)$ is reduced.

Now note that $\pi^{-1}(\ttt^\reg)$ is an open dense subset in $\ZZ$. It is open because $\ttt^\reg$ is open in $\ttt$. To prove that it is dense, assume to the contrary that there exists an $x\in \ZZ\setminus\overline{\pi^{-1}(\ttt^\reg)}$. Let $Y$ be its irreducible component in $\ZZ$. As $\ZZ = \overline{\pi^{-1}(\ttt^\reg)} \cup \pi^{-1}(\ttt\setminus\ttt^\reg)$ and both sets are closed, by irreducibility $Y$ has to be contained in one of them. As $x$ is not contained in the former, $Y$ has to be contained in the latter, so that $\pi(Y)\subset \ttt\setminus\ttt^\reg$. As $\ttt\setminus\ttt^\reg$ is a union of hyperplanes in $\ttt$, the same argument shows that $\pi(Y)$ lies within one of them (of dimension $r-1$). Considering $\pi|_Y$ as mapping to $\overline{\pi(Y)}$ and reducing if needed, we get a dominant map between integral schemes. Note that as $\ZZ$ is complete intersection, it is Cohen--Macaulay, and thus equidimensional by \cite[Theorems 17.6 and 6.5]{Matsu}. As $\ttt\times\{o\}$ is closed in $\ZZ$ and of dimension $r$, the dimension of $\ZZ$ is at least $r$. Therefore, by the fiber dimension theorem (see\cite[Exercise~II.3.22(b)]{Hart}), the fibers of $\pi|_Y$ are at least $1$-dimensional. But they are finite by Lemma~\ref{regzer}, so we get a contradiction.

Now as $\pi^{-1}(\ttt^\reg)$ is an open dense subset in $\ZZ$, it contains its generic points; hence $\ZZ$ is generically reduced. Using that $\ZZ$ is Cohen--Macaulay, by \cite[Proposition 14.124]{GW} we get that $\ZZ$ is reduced.
\end{proof}

\subsection{The homomorphism \texorpdfstring{$\rho$}{ρ}}\label{sectionmap}
Let $c \in H_\Ts^*(X)$. In Section~\ref{structure} we show that every element $(\w,x)$ of $\ZZ$ satisfies $x = M_\w \zeta_i$. Here $M_\w$ is some element of $\Hs$ depending on $\w$, and $\zeta_i$ is a uniquely determined fixed point of $\Ts$-action. The localisation $c|_{\zeta_i}$ of $c$ to the torus-fixed point can be now seen as a polynomial in $\ttt$ because $H_\Ts^*(\pt) = \C[\ttt]$.
We then define
\begin{align} \label{rho} \rho(c)(\w,x) = c|_{\zeta_i}(\w).\end{align}
This follows the idea of \cite{BC}, where $\rho$ is defined this way for $\Bs_2$.
For any $c\in H_\Ts^*(X)$ this defines a function $\rho(c)$ on the set of closed points $\ZZ$. This clearly gives a $\C[\ttt]$-homomorphism between $H_\Ts^*(X)$ and the algebra of all $\C$-valued functions on $\ZZ$. We have to prove that for any $c\in H_\Ts^*(X)$ the image $\rho(c)$ defines a regular function, which is unique by Lemma~\ref{lemcm}. Thus we get a $\C[\ttt]$-homomorphism
$$\rho\colon  H_\Ts^*(X)\lra \C[\ZZ].$$
In general, assume that we are given an algebraic group $\Hs$ and an $\Hs$-variety $A$. For any $\Hs$-linearised bundle $\Ee$ on $A$, we may consider its equivariant Chern classes $c^\Hs_k(\Ee)\in H^{2k}_\Hs(A)$. Let $p\in A$ be a fixed point of $\Hs$. From the naturality of Chern classes, we get $c^\Hs_k(\Ee)|_p = c^\Hs_k(\Ee_p)$, where $\Ee_p$ is the fiber of $\Ee$ over $p$. This belongs to $H_\Hs^*(\pt)\subset \C[\he]$, and for any $y\in \he$ we get
\begin{align}
\label{cherntr}
c^\Hs_k(\Ee)|_p(y) = \Tr_{\Lambda^k \Ee_p}(\Lambda^k y_p).
\end{align}
Here $y_p$ is the infinitesimal action of $y\in\he$ on $\Ee_p$, which is a representation of $\Hs$.

\begin{lemma}\label{nicefun}
 Let $\Ee$ be an $\Hs$-linearised vector bundle on $X$, and let $k$ be a non-negative integer. Then for any $(\w,x)\in\ZZ$ we have
 $$\rho\left(c_k^\Ts(\Ee)\right)(\w,x) = \Tr_{\Lambda^k \Ee_x}\left(\Lambda^k (e+\w)_x\right).$$
 In particular, $\rho(c_k^\Ts(\Ee))\in \C[\ZZ]$.
\end{lemma}

\begin{proof}
We have $x = M_\w \zeta_i$ for some $\zeta_i\in X^\Ts$ and $M_\w\in\Hs$. Moreover,
$$e+\w = \Ad_{M_\w}(\w+e')$$
for some $e'\in \he_n$ that vanishes at $\zeta_i$ and commutes with $\w$. Note that, as $\Ee$ is $\Hs$-linearised, 
$$\Tr_{\Lambda^k \Ee_x}\left(\Lambda^k (e+\w)_x\right) = \Tr_{\Lambda^k \Ee_{M_\w^{-1}x}}\left(\Lambda^k \left( \Ad_{M_\w^{-1}}(e+\w) \right)_{M_\w^{-1}x}\right)
= \Tr_{\Lambda^k \Ee_{\zeta_i}}\left(\Lambda^k(\w+e')_{\zeta_i}\right).
$$
From \eqref{rho} and \eqref{cherntr} we have
$$\rho\left(c_k^\Ts(\Ee)\right)(\w,x) = c_k^\Ts(\Ee)|_{\zeta_i}(\w) = \Tr_{\Lambda^k \Ee_{\zeta_i}}\left(\Lambda^k w_{\zeta_i}\right).$$
Thus we have to prove that 
$$ \Tr_{\Lambda^k \Ee_{\zeta_i}}\left(\Lambda^k(\w+e')_{\zeta_i}\right) = \Tr_{\Lambda^k \Ee_{\zeta_i}}\left(\Lambda^k w_{\zeta_i}\right).$$
But by the assumptions that $[\w,e']=0$, $\w$ is semisimple and $e'$ is nilpotent, we get that the sum $\w + e'$ is the Jordan decomposition of $\Ad_{M_\w^{-1}}(e+\w)$ in the sense of \cite[Theorem~4.4]{Borel}. Then by the naturality of the Jordan decomposition, the derivative of the representation $\Stab_\Hs(\zeta_i)\to \GL(\Ee_{\zeta_i})$ preserves it. Therefore, $\w_{\zeta_i}$, seen as an element of $\gl(\Ee_{\zeta_i})$, is the semisimple part of $(\w+e')_{\zeta_i}$, seen as an element of $\gl(\Ee_{\zeta_i})$.

But for the Jordan decomposition in the general linear group, the eigenvalues of the semisimple part are the same as the eigenvalues of the decomposed element. Because traces of external powers are polynomials in eigenvalues, this concludes the proof.
\end{proof}
{The following lemma is based on \cite[Proposition 3]{CarDef}, which asserts it for $\Bs_2$.}

\begin{lemma}\label{chern1}
The cohomology ring $H^*(X)$ is generated, as a $\C$-algebra, by Chern classes of\, $\Hs$-linearised vector bundles on $X$.
\end{lemma}

\begin{proof}
We know that the fundamental classes of the plus-cells form a basis of $H_*(X)$; hence their Poincar\'{e} duals form a basis of $H^*(X)$. Now we use Baum--Fulton--MacPherson's Grothendieck--Riemann--Roch theorem (see \cite[Theorem 18.3(5)]{Ful}). We get that for any plus-cell $W_i\in X$, the homology class $(\ch(W_i) \td(X_i))\cap [X]$ is equal to the sum of $[W_i]$ and lower-degree terms. Therefore, $\ch(W_i)$ is equal to the sum of the dual class of $[W_i]$ and higher-degree terms. Therefore, Chern characters of the structure sheaves of plus-cells generate $H^*(X)$. 

As the plus-cells are $\Hs$-stable by Theorem~\ref{ABBst}, we get that $\ch$ is surjective when restricted to the Grothendieck group of $\Hs$-equivariant coherent sheaves. By \cite[Corollary 5.8]{Thomason} it is generated by the classes of $\Hs$-equivariant vector bundles, and the conclusion follows.
\end{proof}

\begin{remark}
We did not use the regularity of the action in the proof. In fact, it was enough to know that the fixed points of $\Ts$ are isolated. One could also argue the following in the general case. By \cite[Section~15.1, Example (2)]{Borel} a linear solvable group over $\C$ is split. Then the restriction $K^0_\Hs(X) \to K^0_\Ts(X)$ is an isomorphism, \textit{cf.} \cite[Corollary 2.16]{Merk}, and the restriction $K^0_\Ts(X)\to K^0(X)$ is a surjection, \textit{cf.} \cite[Proposition 3.1]{Merk}. The Chern character is an isomorphism from $K^0(X)\otimes \C$ to $A^*(X)\otimes \C$, \textit{cf.}  \cite[Theorem 18.3]{Ful}, and the cycle class map $A^*(X)\to H^*(X,\Z)$ is an isomorphism due to the paving given by the Bia{\l}ynicki-Birula decomposition, \textit{cf.} \cite[Example 19.1.11]{Ful}. Therefore, the (non-equivariant) Chern character gives a surjection $K^0_\Hs(X)\to H^*(X,\C)$.
\end{remark}

\begin{lemma}\label{chern2}
The equivariant cohomology $H_\Ts^*(X)$ is generated, as a $\C[\ttt]$-algebra, by $\Ts$-equivariant Chern classes of\, $\Hs$-equivariant vector bundles on $X$.
\end{lemma}

\begin{proof}
Recall that $\I$ denotes the maximal ideal of $\C[\ttt]$ cutting out the zero point. Since $X$ is equivariantly formal, we have an exact sequence
$$0\lra \I H_\Ts^*(X)\lra H_\Ts^*(X)\lra H^*(X)\lra 0.$$
By Lemma~\ref{chern1} we get that the $\C$-algebra $H^*(X)$ is generated by Chern classes of $\Hs$-linearised vector bundles on $X$. Then by the graded Nakayama lemma (see Corollary~\ref{cornak}), the $\C[\ttt]$-algebra $H^*_\Ts(X)$ is generated by their equivariant Chern classes.
\end{proof}

\noindent
This together with Lemma~\ref{nicefun} gives the following. 

\begin{corollary}
The map $\rho$ is a homomorphism of\, $\C[\ttt]$-algebras $H_\Ts^*(X)\to\C[\ZZ]$.
\end{corollary}

\subsection{Proof of  isomorphism}

\begin{proof}[Proof of Theorem~\ref{finsolv}]
Clearly, $\rho$ preserves the grading. For the injectivity, note that for any $c\in H_\Ts^*(X)$ we can extract from $\rho(c)$ the localisations $c|_{\zeta_i}$ for all $i$ -- as on the regular locus the function $\rho(c)$ is defined by all those localisations. Recall that $X$ is equivariantly formal; see \eqref{formal}. Therefore, we get the injectivity of $\rho$ from the injectivity of localisation on equivariantly formal spaces; \textit{cf.} \cite[Theorem 1.6.2]{GKM}.

Hence to prove that the map is an isomorphism, it suffices to check that the Poincar\'{e} series of the two sides coincide. Since $X$ is equivariantly formal, $H_\Ts^*(X)$ is a free $\C[\ttt]$-module and
$$H_\Ts^*(X)/\I H_\Ts^*(X) \cong H^*(X).$$
Therefore,
\begin{align}
\label{PHT}
    P_{H^*(X)}(t) = P_{H_\Ts^*(X)}(t)(1-t^2)^{r}.
\end{align}

On the other hand, from Lemma~\ref{lemcm} we know that the generating set of $\I$ is a regular sequence in $\C[\ZZ]$; hence 
\begin{align}
\label{PCZ}
P_{\C[\ZZ]/\I\C[\ZZ]}(t) = P_{\C[\ZZ]}(t)\left(1-t^2\right)^{r}.
\end{align}

Now $\C[\ZZ]/\I\C[\ZZ]$ is the zero scheme of the vector field given by $e$. In addition, the action of the torus $H^t$ satisfies $\Ad_{H^t}(e) = t^2 e$. Therefore, by the Akyildiz--Carrell version of the Carrell--Liebermann theorem (see \cite{ACLS}, and \cite[Theorem 1.1]{AC} for this particular case),  
we have $\C[\ZZ]/\I\C[\ZZ]\cong H^*(X)$ and in particular
$$P_{\C[\ZZ]/\I\C[\ZZ]}(t) = P_{H^*(X)}(t).$$
Therefore, from \eqref{PHT} and \eqref{PCZ} we get
\begin{equation*}\pushQED{\qed}
  P_{\C[\ZZ]}(t) = P_{H_\Ts^*(X)}(t).
\qedhere \popQED
	\end{equation*}
\renewcommand{\qed}{}       
\end{proof}

\begin{remark}
From Theorem~\ref{finsolv} we get that $\C[\ZZ]$ is a finitely generated free module over $\C[\ttt]$. Therefore, the map $\pi\colon \ZZ\to\ttt$ is finite flat.
\end{remark}

\begin{remark}
The theorem can in fact be proved for a slightly larger class of solvable groups. We need $\Hs$ to be a connected linear algebraic solvable group and as before $(e,h)$ to be an integrable $\bb(\ssl_2)$-pair, but it does not necessarily have to be principal. For the proof of Theorem~\ref{ABBst}, we need to assume $\alpha(h) > 0$ for any root $\alpha$ of $\Hs$. However, even this assumption can be made unnecessary as we can consider the subgroup $\Hs'$ generated by $\Ts$ and the additive group generated by $e$. By \cite[Theorem 7.6]{Borel} it is algebraic, and its Lie algebra is generated by $\ttt$ and $e$. As the Lie bracket of $h$-weight vectors adds the weights, we clearly see that all the weights on $\Hs'$ are non-negative multiples of $2$.

Even if we assume that $\Hs$ is generated by $\Ts$ and the additive group generated by $e$, it does not follow that $e$ is regular. Take for example
 $$
 \Hs = \left\{\left.
 \begin{pmatrix}
 t/u^2 & * & * & * \\
 0 & t & * & * \\
 0 & 0 & u & * \\
 0 & 0 & 0 & u/t^2
 \end{pmatrix}\right| t,u\in \Cs
 \right\},
 $$
 where the asterisks are understood to stand for any complex numbers. We choose
 $$
 h = 
 \begin{pmatrix}
 3 & 0 & 0 & 0 \\
 0 & 1 & 0 & 0 \\
 0 & 0 & -1 & 0 \\
 0 & 0 & 0 & -3
 \end{pmatrix}, \qquad
 e = 
 \begin{pmatrix}
 0 & 1 & 0 & 0 \\
 0 & 0 & 1 & 0 \\
 0 & 0 & 0 & 1 \\
 0 & 0 & 0 & 0
 \end{pmatrix}.
 $$
 Because the maximal torus is $2$-dimensional and the centraliser of $e$ is $3$-dimensional,  $e$ is not regular. However, together with the diagonal matrices, it generates $\he$ as a Lie algebra.
 
 In all of our examples of regular action, we only consider principally paired groups, and this extension seems to only include very tropical cases. Therefore, we formulate our results in terms of principally paired groups.
\end{remark}

\subsection{Functoriality}

We now prove  that Theorem~\ref{finsolv} is actually functorial, with respect to both the group and the variety. We prove the latter first.

\begin{proposition} \label{funcprop}
 Assume that $X$ and $Y$ are two $\Hs$-regular varieties and $\phi\colon X\to Y$ is an $\Hs$-equivariant morphism between them. Let $\ZZ_X \cong \Spec H^*_\Ts(X)$ and $\ZZ_Y \cong \Spec H^*_\Ts(Y)$ be the schemes constructed above for $X$ and $Y$, respectively. The map $(\id,\phi)\colon \ttt\times X\to \ttt\times Y$ induces a morphism $\ZZ_X\to \ZZ_Y$, and the following diagram commutes:
$$
\begin{tikzcd}
H^*_\Ts(Y) \arrow[r, "\phi^*"] \arrow[dd, "\rho_Y"]
& H^*_\Ts(X) \arrow[dd, "\rho_X"]
\\ \\
\C[\ZZ_Y] \arrow[r, "{(\id,\phi)^*}"]
& \C[\ZZ_X]\rlap{.}
\end{tikzcd}
$$
In other words, $\rho$ is a natural isomorphism between the functors $H^*_\Ts$ and $\C[\ZZ]$ on the category of $\Hs$-regular varieties.
\end{proposition}

\begin{proof}
Consider a class $c\in H^*_\Ts(Y)$. We want to show that for any $(\w,x)\in\ZZ_X$ the functions $\rho_X(\phi^*(c))$ and $(\id,\phi)^*(\rho_Y(c))$ take the same value on $(\w,x)$. We know from Section~\ref{structure} that $x = M_\w \zeta$, where $M_\w$ is some element of $\Hs$ depending on $\w$ and $\zeta$ is one of the $\Ts$-fixed points of $X$. Obviously, then $\phi(\zeta)$ is a $\Ts$-fixed point in $Y$ and $\phi(x) = M_\w \phi(\zeta)$. We then have
$$(\id,\phi)^*(\rho_Y(c))(\w,x) = \rho_Y(c)(\w,\phi(x)) = c|_{\phi(\zeta)}(\w).$$
On the other hand, 
$$\rho_X(\phi^*(c))(\w,x) = \phi^*(c)|_{\zeta}(\w).$$
Now the equality  above follows from the functoriality of $H^*_\Ts$ and the commutativity of
\begin{equation*}\pushQED{\qed}
\begin{tikzcd}
\{\zeta\} \arrow[dd, "\iota_\zeta"] \arrow[r, "\phi"]
& \{\phi(\zeta)\} \arrow[dd, "\iota_{\phi(\zeta)}"]
\\ \\
X \arrow[r, "\phi"] 
& Y\rlap{.}
\end{tikzcd}
\qedhere \popQED
	\end{equation*}
\renewcommand{\qed}{}
\end{proof}

\begin{proposition}\label{funcgrp}
 Assume that $\Hs_1$, $\Hs_2$ are solvable principally paired groups. Let $\Ts_i\subset \Hs_i$ be the corresponding maximal tori and $e_i\in(\he_i)_n$ the corresponding nilpotent elements in their Lie algebras. Let $\psi\colon \Hs_1\to\Hs_2$ be a homomorphism of algebraic groups satisfying
 $$\psi(\Ts_1)\subset \Ts_2,\quad \psi_*(e_1) = e_2.$$
 Assume that $\Hs_2$ acts regularly on a smooth projective variety $X$. Then the map $\psi$ together with the $\Hs_2$-action induce an action of\, $\Hs_1$ on $X$, which is also regular. In turn, the map $(\psi_*,\id)$ induces a morphism $\ZZ_{\Hs_1} \to \ZZ_{\Hs_2}$, and the following diagram commutes:
$$
\begin{tikzcd}
H^*_{\Ts_2}(X) \arrow[r, "\psi^*"] \arrow[dd, "\rho_{\Hs_2}"]
& H^*_{\Ts_1}(X) \arrow[dd, "\rho_{\Hs_1}"]
\\ \\
\C[\ZZ_{\Hs_2}] \arrow[r, "{(\psi_*,\id)^*}"]
& \C[\ZZ_{\Hs_1}]\rlap{.}
\end{tikzcd}
$$
\end{proposition}

\begin{proof}
 As $\psi_*(e_1) = e_2$, the group $\Hs_1$ clearly acts regularly on $X$. Obviously, if $(\w,x)\in\ZZ_{\Hs_1}$, then $e_1+\w$ vanishes at $x$, and therefore $\psi_*(e_1+\w) = e_2 + \psi(\w)$ vanishes at $x$, hence $(\psi_*,\id)$ maps $\ZZ_{\Hs_1}$ to $\ZZ_{\Hs_2}$.
 
 Now let $c\in H^*_{\Ts_2}(X)$ and $(\w,x)\in \ZZ_{\Hs_1}$. We want to prove that
 $$(\psi_*,\id)^*(\rho_{\Hs_2}(c))(\w,x) = \rho_{\Hs_1}(\psi(c))(\w,x).$$
 We know that $x = M_\w \zeta$ for some $M_\w\in \Hs_1$ depending on $\w$ and an isolated $\Ts_1$-fixed point $\zeta$. Then by Lemma~\ref{lemfix} the point $\zeta$ is fixed by $\Ts_2$. Therefore (\textit{cf.} Remark~\ref{remun}), we have 
 $$(\psi_*,\id)^*(\rho_{\Hs_2}(c))(\w,x) = \rho_{\Hs_2}(c)(\psi_*(\w),x) = c|_{\zeta}(\psi_*(\w))$$
 and
 $$\rho_{\Hs_1}(\psi(c))(\w,x) = \psi(c)|_{\zeta}(\w).$$
 Now the equality follows from the commutativity of
 \begin{equation*}\pushQED{\qed}
\begin{tikzcd}
H^*_{\Ts_2}(\pt) \arrow[r, "\psi^*"] \arrow[dd, "\cong"]
& H^*_{\Ts_1}(\pt) \arrow[dd, "\cong"]
\\ \\
\C[\ttt_2] \arrow[r, "(\psi_*)^*"]
& \C[\ttt_1]\rlap{.}
\end{tikzcd}
\qedhere \popQED
	\end{equation*}
\renewcommand{\qed}{}     
\end{proof}

\subsection{Examples and comments}
We illustrate Theorem~\ref{finsolv} with a few examples.

\begin{figure}[ht!]
\begin{center}
 \includegraphics[width=10cm]{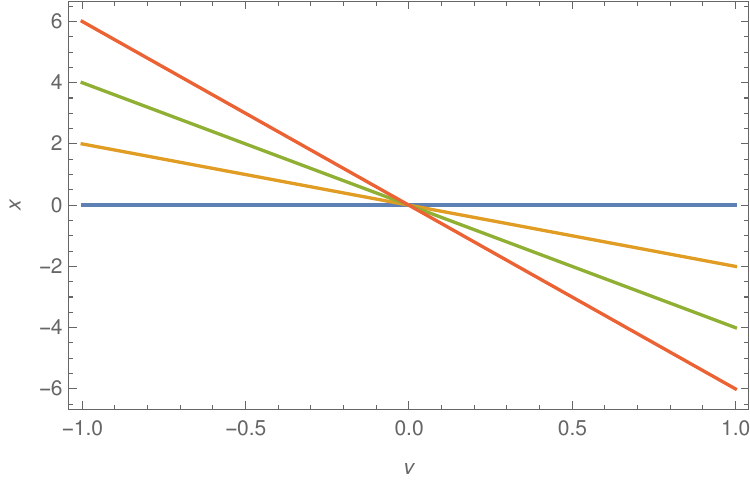}
\end{center}
\caption{$\Spec H_{\Cs}^*(\PP^3)$.}
\label{hcp4}
\end{figure}

\begin{example}\label{exsl22}
 We continue Example~\ref{exsl2}, which already appears in \cite{BC}. The point $o = [1:0:\dots:0]$ is the unique zero of $e$. 
 If $[z_0:z_1:\dots:z_n]$ are the homogeneous coordinates of $\PP^n$, then the scheme $\ZZ$ lies completely in the affine chart $X_o$ of $o$, with affine coordinates $x_i = z_i/z_0$ for $i=1,2,\dots,n$. We have
 $$V_h|_{x_1,\dots,x_n} = (-2x_1,-4x_2,\dots,-2nx_n)$$
 and
 $$V_e|_{x_1,\dots,x_n} = (x_2-x_1x_1, x_3 - x_1x_2, x_4 - x_1x_3,\dots,x_n-x_1x_{n-1},-x_1x_n).$$
 Then
 $$
V_{e+vh}|_{x_1,\dots,x_n} = (x_2-x_1(x_1+2v), x_3 - x_2(x_1+4v), \dots, x_n-x_{n-1}(x_1+2(n-1)v), -x_n(x_1+2n)).
 $$
 If we consider the zero scheme $\ZZ$ of $e+vh$ within $\ttt\times X_o$, then the coordinates $x_2,\dots, x_n$ are clearly determined by $x_1$ and $v$, and we can identify $\ZZ$ with the subscheme of $\Spec \C[v,x_1]$ cut out by the equation
 \[x_1(x_1+2v)(x_1+4v)\cdots(x_1+2nv) = 0.\]
 In other words, $H_{\Cs}^*(\PP^n) = \C[v,x]/\big(x(x+2v)(x+4v)\cdots(x+2nv)\big)$ with $\deg v = \deg x = 2$. See Figure~\ref{hcp4}.
\end{example}

\begin{remark}
Clearly, a product $X\times Y$ of two varieties with a regular $\Hs$-action is also regular, and its equivariant cohomology scheme can be represented as a fiber product; \textit{i.e.} $H_{\Ts}^*(X,Y) = H_{\Ts}^*(X) \otimes_{H_{\Ts}^*} H_{\Ts}^*(Y)$.

In particular, the product $\PP^1 \times \PP^1$ is regular under the action of $\SL_2$, hence also of $\Bs_2$. It embeds in $\PP^3$ via the Segre embedding. The action of $\SL_2$ on $\PP^3$ from Example~\ref{exsl2} is also regular. However, the Segre embedding cannot be $\SL_2$- or even $\Bs_2$-equivariant with respect to those two actions. In fact, using Theorem~\ref{finsolv} we can prove a more general statement. 
\end{remark}

\begin{corollary} \label{surj}
Let a principally paired solvable group $\Hs$ act regularly on a smooth projective variety $X$. Assume that $Z$ is its closed, smooth, $\Hs$-invariant subvariety. Then the induced map on cohomology rings
\[f^*\colon  H^*(X,\C) \lra H^*(Z,\C)\]
is surjective.
\end{corollary}

\begin{proof}
Clearly, $Z$ is also an $\Hs$-regular variety. From Theorem~\ref{finsolv} we have $H^*_\Ts(X) = \C[\ZZ_X]$ and $H^*_\Ts(Z) = \C[\ZZ_Z]$, where $\ZZ_X$ and $\ZZ_Z$ are the zero schemes constructed for $X$ and $Z$ according to Definition~\ref{defz}. But clearly from the definition we see that $\ZZ_Z$ is the (reduced) intersection $\ZZ_X\cap Z$, hence a closed subvariety of $\ZZ_X$. This means that the induced map $\C[\ZZ_X]\to \C[\ZZ_Z]$ is surjective. By Proposition~\ref{funcprop} this is the same as the map induced on equivariant cohomology. By equivariant formality we get the non-equivariant cohomology by tensoring with $\C$ over $H^*_\Ts$, and this operation is right-exact; hence it preserves surjectivity.
\end{proof}

In particular, as $h^2(\PP^1\times\PP^1) = 2$, the product $\PP^1\times\PP^1$ cannot be embedded $\Bs_2$-equivariantly in any of the $\PP^m$ with regular action.

\begin{figure}[ht!]
\begin{center}
\subfloat{
  \includegraphics[width=7cm]{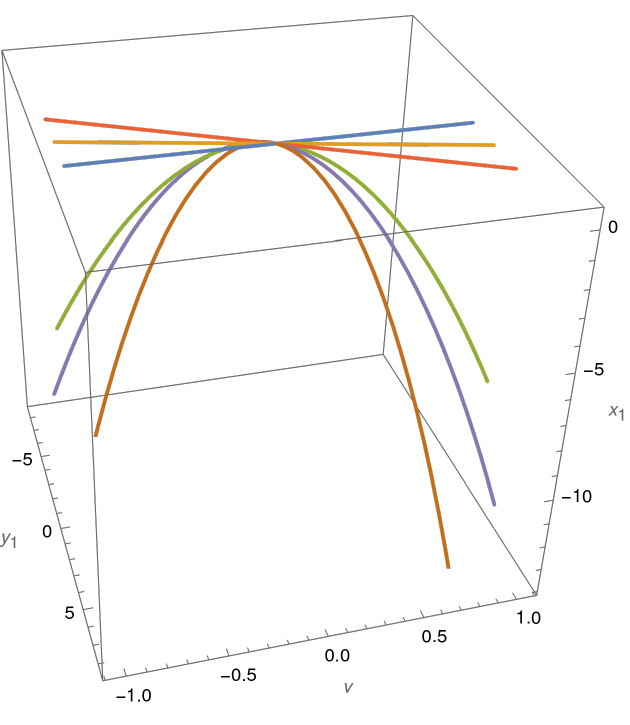}}
  \hfill
\subfloat{
  \includegraphics[width=8.2cm]{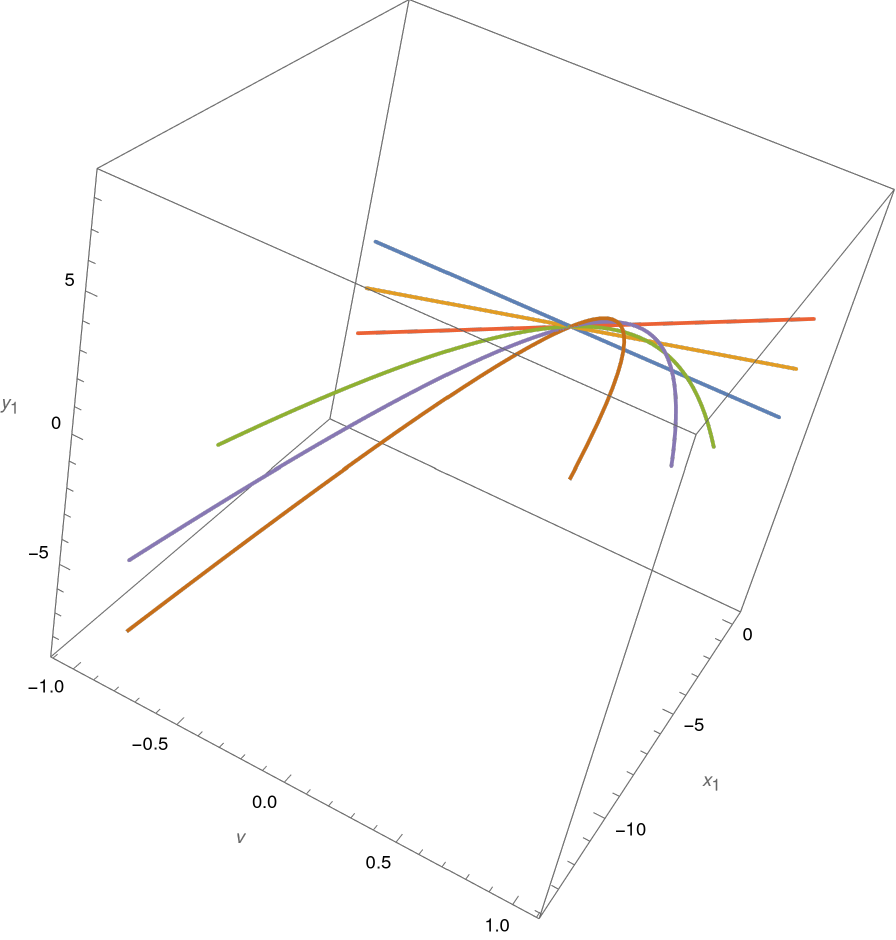}}
\end{center}
\caption{Two different views of $\Spec H_{\Cs}^*(\Gr(2,4))$. Note that all of the components project bijectively to the $v$ axis.}
\label{hcgr}
\end{figure}

\begin{example}\label{csgr}
 As we have defined an action of $\SL_2(\C)$ on any $\C^n$, we can use this to define actions on partial or full flag varieties. Let us consider the action of the upper Borel subgroup of $\SL_2$ on $\C^4$ and the induced action on the Grassmannian $\Gr(2,4)$ of two-planes in $\C^2$. We can identify it with $\SL_4(\C)/\Ps$, where $\Ps$ is the parabolic group of matrices of the form
$$\begin{pmatrix}
    * & * & * & * \\
    * & * & * & * \\
    0 & 0 & * & * \\
    0 & 0 & * & * \\
   \end{pmatrix}.
$$
The only fixed point of $e$ is $o=\Span(e_1,e_2)$, and in the representation above $X_o$ can be thought of as the set of classes of matrices of the form
$$\begin{pmatrix}
    1 & 0 & * & * \\
    0 & 1 & * & * \\
    x_1 & y_1 & * & * \\
    x_2 & y_2 & * & * \\
   \end{pmatrix}.
$$
Then if we write down the coordinates $x_1$, $y_1$, $x_2$, $y_2$ in this order, one checks that
$$V_e|_{x_1,y_1,x_2,y_2} = (x_2-x_1y_1,-x_1-y_1^2+y_2,-x_1y_2,-x_2-y_1y_2)$$
and
$$V_h|_{x_1,y_1,x_2,y_2} = (4x_1,2y_1,6x_2,4y_2).$$
Therefore, the equations of $\ZZ$ in $\C[v,x_1,y_1,x_2,y_2]$ are
$$4vx_1 + x_2-x_1y_1 = 0,\quad
2vy_1-x_1-y_1^2+y_2 = 0, \quad
6vx_2-x_1y_2 = 0,\quad
4vy_2-x_2-y_1y_2 = 0.$$
We can determine $x_2$ and $y_2$ from the first two equations, and plugging them into the other two, we get
$$
x_1 (x_1 + 24 v^2 - 8 v y_1 + y_1^2) = 0,\quad
(y_1-4v) (2 x1 - 2 vy_1 + y_1^2) = 0.
$$
This gives six one-parameter families of solutions (one for each torus-fixed point):
\begin{align*}
&(x_1 = 0, y_1 = 0);&\quad 
&(x_1 = 0, y_1 = 2v);&\quad 
&(x_1 = -8v^2, y_1 = 4v);\\
&(x_1 = 0, y_1 = 4v);&\quad 
&(x_1 = -12v^2, y_1 = 6v);&\quad
&(x_1 = -24v^2, y_1 = 8v);
\end{align*}
see Figure~\ref{hcgr}.
\end{example}

\begin{figure}[ht!]
\begin{center}
 \includegraphics[width=7cm]{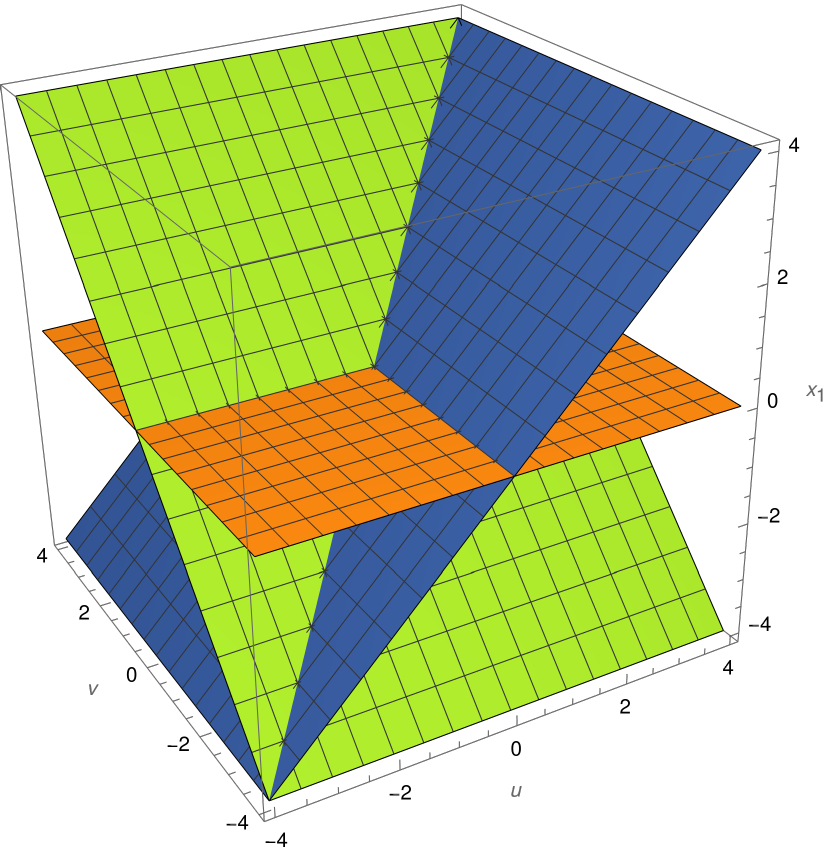}
\end{center}
\caption{$\Spec H_{\Ts}^*(\PP^2)$.}
\label{htpn}
\end{figure}

\begin{example}\label{exsl3p2}
 Let us now switch to groups of higher rank. As in Example~\ref{exgr2}, we can consider the regular nilpotent
 $$e = \begin{pmatrix}
       0 & 1 & 0 & 0 & \dots & 0\\
       0 & 0 & 1 & 0 & \dots & 0\\
       0 & 0 & 0 & 1 & \dots & 0\\
       \vdots & \vdots & \vdots & \vdots & \ddots & \vdots \\
       0 & 0 & 0 & 0 & \dots & 1 \\
       0 & 0 & 0 & 0 & \dots & 0
      \end{pmatrix}
 $$
 in $\SL_{n+1}$. We have the regular action of $\SL_{n+1}$ on $\PP^n$, which in particular restricts to a regular action of its upper Borel subgroup. We continue using notation from Example~\ref{exgr} for the elements of $\ttt$. As in Example~\ref{exsl22}, we have 
 $$V_e|_{x_1,\dots,x_n} = (x_2-x_1x_1, x_3 - x_1x_2, x_4 - x_1x_3,\dots,x_n-x_1x_{n-1},-x_1x_n).$$
 For the element $(v_1,v_2,\dots,v_{n})\in\C^n$, which corresponds to $\diag(0,v_1,v_2,\dots,v_{n}) - \frac{v_1+v_2+\dots+v_{n}}{n+1} I_{n+1}$, the associated vector field at $(x_1,x_2,\dots,x_n)$ is equal to
 $(v_1x_1,v_2x_2,\dots,v_nx_n)$. Hence 
 $$V_{e+(v_1,v_2,\dots,v_n)}|_{x_1,\dots,x_n} = (x_2-x_1(x_1-v_1), x_3 - x_2(x_1-v_2), \dots, x_n-x_{n-1}(x_1-v_{n-1}), -x_n(x_1-v_n)).$$
 Thus we can determine $x_2$, $x_3$, \dots, $x_n$ from $x_1$ and $v_1$, $v_2$, \dots, $v_n$. The scheme $\ZZ$ can then be realised within $\Spec\C[v_1,v_2,\dots,v_n,x_1]$ and cut out by one equation
 $$x_1(x_1-v_1)(x_1-v_2)\cdots(x_1-v_n) = 0.$$
 This scheme consists of $n+1$ hyperplanes. Their intersections, when projected on the $(v_1,\dots,v_n)$-plane, form the toric fan of $\PP^{n}$. The functions on the scheme consist of $n+1$ polynomials, one for each component, that agree on the intersections. This agrees with the classical description of equivariant cohomology of a toric variety as piecewise polynomials on the fan (see \textit{e.g.} \cite[Section 2.2]{Briontor}).
 For $n=2$ the scheme is depicted in Figure~\ref{htpn}.
\end{example}

 \begin{example}
 \label{exflag}
  We can extend the previous example to full flag varieties. Take for example the variety $F_3 = \SL_3/\Bs$ of full flags in $\C^3$. The only zero of $e$ is the flag $\Span(e_1)\subset\Span(e_1,e_2)$, and in this case the cell $X_o$  consists of the flags represented by matrices of the form
 $$
 \begin{pmatrix}
  1 & 0 & * \\
  a & 1 & * \\
  b & c & *
 \end{pmatrix} \in \SL_3(\C).
 $$
 One finds that
 $$V_e|_{a,b,c} = (-a^2+b,-ab,-b+ac-c^2).$$
 If, as before, we consider a pair $\w = (v_1,v_2)\in\C^2$ as an element of $\ttt$, then we have
 $$V_\w|_{a,b,c} = 
 (v_1a,v_2b,(v_2-v_1)c).$$
 Hence the equations for $V_{e+\w} = 0$ are
 $$-a^2+b+v_1a = 0,\qquad -ab+v_2b,\qquad -b+ac-c^2+(v_2-v_1)c.$$
 Plugging $b$ from the first one into the others yields two equations
 $$a(a-v_1)(a-v_2) = 0,\qquad -a^2+av_1+ac-c^2-cv_1+cv_2 =0.$$
 By splitting the first equation into cases, we easily get the six families of solutions (one for each coordinate flag):
 \begin{align*}
 &(a = 0, c = 0);&\quad 
 &(a = v_1, c = 0);&\quad 
 &(a = v_1, c = v_2);\\
 &(a = v_2, c = v_2);&\quad 
 &(a = 0, c = -v_1+v_2);&\quad
 &(a = v_2, c = -v_1+v_2).
 \end{align*}
 \end{example}

\begin{example}
Another natural family of examples are the Bott--Samelson resolutions of Schubert varieties (see~\cite{BS,Hansen,Demazure}). We first recall their construction here. Let $\Gs$ be a semisimple group of rank $r$, with simple roots $\alpha_1$, $\alpha_2$, \dots, $\alpha_r$. The reflections $s_1$, $s_2$, \dots, $s_r$ in the simple roots generate the Weyl group $\W$ of $\Gs$. Let $(e_i,f_i,h_i)$ be an $\ssl_2$-triple corresponding to $\alpha_i$. For any sequence $\underline{\omega}=(\alpha_{i_1},\alpha_{i_2},\dots,\alpha_{i_l})$ of simple roots, we can construct the \emph{Bott--Samelson variety} as follows:
$$X_{\underline{\omega}}
=\Ps_{i_1}\times_{\Bs}\Ps_{i_2}\times_{\Bs}\dots \times_{\Bs} \Ps_{i_l}/\Bs,
$$
where $\Bs$ is the Borel subgroup of $\Gs$ and $\Ps_i$ is the minimal (non-Borel) parabolic subgroup corresponding to the root $\alpha_i$. Here $\Bs$ acts on $\Ps_i$ both on the left and on the right; hence we can define the mixed quotients as above, and the last quotient is by the right $\Bs$-action on $\Ps_{i_l}$. The variety admits the multiplication map $X_{\underline{\omega}}\to \Gs/\Bs$. If $\underline{\omega}$ is a reduced word representing an element $\omega\in\W$, then this map is a resolution of the Schubert variety $X_\omega = \ov{\Bs\omega \Bs/\Bs}$. The Borel subgroup $\Bs$ acts on the Bott--Samelson variety on the left.

\begin{lemma}\label{botsamreg}
 The Bott--Samelson resolutions are regular $\Bs$-varieties.
\end{lemma}

\begin{proof}
Assume that an element $x\in X_{\underline{\omega}}$ represented by $(g_1,g_2,\dots,g_l)$ is fixed by the regular nilpotent $e$. As $e$ generates an additive subgroup $\exp(te)$ inside $\Bs$, every zero of $e$ is fixed by this subgroup, and in particular by $b_1 = u = \exp(e)$. This means that in the Bott--Samelson variety
\begin{align}\label{bs1}
[(b_1g_1,g_2,g_3,\dots,g_l)] = [(g_1,g_2,g_3,\dots,g_l)].
\end{align}
First, this means that $b_1 g_1 = g_1 b_2$ for some $b_2\in \Bs$, hence $b_1\in g_1 \Bs g_1^{-1}$. As $b_1$ is a regular unipotent element, there is only one Borel subgroup, namely $\Bs$, which contains $b_1$ (see the discussion in Example~\ref{exgr2}). As $N_{\Gs}(\Bs) = \Bs$, we have $g_1\in \Bs$. From $b_1 g_1 = g_1 b_2$ we have that $b_2$ is conjugate to $b_1$; hence it is also a regular unipotent in $\Bs$. From \eqref{bs1} we have
$$[(b_2g_2,g_3,\dots,g_l)] = [(g_2,g_3,\dots,g_l)]$$
in the Bott--Samelson variety corresponding to the sequence $(\alpha_{i_2},\dots,\alpha_{i_l})$. Applying the same reasoning, we get inductively that $g_1,g_2,\dots,g_l\in\Bs$, hence $[(g_1,g_2,g_3,\dots,g_l)] = [(1,1,\dots,1)]$ in $X_{\underline{\omega}}$.\footnote{We thank  Jakub L\"{o}wit for this argument.}
\end{proof}

\noindent This means that using the method above, we can determine $H_{\Ts}^*(X_{\underline{\omega}})$, where $\Ts$ is the maximal torus inside $\Bs$. The open Bia{\l}ynicki-Birula cell $X_o$ consists of the classes
$$[(\exp(x_1f_{i_1}),\exp(x_2f_{i_2}),\dots,\exp(x_lf_{i_l})]_{x_1,x_2,\dots,x_l\in\C},$$
and we would like to find the scheme $\ZZ$ inside $X_o\times \ttt$. We need to determine the infinitesimal action of $\Bs$ on that cell. We will proceed coordinate by coordinate. Note that for $i\in\{1,2,\dots,r\}$ the group $\Ps_{i}$ contains $\{\exp(t \cdot f_i): t\in\C\} \cdot \Bs$ as a dense subset. Therefore, for any $x\in\C$ there exists an open neighbourhood $U\subset \Bs$ of $1_B$ such that for all $g\in U$ we have
$$g \cdot \exp(x f_i) = \exp(b(g)f_i) h(g)$$
for some maps $b\colon U\to \C$ and $h\colon U\to \Bs$ with $b(1) = x$ and $h(1) = 1$. The two sides of the equality are functions of $g$. Let us differentiate them at $g=1$ in the direction of $y\in \bb$. We get
$$y \cdot \exp(x f_i) = \exp(x f_i) \cdot (Db|_1(y) f_i + Dh|_1(y)),$$
where on the left-hand side, the dot denotes the right translation by $\exp(x f_i)$ and on the right-hand side, it analogously denotes the left translation. Therefore, 
$$Db|_1(y) f_i + Dh|_1(y) = \Ad_{\exp(-x f_i)}(y).$$
Now let $y = e + w$, where $w\in\Ts$. Then
\begin{multline*}
\Ad_{\exp(-x f_i)}(y) = \Ad_{\exp(-x f_i)}(e) + \Ad_{\exp(-x f_i)}(w) = (e+xh_i-x^2f_i) + (w-\alpha_i(w)x f_i) \\
= (-\alpha_i(w)x-x^2)f_i + (e + w + xh_i).
\end{multline*}
Thus we get $Db|_1(y) = -\alpha_i(w)x-x^2$ and $Dh|_1(y) = e + w + xh_i$. Hence the infinitesimal action on $X_{\underline{\omega}}$ in the direction $e+w$ yields the vector of the first coordinate $-\alpha_{i_1}(w)x_1-x_1^2$ and induces
an infinitesimal action of $e + w + x_1h_{i_1}$ on the second coordinate. We can apply this procedure inductively and get that the $\supth{j}$ coordinate is acted upon by
$$e+w+\sum_{k=1}^{j-1} x_k h_{i_k}$$
and the corresponding coordinate of the vector field $V_{e+w}$ is
$$-\sum_{k=1}^{j-1} \alpha_{i_j}(h_{i_k})x_k x_j -\alpha_{i_j}(w)x_j-x_j^2.$$
Therefore, if we define the numbers $b_{jk} = \alpha_{i_j}(h_{i_k})$, we obtain the following presentation of the equivariant cohomology ring:
$$H_{\Ts}^*(X_{\underline{\omega}}) = 
\C[\ttt][x_1,x_2,\dots,x_l]/\left( x_j^2 = -\sum_{k<j} b_{jk} x_kx_j -\alpha_{i_j}(w)x_j \right),
$$
where $w$ denotes the $\ttt$ coordinate. Note that \textit{e.g.} for $\alpha_1, \dots, \alpha_r$ being the standard simple roots of $\SL_n$, those numbers vanish whenever $|i_j-i_k|>1$.

The variety has $2^l$ torus-fixed points, and hence the equivariant cohomology ring is a free module over $\C[\ttt]$ of rank $2^l$. An additive basis consists of all the square-free monomials in $x_1,x_2,\dots,x_l$. We then recover  the results obtained \textit{e.g.} in \cite[Proposition 4.2]{BS} or \cite[Proposition 3.7]{Willems}. 
\end{example}

\section{Reductive and arbitrary principally paired algebraic groups}
\label{redarbsec}

\subsection{Reductive groups}
In this section we will make a transition from solvable groups to reductive groups. We do that by restricting to Borel subgroups and utilizing Theorem~\ref{finsolv}.

Let  $\Gs$ be a complex reductive algebraic group of rank $r$. We assume that $e\in\geg = \Lie(\Gs)$ is a regular nilpotent element. Let $f, h\in\geg$ denote the remaining elements of an $\ssl_2$-triple $(e,f,h)$ (see the discussion in Section~\ref{sl2p}). In fact, all of the regular nilpotents are conjugate (see \cite[Section 3, Theorem 1]{Kostsec}). Hence, we can actually assume $e = x_1 + x_2 + \dots + x_s$, as in Example~\ref{exgr}. In particular, $h$ is semisimple and contained in the unique Borel subalgebra $\bb$ of $\geg$ containing $e$. It integrates to a map $H^t\colon \Cs\to \Gs$ with finite kernel. We denote by $\Ss = e+C_\geg(f)$ the corresponding Kostant section (\textit{cf.} \cite[Theorem 0.10]{Kostsec}). Kostant's theorem also gives $\C[\Ss] = \C[\geg]^\Gs = \C[\ttt]^\Ws = H^*_\Gs(\pt)$. The goal will be to prove the following result.

\begin{theorem}
\label{semisimp}
Let $\Gs$ be as above, and assume that $\Gs$ acts regularly on a connected smooth projective variety~$X$. Let $\ZZ_\Gs$ be the closed subscheme of $\Ss\times X$ defined as the zero set of the total vector field $($cf.\ Definition~\ref{totvec}\,$)$ restricted to $\Ss\times X$.
Then $\ZZ_\Gs$ is an affine, reduced scheme and $H^*_\Gs(X) \cong \C[\ZZ_\Gs]$ as graded $\C[\Ss]$-algebras, where the grading on right-hand side is defined by the action of\, $\Cs$ on $\Ss$ via $\frac{1}{t^2}\Ad_{H^t}$ and on $X$ via $H^t$. In other words, $\ZZ_\Gs = \Spec H^*_\Gs(X)$, $\Ss = \Spec H^*_\Gs$, and the pullback of functions along the projection $\ZZ_\Gs\to \Ss$ yields the structure map $H^*_\Gs\to H^*_\Gs(X)$, so we have the following diagram: 

$$ \begin{tikzcd}
 \ZZ_\Gs  \arrow{d}{\pi} \arrow{r}{\cong}  &
  \Spec H^*_\Gs(X;\C) \arrow{d} \\
   \Ss \arrow{r}{\cong}&
  \Spec H^*_\Gs\rlap{.}
\end{tikzcd}$$

Moreover, the isomorphism $H^*_\Gs(X) \cong \C[\ZZ_\Gs]$ of graded $\C[S]$-algebras is functorial  in both $X$ and $\Gs$. The admissible morphisms are those that map a $\Gs_1$-regular variety $X$ to $\Gs_2$-regular variety $Y$ in a $\Gs_1$-equivariant way, where $\Gs_1\to\Gs_2$ is a homomorphism between two reductive algebraic groups which maps the fixed principal $\ssl_2$-triple to the other fixed principal $\ssl_2$-triple.

$$ \begin{tikzcd}
 \ZZ_{\Gs_1}(X)  \arrow{d}{} \arrow{r}{\cong}  &
  \Spec H^*_{\Gs_1}(X;\C) \arrow{d} \\
 \ZZ_{\Gs_2}(Y)  \arrow{r}{\cong}&
  \Spec H^*_{\Gs_2}(Y;\C).
\end{tikzcd}$$

\end{theorem}

Note that $H^*_\Gs(X) = H^*_\Ts(X)^\W$, where $\Ts$ is the maximal torus and $\W = N_\Gs(\Ts)/\Ts$ is the Weyl group of $\Gs$ (see \textit{e.g.} \cite[Proposition III.1]{equiv}). Therefore, we will be able to make use of the result for solvable groups (Theorem~\ref{finsolv}).

\subsection{Motivating example: \texorpdfstring{$\Gs = \SL_2(\C)$}{G=SL2(C)}}

For $\Gs = \SL_2(\C)$ we can choose the canonical $e$, $f$, $h$:

$$e = \begin{pmatrix}
0 & 1 \\
0 & 0
\end{pmatrix}, \quad
f = \begin{pmatrix}
0 & 0 \\
1 & 0
\end{pmatrix}, \quad
h = \begin{pmatrix}
1 & 0 \\
0 & -1
\end{pmatrix}.
$$
Then we get $\Ss = \{e+vf : v\in\C \}$. Again, let us adapt the convention from Example~\ref{exgr} for the basis of $\ttt$; \textit{i.e.} a number $v\in\C$ will denote $-vh/2$. We know that $H_\Ts^*(X) = \C[\ZZ_{\Bs_2}]$, where $\ZZ_{\Bs_2}$ is defined as before for solvable (Borel) subgroup $\Bs_2$ of $\SL_2(\C)$ consisting of upper-triangular matrices. 
Let us now see how the Weyl group (in this case $\Sigma_2 = \{1, \epsilon\}$) acts on $H^*_\Ts(X)$. We have the following commutative diagram: 
$$
\begin{tikzcd}
H^*_\Ts(X) \arrow[r, "\epsilon^*"] \arrow[dd, "\iota_{\epsilon\zeta_i}^*"]
& H^*_\Ts(X) \arrow[dd, "\iota_{\zeta_i}^*"]
\\ \\
H^*_\Ts(\epsilon\zeta_i) \arrow[r, "\epsilon^*"]
& H^*_\Ts(\zeta_i)\rlap{.}
\end{tikzcd}
$$
Note that in the bottom row we have the (contravariant) action of $\W$ on $H^*_\Ts(\pt) \cong \C[\ttt]$, which is defined by the (covariant) adjoint action of $\W$ on $\ttt$. In the case of $\SL_2$, the element $\epsilon$ acts on $\ttt$ by $v\mapsto -v$.

Therefore, we get that for any $c\in H^*_\Ts(X)$ and any $\Ts$-fixed point $\zeta_i$, we have
$$(\epsilon^*c)|_{\zeta_i} = (c|_{\epsilon\zeta_i}) \circ \epsilon,$$
where $\epsilon$ is  seen as a map $\ttt\to\ttt$. This determines $\epsilon^*c$ completely as the restriction $H^*_\Ts(X)\to \bigoplus H^*_\Ts(\zeta_i)$ is injective.
Hence when we apply the isomorphism $\rho\colon H^*_\Ts(X)\to \C[\ZZ_{\Bs_2}]$, we  get
$$\rho(\epsilon^* c)(\w,M_{\w}\zeta_i) = \rho(c)(\epsilon \w,M_{\epsilon\w}\epsilon\zeta_i).$$
We get an algebra homomorphism $\C[\ZZ_{\Bs_2}]\to\C[\ZZ_{\Bs_2}]$, which has to come from a morphism $\ZZ_{\Bs_2}\to\ZZ_{\Bs_2}$. This morphism sends $M_{\w}\zeta_i$ to $M_{\epsilon\w}\epsilon\zeta_i$.
\\
We will now look at the adjoint action of elements of the form
$$\exp(sf) = \begin{pmatrix}
1 & 0 \\
s & 1
\end{pmatrix}
\in \SL_2.
$$
We have
\begin{align}
\label{conquo}
\Ad_{\exp(tf)}(e+th) &= e + t^2 f,  \\
\label{conmin}
\Ad_{\exp(2tf)}(e+th) &= e - th.
\end{align}
From \eqref{conmin} we infer (by Lemma~\ref{lemad}) that the map
$$\psi_\epsilon\colon (v,x) \longmapsto (-v,\exp(-vf) x)$$
is an isomorphism of $\ZZ_{\Bs_2}$ (note that in our choice of basis, the number $v$ denotes $-vh/2$). We claim that it is equal to the above (\textit{i.e.} it is dual to $\rho\circ \epsilon^*\circ \rho^{-1}$). Clearly, the action on the first factor agrees. Now we have
$$\exp(-vf)(M_v \zeta_i) = 
\begin{pmatrix}
1 & 0 \\
-v & 1
\end{pmatrix}
\begin{pmatrix}
1 & 1/v \\
0 & 1
\end{pmatrix}
\zeta_i, 
$$
and we get
$$M_{-v}^{-1}\exp(-vf)(M_v \zeta_i)=
\begin{pmatrix}
1 & 1/v \\
0 & 1
\end{pmatrix}
\begin{pmatrix}
1 & 0 \\
-v & 1
\end{pmatrix}
\begin{pmatrix}
1 & 1/v \\
0 & 1
\end{pmatrix}
\zeta_i
=
\begin{pmatrix}
0 & 1/v \\
-v & 0
\end{pmatrix}
\zeta_i = \epsilon\zeta_i.
$$
Therefore, 
$$\psi_\epsilon(M_v\zeta_i)= \exp(-vf)(M_v \zeta_i) = M_{-v}\epsilon\zeta_i $$
and indeed
$$\rho(\epsilon^*c)(v,x) = \rho(c)(\psi_\epsilon(v,x)).$$
Thus $\Spec H^*_{\SL_2(\C)}(X)$ is the GIT quotient of $\ZZ_{\Bs_2}$ over this action.

Now from \eqref{conquo} we get that the map $\phi\colon  (v,x)\mapsto (v,\exp(-vf/2)x)$ is an isomorphism between $\ZZ_{\Bs_2}$ and $\ZZ' = \{(v,x)\in \C\times X: (V_{e+v^2/4 f})|_{x} = 0\}$. Therefore, we might as well look for the GIT quotient of $\ZZ'$ by $\phi \circ \psi_\epsilon\circ\phi^{-1}$.
We get
$$\phi \circ \psi_\epsilon\circ\phi^{-1}(v,x) = \phi\circ\psi_\epsilon(v,\exp(vf/2)x) = \phi(-v,\exp(-vf/2)x) = (-v,x).$$
The GIT quotient of $\ZZ' = \{(v,x): (V_{e+v^2/4 f})|_{x} = 0\}$ by this action is clearly isomorphic to $\ZZ_\Gs = \{(t,x)\in \C\times X: (V_{e+tf})|_{x} = 0\}$.

\subsection{General case}
\label{secred}
We will want to mimic the proof for $\SL_2$ in the general reductive case. Let $\ZZ_\Bs$ be the scheme from Section~\ref{solvsec}, defined for the Borel subgroup $\Bs$ of $\Gs$. We need the following:
\begin{itemize}
 \item We need regular maps $A\colon \ttt\to \Gs$ and $\chi\colon \ttt\to \Ss$ that satisfy
 $$\Ad_{A(\w)}(e+\w) = \chi(\w),$$
 so that $(\id_{\ttt},A(\w))$ maps $\ZZ_\Bs$ to $\ZZ'$, where
 $$ \ZZ' = \{(\w,x)\in \ttt\times X: V_{\chi(\w)}|_{x} = 0\}.$$
 \item Moreover, we want $\chi$ to be $\W$-invariant and induce an isomorphism $\ttt/\!\!/\W\to \Ss$, so that we can construct 
 $\ZZ_{\Gs}$ as a quotient of $\ZZ'$.
 \item We want to realise the Weyl group action on $\ZZ_{\Bs}$ by the action on the second factor; \textit{i.e.} for each $\eta\in\W$ we want to define a map $B_{\eta}\colon \ttt\to \Gs$ such that
 $$(\w,x) \longmapsto (\eta(\w), B_\eta(\w) \cdot x)$$
 is the action of Weyl group.
 \item If we conjugate above with the isomorphism $\ZZ_{\Bs}\to\ZZ'$, we want to get a map that fixes the $X$-coordinate. In other words,
 $$A(\eta\w) B_\eta(\w) A^{-1}(\w) = 1;$$
 \textit{i.e.} $B_\eta(\w) = A(\eta\w)^{-1} A(\w)$.
\end{itemize}
We will now formalise these ideas. First, let $\Bs$ be the unique Borel subgroup of $\Gs$ containing the regular nilpotent $e$ (\textit{cf.} Section~\ref{sl2p}). We denote by $\U$ the corresponding maximal unipotent subgroup and by $\Bs^-$, $\U^-$ the opposite Borel and unipotent subgroup. Let $\bb$, $\uu$, $\bb^-$, $\uu^-$ denote the corresponding Lie algebras. As above, by $\ZZ_\Bs\subset \ttt\times X$ we denote the zero scheme defined by the action of $\Bs$, which by Theorem~\ref{finsolv} is isomorphic to $\Spec H^*_\Ts(X)$. We start by finding the map $A$. We know from Lemma~\ref{lembal} that the map
$$ \Ad_{-}(-)\colon  \U^- \times \Ss \lra e+\bb^-$$
is an isomorphism. Let us consider the preimage of $e+\ttt$ and denote by $A(\w)\in \U^-$, $\chi(\w)\in \Ss$ the elements such that
 \begin{equation}
 \label{adchi}
 \Ad_{A(\w)}(e+\w) = \chi(\w).
 \end{equation}
 We then know  from Proposition~\ref{isoquot} and \eqref{adchi} that the map $\phi$ defined as
 $$\phi(\w,x) = (\w,A(\w)x)$$
 is an isomorphism from
 $$\ZZ_\Bs = \{(\w,x)\in \ttt\times X: V_{e+\w}|_{x} = 0\}$$
 to 
 $$\ZZ' = \{(\w,x)\in \ttt\times X: V_{\chi(\w)}|_{x} = 0\}.$$
 Moreover, let 
 $$B_\eta(\w) = A(\eta\w)^{-1} A(\w)$$
 for any $\eta\in\W$, $\w\in \Ts$. Then the map $\psi_\eta$ defined as
 $$\psi_\eta = \phi^{-1} \circ (\eta,\id) \circ\phi,$$
 \textit{i.e.} $\psi_\eta(\w,x) = (\eta\w, B_\eta(\w)x)$, is an automorphism of $\ZZ_\Bs$. Here $\eta$ is seen as a map $\ttt\to\ttt$.

\begin{lemma}
\label{lemweyl}
 The map $\psi_\eta$ defines the action of\, $\W$ on $\ZZ_\Bs$. In other words, $\W$ acts on the right on $H^*_\Ts(X)$, and the dual left action on $\ZZ_\Bs$ is defined by $\psi$.
\end{lemma}

\begin{proof}
 For any $\eta\in\W$ we have the commutative diagram
 $$
\begin{tikzcd}
H^*_\Ts(X) \arrow[r, "\eta^*"] \arrow[dd, "\iota_{\eta\zeta_i}^*"]
& H^*_\Ts(X) \arrow[dd, "\iota_{\zeta_i}^*"]
\\ \\
H^*_\Ts(\eta\zeta_i) \arrow[r, "\eta^*"]
& H^*_\Ts(\zeta_i)\rlap{.}
\end{tikzcd}
$$
 In the bottom row both entries are isomorphic to $\C[\ttt]$, and the map is precomposition with $\eta\colon \ttt\to\ttt$.
 Therefore, for any $c\in H^*_\Ts(X)$ and any $\Ts$-fixed point $\zeta_i$, we get
 $$(\eta^*c)|_{\zeta_i} = (c|_{\eta\zeta_i}) \circ \eta.$$
 This determines $\eta^*c$ completely as the restriction $H^*_\Ts(X)\to \bigoplus H^*_\Ts(\zeta_i)$ is injective. We want to determine what this action of $\W$ induces on $\ZZ_\Bs$. The action of $\Ws$ on $\ZZ_\Bs$, which we will denote by $\eta\mapsto \eta_*$, has to satisfy the equality
 $$\rho(c)(\eta_* (\w,x)) = \rho(\eta^*(c))(\w,x).$$
 From the proof of Lemma~\ref{lemcm}, we know that $\ZZ_\Bs\cap (\ttt^\reg\times X)$ is dense in $\ZZ_\Bs$. Therefore, to determine $\eta^*$, it is enough to determine its values $\eta_*(\w,x)$ for $\w$ regular. In this case if $(\w,x)\in\ZZ_\Bs$, then by Section~\ref{structure} we have that $\w = M_\w \zeta_i$, where $M_\w\in \Bs$ is such that $\Ad_{M_\w}(\w) = e+\w$. Then $\rho(c)(\w,x) = c|_{\zeta_i}(\w)$. Hence 
 $$\rho(\eta^*(c))(\w,x) = \eta^*(c)|_{\zeta_i}(\w) = (c|_{\eta\zeta_i}) (\eta\w)
 = \rho(c)(\eta\w,M_{\eta\w}\eta\zeta_i).$$
 Thus 
 $$\eta_*(\w,M_{\w}\zeta_i) = (\eta\w,M_{\eta\w}\eta\zeta_i).$$
 We claim that $\eta_* = \psi_\eta$, \textit{i.e.} $B_\eta(\w) M_{\w}\zeta_i = M_{\eta\w}\eta\zeta_i$. We have to prove that $C_{\eta,\w} = M_{\eta\w}^{-1} B_\eta(\w) M_{\w}$ sends $\zeta_i$ to $\eta\zeta_i$. 

  Note that
 \begin{multline*}
 \Ad_{C_{\eta,\w}}(\w) = \Ad_{M_{\eta\w}^{-1} B_\eta(\w) M_{\w}}(\w) = 
 \Ad_{M_{\eta\w}^{-1} A_{\eta\w}^{-1}A_{\w}M_{\w}}(\w) = 
 \Ad_{M_{\eta\w}^{-1} A_{\eta\w}^{-1}A_{\w}}(e+\w) \\
 =
 \Ad_{M_{\eta\w}^{-1} A_{\eta\w}^{-1}}(\chi(\w))= 
 \Ad_{M_{\eta\w}^{-1} A_{\eta\w}^{-1}}(\chi(\eta\w))= 
 \Ad_{M_{\eta\w}^{-1}} (e + \eta\w)= \eta\w.
 \end{multline*}
 Therefore, for any representative $\tilde{\eta} \in N_{\Gs}(\Ts)$ of $\eta$, we have
 $$\tilde{\eta}^{-1}C_{\eta,\w} \in C_{\Hs}(\w).$$
 As $\w$ is regular, its centraliser within $\he$ is just $\ttt$. It is the Lie algebra of $C_{\Gs}(\w)$, which is connected by \cite[Corollary 3.11]{Steinberg}, hence equal to $\Ts$. Therefore, $\tilde{\eta}^{-1}C_{\eta,\w}\in \Ts$; hence $C_{\eta,\w}$ represents the class of $\eta$ in $N_\Gs(\Ts)/\Ts$. Thus $C_{\eta,\w}$ sends $\zeta_i$ to $\eta\zeta_i$, as we wanted to prove.
\end{proof}

\begin{proof}[Proof of Theorem~\ref{semisimp}]
 We saw above that the map $\phi$ defined as $\phi(\w,x) = (\w,A(\w)x)$ is an isomorphism from $\ZZ_\Bs$ to $\ZZ' = \{(\w,x)\in \ttt\times X: V_{\chi(\w)}|_{x} = 0\}$. Then we can conjugate the maps $\psi_\eta$ with this isomorphism,  getting maps $\phi\circ\psi_\eta\circ\phi^{-1}\colon \ZZ'\to \ZZ'$. We have
 \begin{multline}
 \label{actw}
 \phi\circ\psi_\eta\circ\phi^{-1}(\w,x) = \phi\circ\psi_\eta(\w,A(\w)^{-1}x)
 = \phi(\eta\w,B_\eta(\w)A(\w)^{-1}x) \\ 
 = (\eta\w,A(\eta\w)B_\eta(\w)A(\w)^{-1}x) = (\eta\w,x).
 \end{multline}
The last equality follows from the definition $B_\eta(\w) = A(\eta\w)^{-1}A(\w)$. By Lemma~\ref{lemweyl} the map $\phi\circ\psi_\eta\circ\phi^{-1}$ gives the action of $\W$ on $\ZZ'\cong \ZZ_\Bs \cong \Spec H^*_\Ts(X)$.
 
 We have $H^*_\Gs(X) = H^*_\Ts(X)^\W$ and therefore $\Spec H^*_\Gs(X) = \Spec H^*_\Ts(X)/\!\!/\W =\ZZ'/\!\!/\W$. But we know from \eqref{actw} that $\W$ acts only on the $\ttt$-coordinate of $\ZZ'$ and moreover from Proposition~\ref{kostiso} that the map $\chi$ induces an isomorphism $\ttt/\!\!/\W\to\Ss$. Therefore, 
 $$\Spec H^*_\Gs(X) = \ZZ'/\!\!/W = 
 \{(\w,x)\in \ttt\times X: V_{\chi(\w)}|_{x} = 0\}/\!\!/\W =
 \{(v,x)\in \Ss\times X: V_{v}|_{x} = 0\} = \ZZ_\Gs.
 $$
 The zero scheme $\ZZ_\Gs$ is reduced because $\ZZ' \cong \ZZ_{\Bs}$ is reduced by Lemma~\ref{lemcm}. The agreement of $\C[\Ss]$-algebra structures follows from the commutativity of the diagram
 $$
\begin{tikzcd}
\Spec H^*_\Ts(X) = \ZZ_{\Bs} \arrow[r, "\cong"]  \arrow[dd, "\pi_{\Bs}"] & \ZZ' \arrow[r, "/\!\!/\W"]
& \Spec H^*_\Gs(X) = \ZZ_\Gs \arrow[dd, "\pi_\Gs"]
\\ \\
\Spec H^*_\Ts = \ttt \arrow[rr, "/\!\!/\W"]
& & \Spec H^*_\Gs = \ttt/\!\!/\W
\end{tikzcd}
$$
and the analogous statement for $\Bs$ in Theorem~\ref{finsolv}.

It remains to show that the grading agrees on the two sides. We know from Theorem~\ref{finsolv} that the grading in the solvable case is defined by the weights of the torus acting on $\ttt\times X$ by $\left(\frac{1}{t^2},H^t\right)$. We have to prove that it descends to the action by $\left(\frac{1}{t^2}\Ad_{H^t},H^t\right)$. But we have
$$\Ad_{A(\w)}(e+\w) = \chi(\w)$$
and thus
$$\Ad_{H^tA(\w)H^{-t}}(\Ad_{H^t}(e+\w)) = \Ad_{H^t}(\chi(\w)), $$
and dividing both sides by $t^2$ gives
$$\Ad_{H^tA(\w)H^{-t}}\left(e+\frac{\w}{t^2}\right) = \frac{1}{t^2}\Ad_{H^t}(\chi(\w)).$$
However, 
$$H^tA(\w)H^{-t}\in \U^-,\qquad \frac{1}{t^2}\Ad_{H^t}(\chi(\w))\in \Ss.$$ 
The latter follows from $\frac{1}{t^2}\Ad_{H^t}(e) = e$ and $\Ad_{H^t}$ preserving the centraliser of $f$, as $\Ad_{H^t}(f) = \frac{1}{t^2}f$. Therefore, by uniqueness we have
$$H^tA(\w)H^{-t} = A\left(\frac{\w}{t^2}\right), \qquad \frac{1}{t^2}\Ad_{H^t}(\chi(\w)) = \chi\left(\frac{\w}{t^2}\right).$$
The quotient map $\ZZ_\Bs\to \ZZ_\Gs$ sends $(\w,x)$ to $(\chi(\w),A(\w)x)$. And by the above it sends $t\cdot(\w,x) = \left(\frac{\w}{t^2},H^t x\right)$ to 
$$\left(\chi\left(\frac{\w}{t^2}\right),A\left(\frac{\w}{t^2}\right)H^t x\right) = \left(\frac{1}{t^2}\Ad_{H^t}\left(\chi(\w)\right),H^tA(\w)H^{-t} H^t x\right) = \left(\frac{1}{t^2}\Ad_{H^t}\left(\chi(\w)\right),H^t A(\w) x\right),$$
which proves that the action of $\Cs$ on $\ZZ_{\Bs}$ descends to the action by $\left(\frac{1}{t^2}\Ad_{H^t},H^t\right)$ on $\ZZ_{\Gs}$.

The functoriality follows immediately from the functoriality for $\Bs$ (\textit{cf.} Propositions~\ref{funcprop} and~\ref{funcgrp}).
\end{proof}

\begin{remark}\label{affred}
 We know $C_\geg(f)\subset \bb^-$ (see Section~\ref{sl2p}) and all the weights of the $\Hs^t$ action on $\bb^-$ are non-positive (\textit{cf.} Lemma~\ref{posint}). Therefore, the argument as in Lemma~\ref{lemcm} shows that $\ZZ_\Gs$ lies in $\Ss\times X_o$. This means that for any computations we have to consider only an affine part $X_o$ of $X$.
\end{remark}

\begin{remark}\label{chernred}
 In the spirit of Lemma~\ref{nicefun}, we can determine in the reductive case too what functions on $\ZZ_\Gs$ the particular Chern classes are mapped to. Assume that $\Ee$ is a $\Gs$-linearised vector bundle on $X$. Let $k$ be a non-negative integer, and consider $c_k^\Gs(\Ee)\in H^*_\Gs(X) = H^*_\Ts(X)^\Ws$. If we first consider the map $\rho\colon H^*_\Ts(X)\to \C[\ZZ_\Bs]$ from Section~\ref{sectionmap}, then from Lemma~\ref{nicefun} we know for any $(\w,x)\in \ZZ_\Bs$ that 
 $$\rho(c_k^\Ts(\Ee))(\w,x) = \Tr_{\Lambda^k \Ee_x}(\Lambda^k (e+\w)_x).$$
 The map $\phi$ defined as
 $$\phi(\w,x) = (\w,A(\w)x)$$
 maps $\ZZ_\Bs$ isomorphically to $\ZZ'$. Then $c_k^\Ts(\Ee)$ defines on $\ZZ'$ the function $\rho(c_k^\Ts(\Ee))\circ\phi^{-1}$ which satisfies
 $$\rho(c_k^\Ts(\Ee))\circ\phi^{-1}(\w,y) = \rho(c_k^\Ts(\Ee))(\w,A(\w)^{-1}y)
 = \Tr_{\Lambda^k \Ee_{A(\w)^{-1}y}}(\Lambda^k (e+\w)_{A(\w)^{-1}y}).
 $$
 As $\Ee$ is $\Gs$-invariant, this is equal to
 $$\Tr_{\Lambda^k \Ee_{y}}(\Lambda^k \Ad_{A(\w)}(e+\w)_{y}) = 
 \Tr_{\Lambda^k \Ee_{y}}(\Lambda^k \chi(\w)_{y}). 
 $$
 This means that on the quotient $\ZZ_\Gs$ the function $\rho_\Gs(c_k^\Gs(\Ee))$ corresponding to $c_k^\Gs(\Ee)$ satisfies
 $$\rho_\Gs(c_k^\Gs(\Ee))(v,x) = \Tr_{\Lambda^k \Ee_{x}}(\Lambda^k v_{x}).$$
\end{remark}
\noindent Let us also note that the $\Gs$-equivariant Chern classes generate the equivariant cohomology ring.

\begin{lemma}\label{genred}
 In the setting above, the $\Gs$-equivariant cohomology $H^*_\Gs(X)$ is generated as a $\C[\ttt]^\Ws$-algebra by the equivariant Chern classes of\, $\Gs$-equivariant vector bundles.
\end{lemma}

\begin{proof}
  By the Nakayama lemma it is enough to prove (see the proof of Lemma~\ref{chern2}) that the non-equivariant cohomology $H^*(X)$ is generated by Chern classes of $\Gs$-equivariant vector bundles.
 
 We know from the proof of Lemma~\ref{chern1} that $H^*(X)$ is generated by Chern characters of $\Ts$-equivariant coherent sheaves. For any such sheaf $\F$, we can consider the ``averaged'' sheaf $\F_\Ws = \frac{1}{|\Ws|} \bigoplus_{\eta\in \Ws} \eta_* \Ws$. As the group $\Gs$ is connected, for any $g\in \Gs$ we have $\ch(g_* \F) = \ch(\F)$, hence $\ch(\F_\Ws) = \ch(\F)$. Therefore, $H^*(X)$ is generated by Chern characters of $N_\Gs(\Ts)$-equivariant coherent sheaves. Then again by \cite[Corollary~5.8]{Thomason} it is generated by Chern characters of $N_\Gs(\Ts)$-equivariant vector bundles. Every $N_\Gs(\Ts)$-equivariant vector bundle is a $\Ws$-invariant element of $K_\Ts(X)$. However, we know by \cite[Corollary 6.7]{HLS} that $K_\Ts(X)^\Ws = K_\Gs(X)$; hence $H^*(X)$ is generated by Chern classes of $\Gs$-equivariant vector bundles.
\end{proof}

\subsection{Examples} \label{examples}
We finish this section by providing examples for Theorem~\ref{semisimp}. These are extensions of the examples above for Theorem~\ref{finsolv}.

\begin{figure}[ht!]
\begin{center}
 \subfloat{
  \includegraphics[width=7cm]{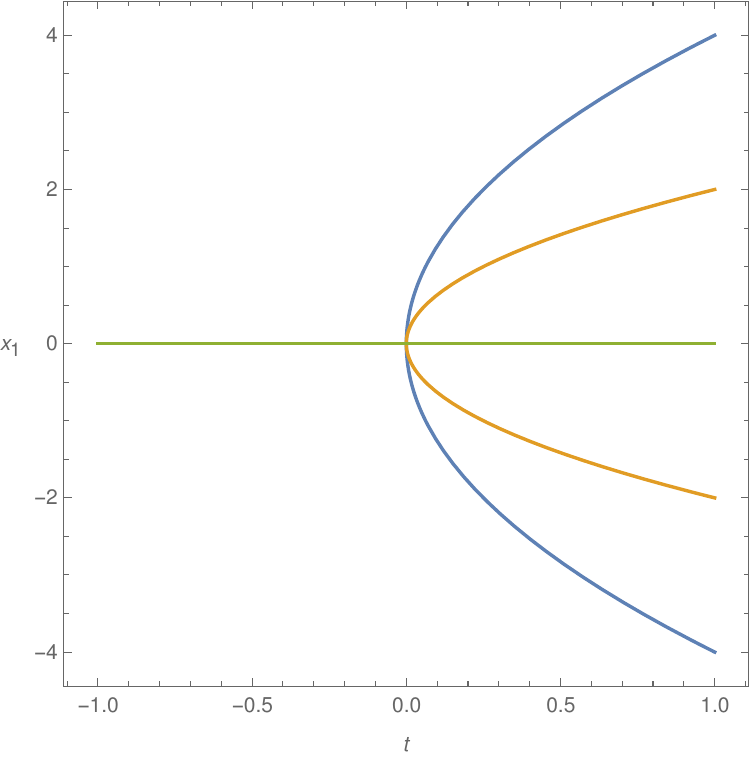}}
  \hfill
\subfloat{
  \includegraphics[width=7cm]{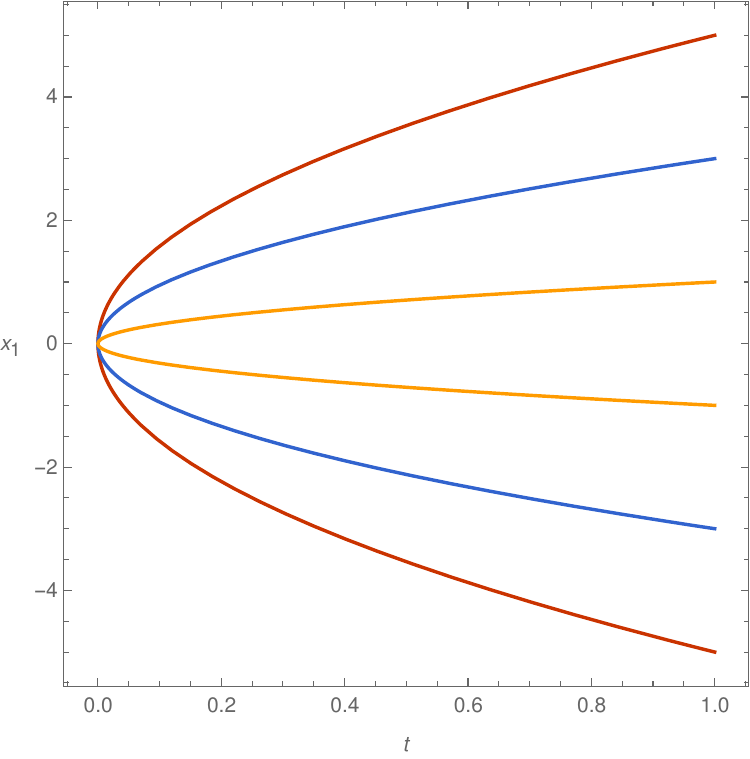}}
\end{center}
\caption{$\Spec H_{\SL_2(\C)}^*(\PP^4)$ and $\Spec H_{\SL_2(\C)}^*(\PP^5)$.}
\label{sl2pn}
\end{figure}

\begin{example}\label{exslpn}
 We continue Example~\ref{exsl22}. There, we found the $\Cs$-equivariant cohomology of $\PP^n$. Now, using the tools above, we can also find $\Spec H_{\SL_2(\C)}(\PP^n)$. We know that the map $(v,x)\mapsto (v,(I+vf)x)$ maps the zeros of $V_{e+vh}$ isomorphically to the zeros of $V_{e+v^2 f}$. The former form the subscheme cut out by $x_1(x_1+2v)(x_1+4v)\cdots(x_1+2nv) = 0$ in the $(v,x_1)$-plane. Note that 
 $$f = \begin{pmatrix}
       0 & 0 & 0 & 0 & \dots & 0\\
       1\cdot n & 0 & 0 & 0 & \dots & 0\\
       0 & 2\cdot (n-1) & 0 & 0 & \dots & 0\\
       \vdots & \vdots & \vdots & \vdots & \ddots & \vdots \\
       0 & 0 & (n-1)\cdot 2 & 0 & \dots & 0 \\
       0 & 0 & 0 & n\cdot 1 & \dots & 0
      \end{pmatrix}; 
 $$
 hence the map $I+vf$ acts on the $x_1$-coordinate by adding $nv$. This means that the zeros of $V_{e+v^2 f}$ are defined by 
 $$(x_1-nv)(x_1-(n-2)v)(x_1-(n-4)v)\cdots(x_1+(n-2)v)(x_1+nv) = 0.$$
 Bringing the symmetric factors together, we get
 $$
 \begin{cases}
  (x_1^2-n^2v^2)(x_1^2-(n-2)^2v^2)\cdots(x_1^2-4v^2)x_1 = 0 
  &\text{ for $n$ even,}\\
  (x_1^2-n^2v^2)(x_1^2-(n-2)^2v^2)\cdots(x_1^2-9v^2)(x_1^2-v^2)= 0 
  &\text{ for $n$ odd.}\\
 \end{cases}
 $$
 Therefore, 
 $$
 H_{\SL_2(\C)}^*(\PP^n)=
 \begin{cases}
  \C[t,x_1]/\big((x_1^2-n^2t)(x_1^2-(n-2)^2t)\cdots(x_1^2-4t)x_1\big) &\text{ for $n$ even,}\\
  \C[t,x_1]/\big((x_1^2-n^2t)(x_1^2-(n-2)^2t)\cdots(x_1^2-9t)(x_1^2-t)\big) &\text{ for $n$ odd.}
 \end{cases}
 $$
 The scheme has $\lceil\frac{n+1}{2}\rceil$ components, one for each orbit of the action of $\W=\Z/2\Z$ on $(\PP^n)^{\Cs}$. The parabolas correspond to two-element orbits, and the line (for even $n$) corresponds to the unique fixed point of $\Cs$ fixed by $\W$. It is equal to $[\underbrace{0:0:\dots:0}_{n/2}:1:\underbrace{0:\dots:0:0}_{n/2}]$. Examples of the scheme for $n=4$ and $n=5$ are depicted in Figure~\ref{sl2pn}.
\end{example}

\begin{figure}[ht!]
\begin{center}
\subfloat{
  \includegraphics[width=7cm]{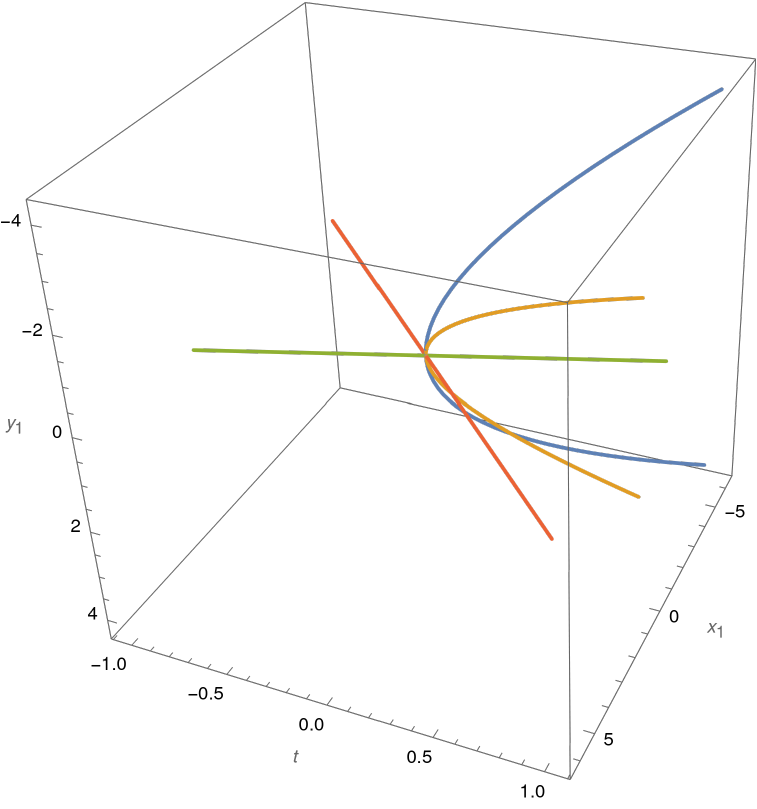}}
  \hfill
\subfloat{
  \includegraphics[width=7cm]{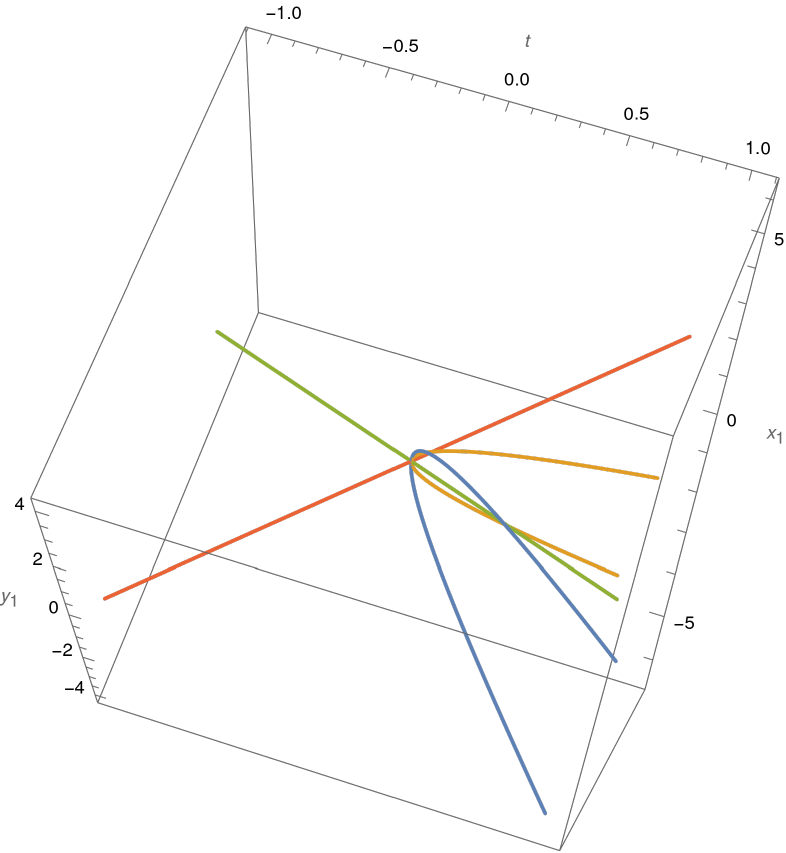}}
\end{center}
\caption{Two different views of $\Spec H_{SL_2(\C)}^*(\Gr(2,4))$.}
\label{sl2gr}
\end{figure}

\begin{example}
 We continue Example~\ref{csgr}. The principal subgroup $\SL_2(\C)\subset \SL_4(\C)$  acts on $\Gr(2,4)$. One can check that
 $$V_f|_{x_1,y_1,x_2,y_2}
 = (-3y_1,4,3x_1-3y_2,3y_1).$$
 Then
 $$V_{e+tf}|_{x_1,y_1,x_2,y_2}
 =(x_2-x_1y_1 -3ty_1,
 -x_1-y_1^2+y_2 +4t,
 -x_1y_2 +3tx_1 -3ty_2,
 -x_2-y_1y_2 +3ty_1).
 $$
 As before, from the first two equations of $V_{e+tf} = 0$, we can determine $x_2$ and $y_2$, so $\Spec H_{\SL_2(\C)}^*(\Gr(2,4))$ can be embedded in $\C[t,x_1,y_1]$. Its equations are
 $$12t^2+4tx_1-x_1^2-3ty_1^2-x_1y_1^2 = 0,\quad
 y_1(4t-2x_1-y_1^2) = 0.$$
 By considering two possibilities in the latter, one easily arrives at four possibilities:
 $$
 (x_1 = -2t, y_1 = 0),
 \quad
 (x_1 = 6t, y_1 = 0),
 \quad
 (x_1 = -6t, y_1^2 = 16t),
 \quad
 (x_1 = 0, y_1^2 = 4t).
 $$
 As in the previous example, the components correspond to orbits of $\W$ acting on $\Gr(2,4)^{\Cs}$. The former two correspond to one-element orbits, \textit{i.e.} $\{\Span(e_2,e_3)\}$ and $\{\Span(e_1,e_4)\}$, and the latter come from two two-element orbits. The scheme embedded in $(t,x_1,y_1)$-space is presented in Figure~\ref{sl2gr}.
\end{example}

\begin{example}
 We consider an example for a group of higher rank, $\SL_3(\C)$, that we can still draw. Let it act on $\PP^2$ in the standard way. In Example~\ref{exsl3p2} we calculated the equivariant cohomology of $\PP^2$ with respect to a ($2$-dimensional) torus. Now we will compute the $\SL_3$-equivariant cohomology. The Kostant section is
 $$\Ss = \left\{
 \begin{pmatrix}
 0 & 1 & 0 \\
 c_2 & 0 & 1 \\
 c_3 & c_2 & 0
 \end{pmatrix}: c_2,c_3\in\C
 \right\}. 
 $$
The coordinates $c_2,c_3\in \C[\Ss]\cong H^*(B\SL_3(\C))$ are (up to scalar multiples) the universal Chern classes of principal $\SL_3(\C)$-bundles, or equivalently, of rank $3$ vector bundles with trivial determinant.
 We have already computed that $V_e|_{x_1,x_2} = (x_2-x_1^2,-x_1x_2)$. Then it is easy to see that for 
 $$M = \begin{pmatrix}
 0 & 1 & 0 \\
 c_2 & 0 & 1 \\
 c_3 & c_2 & 0
 \end{pmatrix}$$
 we have
 $V_M|_{x_1,x_2} = (x_2-x_1^2+c_2,-x_1x_2+c_2x_1+c_3)$. As before, we can eliminate $x_2$ by substituting from the first equation, and we get the equation $x_1^3-2c_2x_1-c_3 = 0$. The corresponding scheme $\Spec H^*_{\SL_3(\C)}(\PP^2)$ in the coordinates $c_2$, $c_3$, $x_1$ is illustrated in Figure~\ref{hsl3p2}. It is irreducible as all three torus-fixed points lie in one orbit of the Weyl group. The projection to the $(c_2,c_3)$-plane is generically a $3-1$ map.
 On the right-hand side of Figure~\ref{hsl3p2}, the slice $c_3 = 0$ is marked in red. The elements of $\Ss$ that satisfy $c_3 = 0$ form the Kostant section of the principal $\SL_2$ subgroup -- which acts as in Example~\ref{exslpn}. Therefore, the red scheme is equal to $\Spec H^*_{\SL_2}(\PP^2)$. Additionally, the functoriality of Theorem~\ref{semisimp} implies that restriction to $c_3 = 0$ yields the base restriction map
 $$H^*_{\SL_3}(\PP^2)\lra H^*_{\SL_2}(\PP^2).$$
\end{example}

\begin{figure}[ht!]
\begin{center}
\subfloat{
  \includegraphics[width=7.5cm]{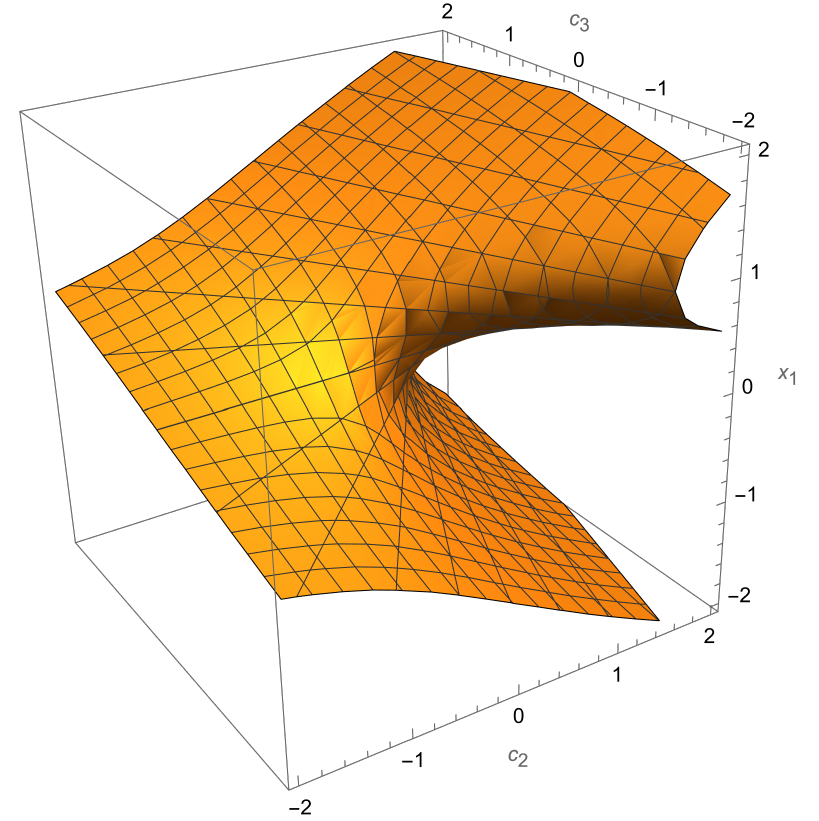}}
  \hfill
\subfloat{
  \includegraphics[width=7.5cm]{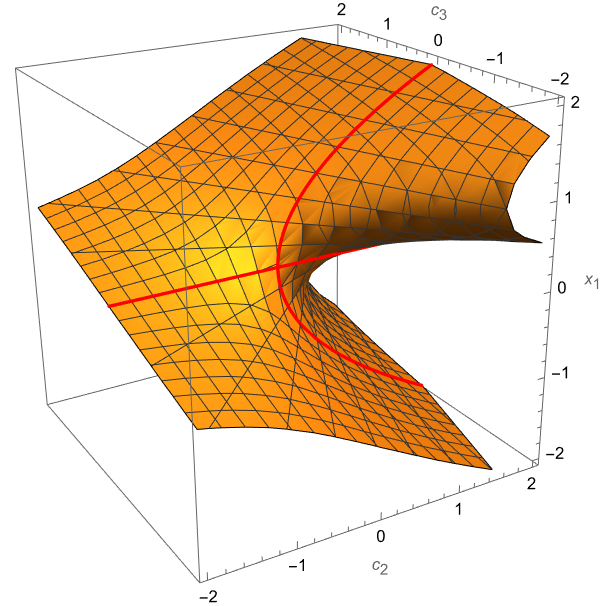}}
\end{center}
\caption{$\Spec H_{\SL_3(\C)}^*(\PP^2)$. On the right the subscheme $\Spec H_{\SL_2(\C)}^*(\PP^2)$ is marked. Compare with Figure~\ref{sl2pn}.}
\label{hsl3p2}
\end{figure}

\begin{example}
 As in Example~\ref{exflag} we now consider the action of $\SL_3(\C)$ on the variety $F_3$ of full flags in $\C^3$. We determined $V_e$ in Example~\ref{exflag}. We can analogously determine the vector fields corresponding to lower-triangular matrices. Then for 
 $$M = \begin{pmatrix}
 0 & 1 & 0 \\
 c_2 & 0 & 1 \\
 c_3 & c_2 & 0
 \end{pmatrix}$$ we easily get
 $$V_M|_{a,b,c} =
 (-a^2+b+c_2,-ab+ac_2+c_3,-b+ac-c^2+c_2).$$
 Plugging in $b$ from the first equation, we obtain
 $$a^3-2c_2a+c_3 = 0,\quad
 a^2-ac+c^2 = 2c_2.$$
 The first equation for $a$ clearly coincides with the equation for $x_1$ from the previous example. One can easily see that the equations mean that $a$ and $-c$ are two of the three roots of the polynomial $x^3-2c_2x+c_3=0$. The map to the $(c_2,c_3)$-plane is generically $6-1$.
 As all the torus-fixed points, \textit{i.e.} coordinate flags, lie in one orbit of the Weyl group, in the GIT quotient of $\Spec(H_\Ts^*(F_3))$, they are joined together and the scheme is irreducible.
\end{example}

\subsection{Principally paired algebraic groups}\label{secarb}
In fact, we can prove the equivalent of Theorem~\ref{semisimp} for a principally paired, but not necessarily reductive, algebraic group. This version will yield a common generalisation to Theorems~\ref{semisimp} and~\ref{finsolv}. Note that $H_\Hs^* = \C[\ttt]^\Ws = \C[\Ss]$ -- see the comment above Theorem~\ref{restkos}. We will prove the following.

\begin{theorem}\label{general}
 Assume that $\Hs$ is a principally paired algebraic group which acts regularly on a smooth projective variety $X$. Let $\ZZ_\Hs$ be the closed subscheme of $\Ss\times X$, defined as the zero set of the total vector field $($cf.\ Definition~\ref{totvec}\,$)$ restricted to $\Ss\times X$.
 
 Then $\ZZ_\Hs$ is an affine reduced scheme, and $H_{\Hs}^*(X)\cong \C[\ZZ_\Hs]$ as graded $\C[\Ss]$-algebras, where the grading on the right-hand side is defined on $\Ss$ via $\frac{1}{t^2}\Ad_{H^t}$ and on $X$ by the action of $\Cs$ via $H^t$. In other words, $\ZZ_\Hs = \Spec H_\Hs^*(X)$, $\Ss = \Spec H_\Hs^*$, and the projection $\ZZ_\Hs\to \Ss$ yields the structure map $H_\Hs^*\to H_\Hs^*(X)$. This isomorphism is functorial as in Theorem~\ref{semisimp}.

$$ \begin{tikzcd}
 \ZZ_\Hs  \arrow{d}{\pi} \arrow{r}{\cong}  &
  \Spec H^*_\Hs(X;\C) \arrow{d} \\
   \Ss \arrow{r}{\cong}&
  \Spec H^*_\Hs.
\end{tikzcd}$$
\end{theorem}

\begin{remark}
 As $\Ns$ is contractible, the Levi subgroup $\Ls\subset \Hs$ is a homotopy retract of $\Hs$, and for any $\Hs$-space $X$ we have $H_\Hs^*(X) = H_\Ls^*(X)$. In particular, if $\Hs$ is solvable, we have $H_\Hs^*(X) = H_\Ts^*(X)$, where $\Ts$ is a maximal torus within $\Hs$. This explains how the theorem above generalises Theorem~\ref{finsolv}.
\end{remark}

\begin{proof}[Proof of Theorem~\ref{general}.]
We will proceed analogously to the proof in Section~\ref{secred}. We follow the notation from Section~\ref{kostgensec}. In particular, $\Bs$ is the Borel subgroup of $H$ such that its Lie algebra $\bb$ contains $e$. We first consider the scheme $\ZZ_\Bs\subset \ttt\times X$, defined as in Section~\ref{solvsec}, \textit{i.e.} the zero scheme of the total vector field on $\geg\times X$, restricted to $(\ttt+e)\times X$.
Then from Lemma~\ref{congen} we get morphisms $A\colon \ttt\to\U^-$, $\chi\colon \ttt\to\Ss$ such that
$$\Ad_{A(\w)}(e+\w) = \chi(\w),$$
so that $(\id,A(\w))$ maps $\ZZ_\Bs$ to $\ZZ'$, where
$$\ZZ' = \{(\w,x)\in \ttt\times X: V_{\chi(\w)}|_{x} = 0\}.$$
In fact, $A$ and $\chi$ are exactly the same as in Section~\ref{secred} (see the proof of Lemma~\ref{congen}). In particular, $\chi$ induces the isomorphism $\ttt/\!\!/\Ws\to\Ss$.
For any regular $\w\in\ttt$, we have the element $M_w\in \Bs$ such that $\Ad_{M_w}(w) = e+w$. Just like in Lemma~\ref{lemweyl}, for any $\eta\in \Ws$ we get that for any regular $\w$ the element
$C_{\eta,\w} = M_{\eta(\w)}^{-1} A(\eta\w)^{-1}A(\w) M_{\w}$ is in the class of $\eta$ in $N_\Ls(\Ts)/\Ts$. Note that here $M_{\eta(\w)}\in \Bs_l$.

This then proves,  as in Section~\ref{secred}, that the Weyl group action on $\ZZ_\Bs$, when transported to $\ZZ'$, is defined by $\eta\mapsto (\eta,\id)$. And then as $\chi:\ttt/\!\!/\Ws \to \Ss$ is an isomorphism, we get that
$$\ZZ_\Hs \cong \ZZ_{\Bs}/\!\!/\Ws = \Spec H^*_\Ts(X)/\!\!/\Ws = \Spec H^*_\Hs(X).$$ 

We have to prove that the grading on $\C[\ZZ_\Hs]$ defined by the grading on $H^*_\Hs(X)$ agrees with the one described in the theorem. We know that in the solvable case the grading is defined by the action of $\Cs$ on $\ZZ_\Bs$ via $\left(\frac{1}{t^2},H^t\right)$ (\textit{cf.} Definition~\ref{defz}). Just like in the reductive case, we need to prove that under the quotient by $\Ws$, it descends to the action by $\left(\frac{1}{t^2}\Ad_{H^t},H^t\right)$. The argument for reductive groups does not translate exactly, as \textit{a priori} we do not know whether $H^t$ preserves the centraliser of $f$. However, we know that $H_l^t$, the one-parameter subgroup generated by $h_l$, does.

On the other hand, as $[h,e] = [h_l,e] = 2e$, from Lemma~\ref{semcen} we infer $h-h_l\in Z(\lel)$. As in the proof of Theorem~\ref{semisimp}, we have
$$\Ad_{H^tA(\w)H^{-t}}\left(e+\frac{\w}{t^2}\right) = \frac{1}{t^2}\Ad_{H^t}(\chi(\w))$$
and
$$H^tA(\w)H^{-t}\in \U^-,\qquad \frac{1}{t^2}\Ad_{H^t}(\chi(\w))\in \Ss,$$ 
where now the latter follows from $\frac{1}{t^2}\Ad_{H^t}(e) = e$ and $\Ad_{H^t} = \Ad_{H_l^t}$ preserving the centraliser of $f$ as $\Ad_{H_l^t}(f) = \frac{1}{t^2}f$. Therefore, we have
$$H^tA(\w)H^{-t} = A\left(\frac{\w}{t^2}\right), \qquad \frac{1}{t^2}\Ad_{H^t}(\chi(\w)) = \chi\left(\frac{\w}{t^2}\right),$$
and the same reasoning follows. This proves Theorem~\ref{general}.
\end{proof}

\begin{example}
Basic examples of non-reductive, non-solvable linear groups are parabolic subgroups of reductive groups. Let us consider such a group $\Ps\subset \Gs$, where $\Gs$ is reductive, and assume that $\Bs\subset \Ps$ is a Borel subgroup of $\Gs$ contained in $\Ps$. Then we can consider a principal $\bb(\ssl_2)$-triple $(e,f,h)$ in $\geg$ such that $e,h\in \bb$. This makes $\Ps$ into a principally paired group, and we can make use of Theorem~\ref{general}.

Suppose that $X$ is a Schubert variety in some partial flag variety $\Gs/\Qs$. Its stabiliser $\Ps$ in $\Gs$ contains $\Bs$; hence it is a parabolic subgroup. In general it is larger than $\Bs$ (see more in \cite[Section 2]{SanVan}). Remember that $\Bs$ acts regularly on $\Gs/\Qs$  (\textit{cf.}~Example~\ref{exgr2}). Therefore, if $X$ is smooth, Theorem~\ref{general} gives the $\Ps$-equivariant cohomology of $X$.
\end{example}

\begin{example}
As in the previous example, assume that $X$ is Schubert variety in $\Gs/\Qs$ fixed by $\Ps$. One can then construct a Bott--Samelson resolution of $X$, \textit{cf.} \cite[Section 2, p. 446]{SanVan}, which is $\Ps$-equivariant. As in Lemma~\ref{botsamreg} such a resolution will be a smooth regular $\Ps$-variety. Hence we can use Theorem~\ref{general} to compute its $\Ps$-equivariant cohomology.
\end{example}

We also extend Lemma~\ref{genred} to principally paired groups.

\begin{lemma}\label{genprinc}
 Assume that a principally paired group $\Hs$ acts regularly on a smooth projective variety $X$. Then the $\Hs$-equivariant cohomology $H^*_\Hs(X)$ is generated as a $\C[\ttt]^\Ws$-algebra by equivariant Chern classes of\, $\Hs$-equivariant vector bundles.
\end{lemma}

\begin{proof}
 As in the proof of Lemma~\ref{genred}, it is enough to prove that the non-equivariant cohomology $H^*(X)$ is generated by Chern classes of $\Hs$-equivariant vector bundles. By Lemma~\ref{genred} we know that it is generated by Chern classes, and in fact by Chern characters of $\Ls$-equivariant vector bundles for $\Ls$ being a Levi subgroup of $\Hs$. We have the following two maps: 
 $$ K_\Hs^0(X) \lra K_\Ls^0(X) \xrightarrow{\;\ch\;} H^*(X). $$
 The former is simply the restriction of $\Hs$-equivariant bundles to $\Ls$-equivariant bundles, and the latter is the Chern character map. We know that the image of composition generates the whole $H^*(X)$. We prove that the first map is, in fact, an isomorphism, which will prove the claim.

 By \cite[Proposition~6.2]{Thomason} the restriction along $X\to \Hs\times^\Ls X$ induces an isomorphism
$$K_\Hs^0(\Hs\times^\Ls X)\lra K_\Ls^0(X).$$
Now $\Hs\times^\Ls X$ maps $\Hs$-equivariantly to $X$ (simply by $[(h,x)] \mapsto hx$), and we will show that this map induces an isomorphism on $K_\Hs^0$.\footnote{This argument is based on a suggestion by Andrzej Weber.} Let $\Ns$ be the unipotent radical of $\Hs$, so that $\Hs = \Ns \rtimes \Ls$. Notice that we then have the $\Hs$-equivariant isomorphism
$$\Hs\times^\Ls X \simeq \Ns \times X,$$
where $\Hs$ acts on $\Ns\times X$ diagonally by conjugation and action.

Indeed, every element of $\Hs$ is uniquely decomposed as $ul$ for $u\in \Ns$, $l\in \Ls$. This means that $\Hs\times^\Ls X \simeq \Ns\times X$. Now we need to see how  $\Hs$ acts on this product. Note that in $\Hs\times^\Ls X$ we have
$$h\cdot[(u,x)] = [(hu,x)] = [(huh^{-1},hx)],$$
and as $huh^{-1}\in \Ns$, upon identification with $\Ns\times X$ we have $h\cdot(u,x) = (huh^{-1},hx)$.

We want to prove that the map $\Ns\times X\to X$ induces an isomorphism on $K_\Hs^0$. Note that it is not the projection but the action of $N$ on $X$. However, we can split it into the isomorphism $N\times X\to N\times X$ given by $(u,x)\mapsto (u,ux)$, and the projection. Note that this isomorphism is in fact $\Hs$-invariant as 
$$h\cdot (u,ux) = (huh^{-1}, hux) = (huh^{-1}, huh^{-1} hx).$$
Therefore, we have to show that the projection $\Ns\times X \to X$ yields an isomorphism on~$K_\Hs^0$.

Now by \cite[Proposition 14.32]{Milne} the algebraic exponential map for the unipotent group $\exp\colon \nen\to \Ns$ is an isomorphism of schemes. Thus in fact $\Ns\times X \simeq \nen\times X$ has the structure of a (trivial) vector bundle over $X$. Note that $\Hs$ acts on it linearly. Indeed, we have $h\exp(v) h^{-1} = \exp(h v h^{-1})$, and the adjoint representation of $\Hs$ on $\nen$ is linear. Then by \cite[Theorem~4.1]{Thomason} the projection $\Ns\times X\to X$ gives an isomorphism on~$K_\Hs^0$.
\end{proof}

\section{Extensions: Singular varieties and total zero schemes}

In this section we discuss two directions to extend our results. 
First we discuss generalisations to singular varieties.

\subsection{Singular varieties}
\label{secsing}

Our main Theorem~\ref{general} may be generalised to singular varieties, in the spirit of \cite[Section 7]{BC}. A sufficient condition will be the existence of
an embedding in a smooth regular variety such that the corresponding map on ordinary cohomology is surjective (\textit{cf.} Corollary~\ref{surj}).

\begin{proposition} \label{sing}
 Assume that $\Hs$ is a principally paired algebraic group, and let $\Ss$ be the Kostant section within $\Hs$, as defined in Section~\ref{secarb}. Let $\Bs$ be a Borel subgroup of $\Hs$. Assume that $\Hs$ acts regularly on a smooth projective variety $X$, and let $\ZZ_\Hs^X$ be the zero scheme defined in Theorem~\ref{general} for the $\Hs$-action on $X$.
 
 Assume $Y$ is a closed $\Hs$-invariant subvariety whose cohomology is generated by Chern classes of\, $\Bs$-linearised vector bundles. Then analogously to Section~\ref{secarb}, we can define an isomorphism of graded $\C[\Ss]$-algebras $H_\Hs^*(Y)\to\C[\ZZ_\Hs^Y]$, where $\ZZ_\Hs^Y$ is the reduced intersection $\ZZ_\Hs^Y = \ZZ_\Hs^X\cap (\Ss\times Y)$. The isomorphism makes the diagram
\begin{equation}\label{diagsing}
\begin{tikzcd}
H^*_\Hs(X) \arrow[r] \arrow[dd]
& H^*_\Hs(Y) \arrow[dd]
\\ \\
\C[\ZZ_\Hs^X] \arrow[r]
& \C[\ZZ_\Hs^Y]
\end{tikzcd}
\end{equation}
 commutative.
 The assumption on the cohomology of $Y$ holds in particular if the inclusion $Y\to X$ induces a surjective map $H^*(X)\to H^*(Y)$ on ordinary cohomology.
\end{proposition}

\begin{proof}
The proof is essentially the same as in \cite[Section 7]{BC}. We only sketch it here. First assume  that $\Hs$ is solvable. Every point of the variety $\ZZ_\Hs^Y$ is of the form $(\w,M_\w \zeta)$, where $M_\w\in\Hs$ depends on $\w$ and $\zeta$ is a $\Ts$-fixed point contained in $Y$. Therefore, for any $c\in H^*_\Ts(Y)$ we can define $\rho_Y(c)$ (we only localise to points in $Y$). The condition on the cohomology of $Y$ allows us to use Lemma~\ref{nicefun} to show that $\rho_Y$ actually maps $H^*_\Ts(Y)$ to $\C[\ZZ_\Hs^Y]$. The injectivity follows again from the injectivity of localisation on equivariantly formal spaces (\textit{cf.} \cite[Theorem 1.2.2]{GKM}). The diagram is obviously commutative, and the surjectivity then follows from the surjectivity of the restriction $\C[\ZZ_\Hs^X]\to\C[\ZZ_\Hs^Y]$ to closed subvariety.

Now assume that $\Hs$ is arbitrary principally paired group. Let $\Bs$ be its Borel subgroup, and by $\ZZ_\Bs^Y$ denote the appropriate zero scheme defined for $\Bs$ acting on $Y$. As $Y$ is $\Hs$-invariant, the arguments from the proof of Theorem~\ref{general} show that $\C[\ZZ_\Hs^Y] = \C[\ZZ_\Bs^Y]^{\Ws}$, and the conclusion follows.
The last line of the proposition is implied by Lemma~\ref{chern2}.
\end{proof}

\begin{example}
  Let $\Hs = \Bs$, the Borel subgroup of a reductive group $\Gs$. Natural examples of singular regular $\Bs$-varieties are Schubert varieties in the flag variety $\Gs/\Bs$ or any other subvarieties that are unions of Bruhat cells; see \cite[Theorem 5 with remarks]{ACL}.
  In general, Schubert varieties are stabilised by parabolic subgroups (see \cite[Section 2]{SanVan}). Those are therefore singular $\Ps$-regular varieties for parabolic groups $\Ps$.
\end{example}

\begin{example}
Assume that $X = \Gs/\Bs$ is the flag variety of type A, hence $\Gs = \SL_m(\C)$. Then if $Y$ is any Springer fiber within $X$, the restriction on cohomology $H^*(X)\to H^*(Y)$ is surjective (see \cite{KumPro}); hence Proposition~\ref{sing} also holds in that case.
\end{example}
However, there exist $\Gs$-invariant subvarieties for which the restriction map on cohomology is not surjective.

\begin{remark}
 The assumption on the surjectivity on the cohomology of $Y$ is necessary in  Proposition~\ref{sing}. Consider the following. Let $\SL_2$ act on $\PP^3$ as in Example~\ref{exslpn}. It comes from a representation $\Sym^3 V$, where $V$ is the fundamental representation of $\SL_2$. It has two extreme (highest and lowest) weights and two ``middle'' weights. The point $o$ of $\PP^3$ which represents the highest-weight space is fixed by the Borel subgroup of upper-triangular matrices, and hence one sees that its orbit is isomorphic to the full flag variety $\SL_2/\Bs_2\cong \PP^1$. However, if we consider a point $p\in \PP^3$ representing a non--highest-weight space, its stabiliser is a torus; \textit{i.e.}~$\Stab_{\SL_2}(p)\cong\Ts$. Hence its $\SL_2$-orbit is not closed. We denote its closure by $Y:=\overline{\SL_2\cdot p}$. We claim that $Y$ is not smooth. We can see this directly, by noticing that it is the projectivised variety of polynomials $a_0 x^3 + a_1 x^2y + a_2 xy^2 + y^3$ with at least two equal roots (vanishing lines), and writing down the discriminant equation. We can also see this using our results. If $Y$ were smooth, by Corollary~\ref{surj} the map $H^*(\PP^3)\to H^*(Y)$ would be surjective, but both varieties admit an action of $\Ts$, with the same set of fixed points; therefore, the map would have to be an isomorphism. This is impossible for dimensional reasons ($H^6(Y) = 0$).
 
 Moreover, not only is $Y$ singular, but in any case the map $H^*(\PP^3)\to H^*(Y)$ cannot be surjective. This would mean that Proposition~\ref{sing} applies. However, as all the $\Ts$-fixed points are already in $Y$, one sees immediately that $\ZZ$ is already in $Y$. Then again, we would have $H^*(Y) = H^*(\PP^3)$, which is impossible for the same reason as above. Thus $H^*(\PP^3)\to H^*(Y)$ is not surjective, and moreover $H^*(\PP^3)$ is not generated by Chern classes of $\Bs_2$-equivariant bundles. This shows that the assumption is necessary in Proposition~\ref{sing}.
\end{remark}

\begin{remark}
 Assume we are given an $\Hs$-invariant subvariety $Y$ of a regular smooth $\Hs$-variety $X$. By Proposition~\ref{sing} and Corollary~\ref{surj}, the surjectivity of the restriction $H^*(X)\to H^*(Y)$ is necessary and sufficient for the existence of an isomorphism $H^*_\Hs(Y)\to \C[\ZZ_\Hs^Y]$ which makes \eqref{diagsing} commutative. Carrell and Kaveh prove in \cite[Theorem 2]{CarKav}, for the case of $\Hs = \Bs_2$, that this is equivalent to $H_\Ts^*(Y)$ being generated by Chern classes of $\Bs_2$-equivariant bundles.
\end{remark}

\subsection{Total zero scheme}
\label{sectot}

Assume that $\Gs$ is a principally paired algebraic group, \textit{e.g.} $\Gs$ reductive. We proved in Theorem~\ref{general} how to see geometrically the spectrum of $\Gs$-equivariant cohomology of $X$ for $\Gs$ acting regularly on a projective variety $X$. However, this needed a choice -- of a concrete $\bb(\ssl_2)$-pair $(e,h)$ and the associated Kostant section. A natural challenge would be to try to find equivariant cohomology as global functions on a scheme that does not depend on choices. 

\begin{definition}\label{totzer}
Let an algebraic group $\Gs$ act on a smooth projective variety $X$. Consider the total vector field on $\geg\times X$ (\textit{cf.} Definition~\ref{totvec}). We call its reduced zero scheme
$$\ZZ_{\tot}\subset \geg\times X$$
the {\em total zero scheme}.
\end{definition}
Now we are ready to show the following.

\begin{theorem}\label{wholeg}
 Assume that $\Gs$ is principally paired. Let it act regularly on a smooth projective variety. Consider the action of\, $\Cs$ on the total zero scheme $\ZZ_{\tot}$ by $t\cdot(v,x) = \left(\frac{1}{t^2} v,x\right)$ and the action of\, $\Gs$  by $g\cdot (v,x) = (\Ad_g(v),g\cdot x)$. Then the ring $\C[\ZZ_{\tot}]^\Gs$ of\, $\Gs$-invariant functions on $\ZZ_{\tot}$ is a graded algebra over $\C[\geg]^\Gs\cong H^*_{\Gs}(\pt)$ isomorphic to $H^*_{\Gs}(\pt)$, where the grading comes from the weights of the $\Cs$-action on $\C[\geg]^\Gs$:
 $$ \begin{tikzcd}
  \C[\ZZ_{\geg}]^\Gs \arrow{r}{\cong} &
  H_\Gs^*(X;\C)  \\
  \C[\geg]^\Gs \arrow{r}{\cong} \arrow{u}&
  H^*_\Gs \arrow{u}.
\end{tikzcd}$$
 \end{theorem}
 
Following the notation from Theorem~\ref{general}, we show that the restriction $\C[\ZZ_\tot]^\Gs \to \C[\ZZ_\Gs]$ is an isomorphism, so that we get the following diagram: 
$$
\begin{tikzcd}
\C[\ZZ_\tot]^\Gs \arrow[r] 
& \C[\ZZ_\Gs] \arrow[r]
& H^*_{\Gs}(X,\C) 
\\ \\
\C[\geg]^\Gs \arrow[r]\arrow[uu]
& \C[\ttt]^\Ws \arrow[r]\arrow[uu]
& H^*_{\Gs}(\pt,\C)\rlap{,}\arrow[uu]
\end{tikzcd}
$$
with all horizontal arrows being isomorphisms. The bottom row follows from Lemma~\ref{genrest}. First we prove that the restriction is an epimorphism.

\begin{lemma}\label{surjtot}
 Under the assumptions of Theorem~\ref{wholeg}, the restriction $\C[\ZZ_\tot]^\Gs \to \C[\ZZ_\Gs]$ is surjective.
\end{lemma}

\begin{proof}
By Lemma~\ref{genprinc} we know
that $\C[\ZZ_\Gs]$ is generated over $\C[\ttt]^\Ws \cong \C[\geg]^\Gs$ by functions $\rho_\Gs(c_k^\Gs(\Ee))$ for positive integers $k$ and $\Gs$-equivariant vector bundles $\Ee$. Those functions are defined by
$$\rho_\Gs(c_k^\Gs(\Ee))(v,x) = \Tr_{\Lambda^k \Ee_{x}}(\Lambda^k v_{x});$$
see Remark~\ref{chernred}. For each such function, we can consider the regular function $f_{k,\Ee}$ defined on $\ZZ_\tot$ by its values:
$$f_{k,\Ee}(v,x) = \Tr_{\Lambda^k \Ee_{x}}(\Lambda^k v_{x}).$$
It is clearly $\Gs$-invariant and restricts to $\rho_\Gs(c_k^\Gs(\Ee))$ on $\ZZ_\Gs$. As $\C[\ZZ_\Gs]$ is generated by such functions, the conclusion follows.
\end{proof}

For the injectivity, let us start with an easy intermediate step. Let $\ZZ_\reg$ be the open subscheme of $\ZZ_\tot$ consisting of the part over $\geg^\reg\subset \geg$ (hence a closed subscheme of $\geg^\reg\times X$). Then we have the following. 
 
\begin{lemma}\label{inj1}
 Let $\Gs$ be a principally paired algebraic group. Assume it acts on a connected smooth projective variety, not necessarily regularly. The restriction $\C[\ZZ_\reg]^\Gs\to \C[\ZZ_\Gs]$ is injective, where $\ZZ_\reg$ and $\ZZ_\Gs$ are defined as above, as zero schemes over $\geg^\reg$ and $\Ss$.
\end{lemma}

\begin{proof}
As $\ZZ_\reg$ is reduced, a function is determined by its values on closed points. 
By Lemma~\ref{kostarb} every $\Gs$-orbit in $\geg^\reg$ intersects $\Ss$; thus the $\Gs$-orbit of any closed point in $\ZZ_\reg$ intersects $\ZZ_\Gs$. It is therefore enough to specify a $\Gs$-invariant function on $\ZZ_\reg$ on closed points of $\ZZ_\Gs$. The result follows.
\end{proof}

To finish the proof, we are only left with proving the  injectivity of the restriction $\C[\ZZ_\tot]^\Gs\to \C[\ZZ_\reg]^\Gs$. We will utilise the following lemma to prove that the restriction $\C[\ZZ_\tot]\to\C[\ZZ_\reg]$ is injective.

\begin{lemma}\label{projinj}
 Let $Y$ be a reduced scheme over a field $k$. Assume $Z$ is a closed subvariety and every closed point $p\in Y$ is contained in a projective closed subvariety that intersects $Z$. Then the restriction map on regular functions $k[Y]\to k[Z]$ is injective.
\end{lemma}

\begin{proof}
Let us assume that $f\in k[Y]$ vanishes on $Z$. Consider any closed point $p\in Y$. Let $A_p$ be a projective closed subvariety that contains $p$ and intersects $Z$ in a closed point $q$. Then $f|_{A_p}$ is a regular function on a projective variety over $k$; hence it is has constant value on all closed points of $A_p$. As $f(q) = 0$, this means that it takes the value $0$. Therefore, $f(p)=0$. Hence $f$ vanishes on every closed point.

As $Y$ is reduced and of finite type over a field, we know that regular functions are uniquely determined by their values on closed points. Therefore, $f = 0$.
\end{proof}

To arrive at the lemma's assumptions, we first prove slightly stronger versions of Lemmas~\ref{lemfix} and~\ref{lemzer}, under the condition that the action of the Lie algebra is integrable.

\begin{lemma}\label{grpfix}
 Assume that a solvable algebraic group $\Hs$ acts on a smooth complex variety $X$. Let $P\subset X$ be a projective irreducible component of the reduced zero scheme of a linear subspace $\V\subset \he$. Then $P$ contains a simultaneous zero of $N_\he(\V)$.
\end{lemma}

\begin{proof}
By \cite[Lemma 7.4]{Borel} we have $N_{\he}(\V) = \Lie(N_\Hs(\V))$. Let $N_\Hs(\V)^{o}$ be the connected component of unity within $N_\Hs(\V)$. We know from Lemma~\ref{lemad} that $N_\Hs(\V)$ preserves the zero set of $\V$. Thus $N_\Hs(\V)^{o}$ preserves its irreducible components, in particular $P$. By the Borel fixed point theorem, \textit{cf.}  \cite[Corollary~17.3]{Milne}, $N_\Hs(\V)^{o}$ must have a fixed point $p\in P$. Then its Lie algebra $ \Lie(N_\Hs(\V)^{o}) = \Lie(N_\Hs(\V)) = N_{\he}(\V)$ vanishes on $p$.
\end{proof}

\begin{lemma}\label{grpzer}
 Assume that an algebraic group $\Hs$ acts on a smooth variety $X$. Let $d,n \in \he$ commute, and assume that the Lie subalgebra generated by $[\he,\he]$ and $n$ is nilpotent. Let $P$ be a projective irreducible component of the reduced zero scheme of $j = d+n$. Then $P$ contains a simultaneous zero of $C_{\he}(d)$, in particular, a zero of any abelian subalgebra of\, $\he$ containing $d$.
\end{lemma}

\begin{proof}
 By restricting to the connected component of the identity, we can assume that $\Hs$ is connected. As $[\he,\he]$ is nilpotent, $\he$ must be solvable; hence $\Hs$ is solvable too.
 
 Let $\kek$ be the Lie subalgebra generated by $[\he,\he]$ and $n$. By Lemma~\ref{grpfix} we get that inside $P$ there is a zero $p$ of $N_{\he}(\C\cdot j)$, which in particular contains $d$ and $n$. As $P$ is irreducible, any irreducible component of the simultaneous zero set of $d$ and $n$ which contains $p$ is completely contained in $P$. Let $P_1\subset P$ be one such irreducible component. As it is closed inside $P$, it also has the structure of a projective scheme.
 
 We will first show that $P_1$ contains a simultaneous zero of $C'(d) = C_{\he}(d)\cap \kek$. As $\kek$ is nilpotent, $C'(d)$ is as well.  By Lemma~\ref{grpfix}, $P_1$ contains a simultaneous zero of $N_{\he}(\Span_\C(d,n))$, hence in particular of $N_{C'(d)}(\C\cdot n)$. Note that by definition everything in $C'(d)$ centralises $d$. As $P_1$ consists of zeros of $d$, it will contain an irreducible component $P_2$ of the common fixed point set of $d$ and $N_{C'(d)}(\C\cdot n)$. As a closed subscheme of $P_1$, $P_2$ is also projective. By the same argument, $P_2$ contains a projective irreducible component $P_3$ of the common fixed point set of $d$ and $N^2_{C'(d)}(\C\cdot n)$. As in the proof of Lemma~\ref{lemzer}, there exists a positive integer $k$ such that $N^k_{C'(d)}(\C\cdot n) = C'(d)$; hence we get a projective irreducible component $P_{k+1}$ of the common fixed point set of $d$ and $C'(d)$. But again as in Lemma~\ref{lemzer}, $C'(d)$, as well as $d$, is normalised by $C_{\he}(d)$. Hence inside $P_{k+1}$ there is a zero of $C_{\he}(d)$.
\end{proof}

\begin{lemma}\label{inj2}
Let $\Gs$ be a principally paired algebraic group. Assume that it acts on a connected smooth projective variety $X$, not necessarily regularly. The restriction $\C[\ZZ_\tot]^\Gs\to \C[\ZZ_\Gs]$ is injective, where $\ZZ_\reg$ and $\ZZ_\Gs$ are defined as before, as zero schemes over $\geg^\reg$ and $\Ss$.
\end{lemma}

\begin{proof}
We have the sequence of restrictions $\C[\ZZ_\tot]^\Gs \to \C[\ZZ_\reg]^\Gs\to \C[\ZZ_\Gs]$. By Lemma~\ref{inj1} we only need to prove that the first map is injective. Obviously, the restriction $\C[\overline{\ZZ_\reg}] \to \C[\ZZ_\reg]$ is injective (here we take the closure of $\ZZ_\reg$ in $\ZZ_\tot$). We will prove that $\C[\ZZ_\tot]\to \C[\overline{\ZZ_\reg}]$ is injective, and this will prove the lemma.

We employ Lemma~\ref{projinj} for this. We have to prove that every point of $\ZZ_\tot$ is contained in a projective subvariety which intersects $\overline{\ZZ_\reg}$. Let $(v,p) \in \ZZ_\tot \subset \geg\times X$. Then $p$ is contained in the zero scheme of the vector field $V_v$, hence in some irreducible component $P$ thereof. It is a closed subscheme of $X$; hence it is projective. Then we have $\{v\}\times P \subset \ZZ_\tot$ as a projective closed subvariety. Let $v = d + n$ be the Jordan decomposition of $v$ (\textit{cf.} \cite[Theorem~4.4]{Borel}). As $d$ and $n$ commute, they are contained in a Lie algebra $\bb$ of some Borel subgroup $\Bs\subset \Gs$. Let $\Ts$ be a maximal torus within $\Bs$ such that $d\in\ttt = \Lie(\Ts)$. Then from Lemma~\ref{grpzer} (take $\Hs = \Bs$) we get that $P$ contains a simultaneous zero $x$ of $\ttt$. It is also a zero of $v$; hence we have $(\ttt+\C\cdot v)\times \{x\} \subset \ZZ_\tot$. Note that $\ttt$ contains a regular element, and as the regular elements within $\geg$ form an open subset, the regular elements of $\ttt+\C\cdot v$ form an open non-empty subset, so they are dense. This means that $(\ttt+\C\cdot v)\times \{x\}\subset \overline{\ZZ_\reg}$, hence in particular $(v,x)\in\overline{\ZZ_\reg}$, and $(v,x)\in \{v\}\times P$, where $\{v\}\times P$ is a projective subvariety of $\ZZ_\tot$; therefore, we are done.
\end{proof}

\begin{proof}[Proof of Theorem~\ref{wholeg}]
The isomorphism follows from Lemmas~\ref{surjtot} and~\ref{inj2}.

For the grading we just have to show that the defined action of $\Cs$ descends under the restriction $\C[\ZZ_\tot]^\Gs \to \C[\ZZ_\Gs]$ to the action defined in Theorem~\ref{general}. Let $f$ be a $\Gs$-invariant function on $\ZZ_\tot$. Then for any $t\in\Cs$ the pullback $t^*f$ of $f$ by $t$ is defined by
$$t^*f(v,x) = f\left(\frac{1}{t^2} v,x\right).$$
As $f$ is $\Gs$-invariant, this means
$$t^*f(v,x) = f\left(\frac{1}{t^2} \Ad_{H^t}v,H^t x\right).$$
When we restrict to $\ZZ_\Gs$, the group $\Cs$ acts precisely by $\left(\frac{1}{t^2} \Ad_{H^t},H^t\right)$ (\textit{cf.} Theorem~\ref{general}). Therefore, the actions agree.
\end{proof}

\begin{example}
Assume that $\Gs$ is a reductive group acting on a partial flag variety $X = \Gs/\Ps$. Then the zero scheme is 
\[\tilde{\geg}_\Ps:=\{(x,\p^\prime)\in \geg\times\Gs/\Ps:x\in\p^\prime\},\]
which agrees with the partial Grothendieck--Springer resolution. Thus  as a $\C[\geg]^\Gs\cong H^*_\Gs$-module, the ring of invariant functions $\C[\tilde{\geg}_\Ps]^\Gs$ is equal to $H^*_\Gs(\Gs/\Ps) = H^*_\Ps = \C[\ttt]^{\Ws_\Ps}$.
\end{example}

\begin{figure}[ht!]
\begin{center}
 \subfloat{
  \includegraphics[width=7cm]{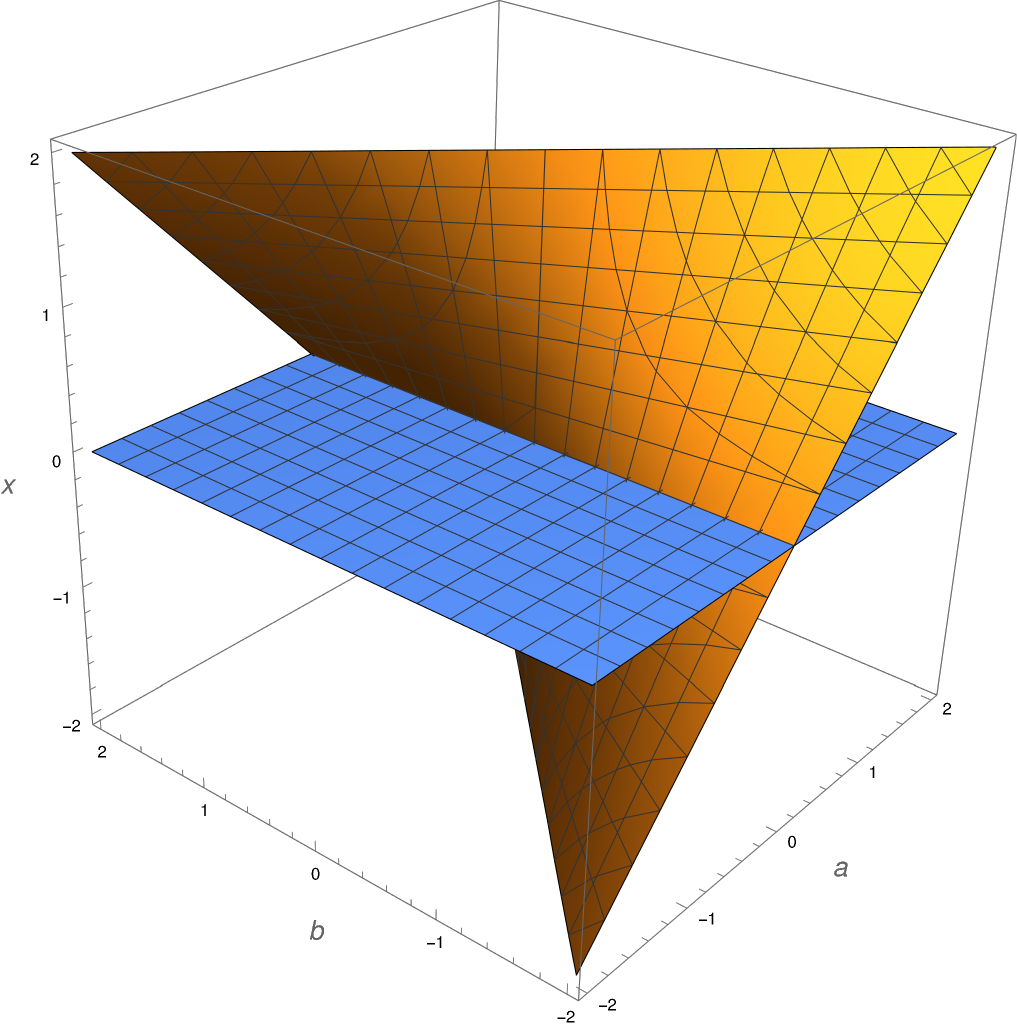}}
  \hfill
\subfloat{
  \includegraphics[width=7cm]{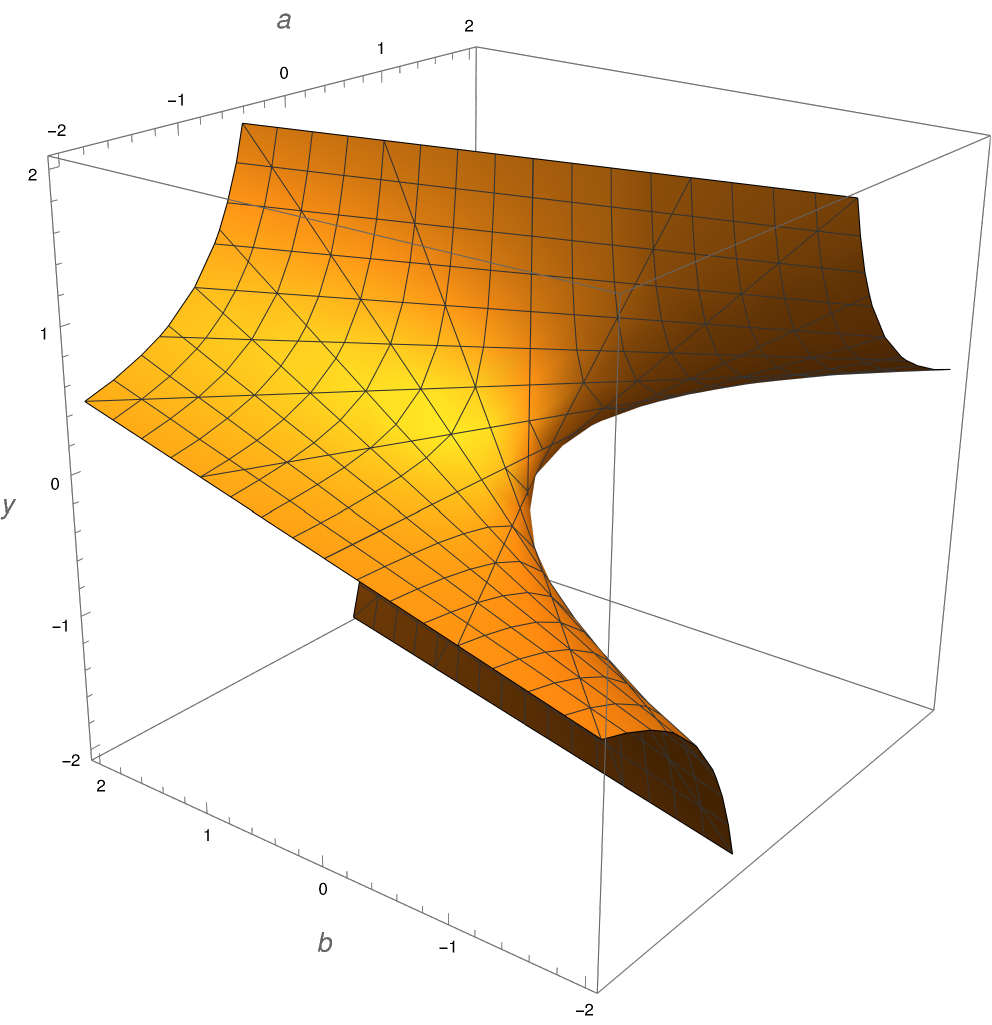}}
\end{center}
\caption{Affine parts of the total zero scheme for the action of $\Bs_2$ on $\PP^1$. The left part misses a line (over $b=0$); the right part misses the blue component.}
\label{totalp1}
\end{figure}

\begin{example}
There is one example that we are able to draw. It is the action of $\Bs_2$ on $\PP^1$ (see Example~\ref{exsl22}). The total zero scheme is not affine, but we can cover it with two affine pieces, coming from affine covers of~$\PP^1$. The first  will be the part contained in $\bb_2 \times \{[1:x]: x\in \C\}$, and the other will be the part contained in $\bb_2 \times \{[y:1] : y\in \C\}$. If we consider coordinates $(a,b)$ on $\bb_2$ that correspond to matrices
$$\begin{pmatrix}
a & b\\
0 & -a
\end{pmatrix} \in \bb_2,
$$
then the surface has the equations $2ax + bx^2 = 0$ in the $(a,b,x)$-plane (the first piece) and $2ay+b = 0$ in the $(a,b,y)$-plane (the second piece). The scheme has two irreducible components; the pieces are drawn in Figure~\ref{totalp1}.

One sees that the $\Bs_2$-invariant functions on the blue part depend only on $a$; hence they form $\C[a]$. Analogously, on the orange part the $\Bs_2$-invariant functions only depend on $a = bx$, as for $a\neq 0$ any two points with the same $a$ are conjugate, and for $a = 0$ we get $b = 0$ or $x = 0$. The former line is a projective line on which an invariant function must attain the same value, and the latter lies in the blue part. This leaves us with two functions from $\C[a]$ with the same constant term. One easily sees that this ring is isomorphic to \textit{e.g.} $\C[a,x]/x(x+2a)$.
\end{example}

\subsection{Equivariant cohomology of GKM spaces via the total zero scheme}
\label{secgkm}

We suspect that the description of equivariant cohomology as the ring of regular functions on the total zero scheme might still hold in a larger generality. For example, one could presume that a sufficiently regular torus action might lead to such a description, even without embedding the torus in a larger solvable group (as in Section~\ref{solvsec}). Here we prove this equality for GKM spaces (see \cite{GKM}), whose equivariant cohomology we know.

\begin{theorem}\label{gkmthm}
Let a torus $\Ts \cong (\Cs)^r$ act on a smooth projective complex variety $X$ with finitely many zero and $1$-dimensional orbits. In other words, the $\Ts$-action makes $X$ a GKM space. Let $\ZZ=\ZZ_{\tot} \subset \ttt\times X$ be the reduced total zero scheme of this action $($cf.\ Definition~\ref{totzer}\,$)$. Then $\C[\ZZ] \cong H^*_\Ts(X)$ as algebras over $\C[\ttt]\simeq H^*_\Ts$:
$$ \begin{tikzcd}
  \C[\ZZ] \arrow{r}{\cong}  &
  H_\Ts^*(X;\C)   \\
   \C[\ttt] \arrow{r}{\cong} \arrow{u}&
  H^*_\Ts \arrow{u}.
\end{tikzcd}$$
\end{theorem}

Let us denote the $\Ts$-fixed points by $\zeta_1$, $\zeta_2$, \dots, $\zeta_s$ and the $1$-dimensional orbits by $E_1$, $E_2$, \dots, $E_{\ell}$. The closure of any $E_i$ is an embedding of $\PP^1$ and contains two fixed points $\zeta_{i_0}$ and $\zeta_{i_\infty}$, which for any $x\in E_i$ are equal to the limits $\lim_{t\to 0} t x$ and $\lim_{t\to\infty} tx$. The action of $\Ts$ on $E_i$ has kernel of codimension $1$, which is uniquely determined by its Lie algebra $\kek_i$.
We then have the following result (\textit{cf.} \cite[Theorem 1.2.2]{GKM}). 

\begin{theorem}[Goresky--Kottwitz--MacPherson, 1998]\label{gkmres}
Assume that a torus $\Ts$ acts on a smooth GKM space $X$. Then the restriction $H_\Ts^*(X,\C)\to H_\Ts^*(X^T,\C) \cong \C[\ttt]^s$ is injective, and its image is
$$
H = \left\{
(f_1,f_2,\dots,f_s)\in \C[\ttt]^s :  f_{i_0}|_{\kek_i} = f_{i_\infty}|_{\kek_i} \text{ for } j=1,2,\dots,\ell
\right\}.
$$
\end{theorem}

We will proceed by finding an injective map $\rho\colon H^*_\Ts(X)\to \C[\ZZ]$ and an injective left inverse $r\colon \C[\ZZ]\to H \cong H^*_\Ts(X)$. We will use Lemma~\ref{projinj} with $Y = \ZZ$ as defined above and $Z = \ttt \times X^\Ts$. Take any $(v,p)\in \ZZ$. The point $p$ lies in the zero scheme $\ZZ_v$ of the vector field on $X$ corresponding to $v$. As $\Ts$ is a commutative group and hence acts trivially on its Lie algebra, it preserves zeros of $v\in\ttt$. Therefore, $\{v\}\times \Ts\cdot p\subset \ZZ$ and $\{v\}\times \overline{\Ts\cdot p}$ is a closed projective subvariety of $\ZZ$. As $\Ts$ acts on it,  by the Borel fixed point theorem,  it contains a fixed point of $\Ts$, hence intersects $Z$ non-trivially. Therefore, this choice of $Y$ and $Z$ satisfies the conditions of Lemma~\ref{projinj}.

We know that there are finitely many distinct types of orbits of the $\Ts$-action on $X$.  This can be seen by embedding $X$ equivariantly in a projective space with a linear action of $\Ts$; see \cite[Theorem~7.3]{Dolgachev}. Therefore, there exists a one-parameter subgroup $\{H^t\}_{t\in\Cs}\subset \Ts$ that is not contained in any proper centraliser. Then the fixed points of $H^t$ are automatically the fixed points of $\Ts$. Consider the Bia{\l}ynicki-Birula minus-decomposition, consisting of cells
$$W_i^{-} = \{x\in X: \lim_{t\to\infty} H^t \cdot x = \zeta_i\}$$
for $\zeta_1$, $\zeta_2$, \dots, $\zeta_s$  the fixed points of $\Ts$.

We first define the map $\rho\colon H^*_\Ts(X)\to \C[\ZZ]$. We will define it on closed points, using the reducedness of $\ZZ$. Let $c\in H^*_\Ts(X)$. Assume that $(v,x)\in \ZZ$, \textit{i.e.} the vector field $v$ is zero at $x$. We know that $x \in W_i^{-}$ for some $i\in\{1,2,\dots,s\}$. The restriction $c|_{\zeta_i}$ is an element of $H_\Ts^*(\pt) \cong \C[\ttt]$. Then define 
$$\rho(c)(v,x) = c|_{\zeta_i} (v).$$
We first have to prove that this defines a regular function for each $c$.

\begin{lemma}
Let $\Ee$ be a $\Ts$-equivariant bundle on $X$. Then
$$\rho(c_k(\Ee))(v,x) = \Tr_{\Lambda^k(\Ee_x)}(\Lambda^k(v)).$$
In particular, $\rho(c_k(\Ee))$ is a regular function on $\ZZ$.
\end{lemma}

\begin{proof}
Let $c = c_k(\Ee)$.
Consider the curve $C = \overline{H^t \cdot x}$. In particular, let $\zeta_i = \lim_{t\to\infty} H^t \cdot x\in C$. We then define $\rho(c)(v,x) = c|_{\zeta_i}(v)$. But we know that this is equal to
$$c_k(\Ee)|_{\zeta_i}(v) = \Tr_{\Lambda^k(\Ee_{\zeta_i})}(\Lambda^k v).$$
However, as $\Ts$ is commutative, the action of any of its elements, in particular of $H^t$, on $X$ is $\Ts$-equivariant; therefore,
$$\Tr_{\Lambda^k(\Ee_{x})}(\Lambda^k v) = \Tr_{\Lambda^k(\Ee_{H^t\cdot x})}(\Lambda^k v)$$
for any $t\in \Cs$. Therefore, the equality also stays true in the limit; hence
\begin{equation*}\pushQED{\qed}
\Tr_{\Lambda^k(\Ee_{x})}(\Lambda^k v) = \Tr_{\Lambda^k(\Ee_{\zeta_i})}(\Lambda^k v) = \rho(c)(v,x).
\qedhere \popQED
	\end{equation*}
\renewcommand{\qed}{}  
\end{proof}

\begin{proof}[Proof of Theorem~\ref{gkmthm}]
We have defined the map $\rho$; we just have to prove that it is an isomorphism. For the injectivity, note that $\ttt\times X^\Ts$ is contained in $\ZZ$. By definition, if $\rho(c)$ is zero on this subspace, then all localisations to $\Ts$-fixed points vanish. But by Theorem~\ref{gkmres} the localisation is injective; hence $c=0$.

The set $\ttt\times X^{\Ts}\subset \ZZ$ is closed, and considering it as a reduced subvariety, by Lemma~\ref{projinj} we get that the restriction map
$$r\colon \C[\ZZ] \lra \C[\ttt\times X^{\Ts}] \cong \C[\ttt]^s$$
is injective. We need to prove that the image lies in $H$ and that $r\circ \rho$ is the localisation map $H^*_{\Ts}(X)\to H$. The latter comes directly from the definition as $\rho(c)(v,\zeta_i) = c|_{\zeta_i}(v)$.

Now we need to prove that for any $E_i$ and $v\in\kek_i$ and $f\in \C[\ZZ]$, we have $f(v,\zeta_{i_0}) = f(v,\zeta_{i_\infty})$. Note that as the infinitesimal action of $\kek_i$ is trivial on $E_i$, we have $\kek_i \times \overline{E_i} \subset \ZZ$. This means that over each $v\in\kek_i$ there is a closed subset $\{v\}\times \overline{E_i}\subset \ZZ$. As the reduced subscheme structure makes this is a projective variety ($\PP^1$, precisely), every global function on $\ZZ$ needs to be constant along this subvariety. As it contains $(v,\zeta_{i_0})$ and $(v,\zeta_{i_\infty})$, we get $f(v,\zeta_{i_0}) = f(v,\zeta_{i_\infty})$.
\end{proof}

\begin{remark}
Thus the ring of regular functions on the total scheme is isomorphic to the equivariant cohomology for regular actions of principally paired group on smooth projective varieties, by Theorem~\ref{wholeg}, as well as for GKM spaces by Theorem~\ref{gkmthm}. We expect this to hold for a larger class of group actions on smooth projective varieties, including spherical varieties.

In the above we used the fact that the torus-fixed points are isolated, but we also needed the GKM cohomology result, \textit{i.e.}~Theorem~\ref{gkmres}. This way we know that any function on the zero scheme will be a cohomology class, as it will determine an element that already lies in $H$. Note that for arbitrary torus actions, every 1-orbit defines a similar condition on the image of localisation, but the image of localisation will in general be  strictly smaller than similarly defined $H$.

We can see that it is not enough to assume for the torus to act with isolated fixed points. Indeed, let us consider $X = \PP^2$, but we restrict the standard action of the $2$-dimensional torus to $1$-dimensional $\Cs$. Take \textit{e.g.} the action $t\cdot [x:y:z] = [x:t^2y:t^4z]$. The only fixed points are $[1:0:0]$ and $[0:1:0]$ and $[0:0:1]$, and hence if we consider any non-zero $v\in \C \cong \Lie(\Cs)$, the associated vector field only has those three zeros. On the other hand, for $v = 0$ the zero scheme is the whole $\PP^2$. Therefore, $\ZZ_\tot \subset \C\times X$ will consist of a vertical $\PP^2$ and three horizontal lines. The action of $t\in \Cs$ multiplies by $t^{-2}$ on each of those lines.

Any global function on $\ZZ_\tot$ determines polynomials $f_1$, $f_2$, $f_3$ on those lines. Then $\C[\ZZ_\tot] = \{(f_1,f_2,f_3)\in \C[x]^3 : f_1(0) = f_2(0) = f_3(0)\}$. There is an injective map $H^*_{\Cs}(\PP^2) \to \C[\ZZ_\tot]$, but it is not surjective. From Example~\ref{exsl22} we have $H_{\Cs}^*(\PP^2) = \C[x,v]/\big(x(x+2v)(x+4v))$. Geometrically, we see the map $\ZZ_\tot\to \Spec H_{\Cs}^*(\PP^2)$ which contracts $\PP^2$ to the point. We see that $h_{\Cs}^2(\PP^2) = 2$, but $\C[\ZZ_\tot]^2 = \{(ax,bx,cx) : a,b,c\in \C\}$ is $3$-dimensional.

Work is ongoing to determine under what assumptions this result holds. For example for many affine Bott--Samelson varieties, Jakub L\"owit proved $H^*_G(X,\C)\cong \C[\ZZ_\tot]$ (private communication), as well as a version for equivariant $K$-theory. Other examples and applications can be found in \cite[Section~4]{hausel-big}.
\end{remark}

\renewcommand\thesection{\Alph{section}}
\setcounter{section}{0}
\section*{Appendix. Graded Nakayama lemma}
\addcontentsline{toc}{section}{Appendix. Graded Nakayama lemma}
\refstepcounter{section}
\setcounter{subsection}{1}

For the sake of completeness, we provide here the proof of the version of the graded Nakayama lemma that we need (see also \cite[Corollary 4.8 and Exercise 4.6]{Eis}).

\medskip
\noindent Let $R$ be an $\Z_{\ge 0}$-graded ring $R = \bigoplus_{n\ge 0} R_n$ and $I = \bigoplus_{n > 0} R_n$ the ideal generated by elements of positive degree.

\begin{lemma}
 If a $\Z_{\ge 0}$-graded $R$-module $M$ satisfies $M = IM$, then $M = 0$.
\end{lemma}

\begin{proof}
 Suppose to the contrary that $a\in M$ is a non-zero homogeneous element of minimal degree $d\in\Z_{\ge 0}$ present in $M$. By the assumption $M = IM$, we have that
 $$a = \sum_{i=1}^k r_i a_i$$
 for some $r_i\in I$, $a_i\in M$. But as $r_i\in I$, the minimal degree present in $r_i$ is at least $1$. As $a_i\in M$, the minimal degree present in $a_i$ is at least $d$. Therefore, the elements $r_ia_i$ have zero parts in degrees less than $d+1$. In particular, we cannot get $a$ as a sum of them, as it has non-zero part in degree $d$.
\end{proof}

\begin{corollary}\label{cornak}
Let $M$ be an $\Z_{\ge 0}$-graded $R$-module $M$. Suppose that elements $(a_j)_{j\in J}$ of $M$ generate the $R/I$-module $M/IM$. Then they generate $M$ as an $R$-module.
\end{corollary}

\begin{proof}
Let us consider the map of $R$-modules $\phi\colon R^J\to M$ defined by the elements $a_j$. We have the exact sequence
$$R^J\xrightarrow{\,\;\phi\;\,} M \lra \coker \phi \lra 0.$$
As tensor product is right-exact, by tensoring with $R/I$ we get an exact sequence of $R/I$-modules:
$$(R/I)^J\lra M/IM \lra (\coker \phi)\otimes_R R/I \lra 0.$$
By assumption the first map is an epimorphism; hence $(\coker \phi)\otimes_R R/I =0$. In other words, $\coker \phi$ satisfies the conditions of lemma. Therefore, $\coker\phi = 0$, so $\phi$ is surjective.
\end{proof}


\end{document}